\documentclass[11pt,leqno]{amsart}

\usepackage{amsthm,amsfonts,amssymb,epsfig,graphics,amsmath,oldgerm}
\usepackage[latin1]{inputenc}\relax
\usepackage{color}

\usepackage{nomencl}
 \RequirePackage{ifthen}
 \renewcommand{\nomgroup}[1]{%
 \ifthenelse{\equal{#1}{H}}{\item[\textbf{Quantities appearing in the equations}]}{%
  \ifthenelse{\equal{#1}{L}}{\item[\textbf{Functional spaces, norms and constants}]}{%
 \ifthenelse{\equal{#1}{G}}{\item[\textbf{General notations}]}{}} } }

\newcommand{\dsp}{\displaystyle}
\newcommand{\RR}{{\mathbb R}}
\newcommand{\NN}{{\mathbb N}}

\newcommand{\HH}{{\mathbb H}}

\newcommand{\VV}{{\mathbb V}}
\newcommand{\WW}{{\mathbb W}}

\newcommand{\cA}{\mathcal{A}}

\newcommand{\cC}{\mathcal{C}}

\newcommand{\cF}{\mathcal{F}}
\newcommand{\cG}{\mathcal{G}}
\newcommand{\cHH}{H}
\newcommand{\cH}{\mathcal{H}}

\newcommand{\cL}{\mathcal{L}}
\newcommand{\cN}{\mathcal{N}}

\newcommand{\cQ}{\mathcal{Q}}
\newcommand{\ucQ}{\underline{\mathcal{Q}}}
\newcommand{\cR}{\mathcal{R}}
\newcommand{\cS}{\mathcal{S}}
\newcommand{\cU}{\mathcal{U}}
\newcommand{\ucU}{\underline{\mathcal{U}}}
\newcommand{\cT}{\mathcal{T}}
\newcommand{\uu}{\underline{u}}
\newcommand{\uq}{\underline{q}}
\newcommand{\uV}{\underline{V}}
\newcommand{\Vp}{\Vert V\Vert_{L^{\infty,p}}}
\newcommand{\Vpt}{\Vert V(t)\Vert_{L^{\infty,p}}}
\newcommand{\uCp}{\underline{C}_1}

\font\tenronde=rsfs10
\font\sevenronde=rsfs7
\font\fiveronde=rsfs5
\newfam\rondefam
\scriptscriptfont\rondefam=\fiveronde
\textfont\rondefam=\tenronde
\scriptfont\rondefam=\sevenronde
\def\ronde{\fam\rondefam\tenronde}

\newcommand{\sS}{{\ronde S}}

\newcommand{\ucM}{\underline{\mathcal M}}
\newcommand{\eps}{\varepsilon }
\newcommand{\vp}{\varphi}
\newcommand{\D}{\partial }

\newcommand{\re}{\mathrm{Re\,}}

\newcommand{\mez}{\frac{1}{2}}

\newcommand{\CMTT}{C\big(T,M_1, M_2,M\big)}

\newtheorem{theo}{Theorem }[section]
\newtheorem{prop}[theo]{Proposition}
\newtheorem{cor}[theo]{Corollary}
\newtheorem{lem}[theo]{Lemme}
\newtheorem{defi}[theo]{Definition}
\newtheorem{ass}[theo]{Assumption}

\newtheorem{rem}[theo]{Remark}

\newtheorem*{theo*}{Theorem}
\numberwithin{equation}{section}

\title[The shoreline problem]{The shoreline problem for the one-dimensional shallow water and Green-Naghdi equations }

\author{D. Lannes and G. M\'etivier}
\address{Institut de Math\'ematiques de Bordeaux\\ Universit\'e de Bordeaux et CNRS UMR 5251\\ 351 Cours de la Lib\'eration \\ 33405 Talence Cedex, France}
\email{David.Lannes@math.u-bordeaux.fr, Guy.Metivier@math.u-bordeaux.fr}

\begin{document} 
\maketitle

\begin{abstract}
The Green-Naghdi equations are a nonlinear dispersive perturbation of the nonlinear shallow water equations,
more precise by one order of approximation. These equations are commonly used for the simulation of coastal flows, and in particular in regions where the water depth vanishes (the shoreline). The local well-posedness of the Green-Naghdi equations (and their justification as an asymptotic model for the water waves equations) has been extensively studied, but always under the assumption that the water depth is bounded from below by a positive constant. The aim of this article is to remove this assumption. The problem then becomes a free-boundary problem since the position of the shoreline is unknown and driven by the solution itself. For the (hyperbolic) nonlinear shallow water equation, this problem is very related to the vacuum problem for a compressible gas. The Green-Naghdi equation include additional nonlinear, dispersive and topography terms with a complex degenerate structure at the boundary. In particular, the degeneracy of the topography terms makes the problem loose its quasilinear structure and become fully nonlinear. Dispersive smoothing also degenerates and its behavior at the boundary can be described by an ODE with regular singularity.
These issues require the development of new tools, some of which of independent interest such as the study of the mixed initial boundary value problem for dispersive perturbations of characteristic hyperbolic systems, elliptic regularization with respect to conormal derivatives, or general Hardy-type inequalities.
\end{abstract} 


\section{Introduction}
 
\subsection{Presentation of the problem}

A commonly used model to describe the evolution of waves in shallow water is the nonlinear shallow water model, which is a system of equations coupling the water height ${\mathtt H}$ to the vertically averaged horizontal velocity ${\mathtt U}$. When the horizontal dimension  is equal to $1$ and denoting by ${\mathtt X}$ the horizontal variable and by $-H_0+{\mathtt B}({\mathtt X})$ a parametrization of the bottom, these equations read
$$
\begin{cases}
\D_{\mathtt t} {\mathtt H} + \D_{\mathtt X} ({\mathtt H} {\mathtt U }  ) = 0  , \\
 \D_{\mathtt t} {\mathtt U } +  {\mathtt U } \D_{\mathtt X} {\mathtt U }    + g \D_{\mathtt X} {\mathtt H} = -g\D_{\mathtt X}{\mathtt B} ,
\end{cases}
$$
where $g$ is the acceleration of gravity. These equations are known to be valid (see \cite{Alvarez-Samaniego:2008ai,iguchi2009} for a rigorous justification) in the shallow water regime corresponding to the condition $0<\mu \ll 1$, where the {\it shallowness parameter} $\mu$ is defined as
$$
\mu =\Big(\frac{\mbox{typical depth}}{\mbox{horizontal scale}}\Big)^2=\frac{H_0^2}{L^2},
$$
 where $L$ corresponds to the order of the wavelength of the waves under consideration.
Introducing the dimensionless quantities
$$
H:=\frac{{\mathtt H}}{H_0}, \quad U=\frac{{\mathtt U }}{\sqrt{gH_0}},\quad B=\frac{{\mathtt B}}{H_0},\quad t=\frac{{\mathtt t}}{L/\sqrt{gH_0}},\quad X=\frac{{\mathtt X}}{L}
$$
these equations can be written
\begin{equation}
 \label{SW} 
\begin{cases}
\D_t H + \D_X (H {U }  ) = 0  , \\
 \D_t {U } +  U \D_X {U}    +  \D_X H = -\D_X B .
\end{cases}
 \end{equation}
The precision of the nonlinear shallow water model \eqref{SW} is $O(\mu)$, meaning that $O(\mu)$ terms have been neglected in these equations  (see for instance \cite{Lannes_book}). A more precise model is furnished by the Green-Naghdi (or Serre, or fully nonlinear Boussinesq) equations. They include the $O(\mu)$ terms and neglect only terms of size $O(\mu^2)$; in their one-dimensional dimensionless form, they can be written\footnote{The dimensionless Green-Naghdi equations traditionally involve two other dimensionless parameters $\varepsilon$ and $\beta$ defined as
$$
\varepsilon=\frac{\mbox{amplitude of surface variations}}{\mbox{typical depth}},\qquad 
\beta=\frac{\mbox{amplitude of bottom variations}}{\mbox{typical depth}}.
$$
Making additional smallness assumptions on these parameters, one can derive simpler systems of equations (such as the Boussinesq systems), but since we are interested here in configurations where the surface and bottom variations can be of the same order as the depth, we set $\varepsilon=\beta=1$ for the sake of simplicity.} (see for instance \cite{Lannes_book})
 \begin{equation}
 \label{GN} 
\begin{cases}
\D_t H + \D_X (H {U }  ) = 0  , \\
{\bf D} (  \D_t {U } +  U \D_X {U} )   +  \D_X H +\mu {\bf Q}_1= -\D_X B ,
\end{cases}
 \end{equation}
 where $H=1+\zeta -B\geq 0$ is the (dimensionless) water depth and $U$ the (dimensionless) horizontal mean velocity. The dispersive operator ${\bf D}$ is given by
 $$
 {\bf D} U =    U  - \frac{\mu}{3H}  \D_X (H^3 \D_X U )+\frac{1}{2H}\big[\D_X(H^2 \D_X B U)-H^2\D_X B \D_X U\big]+(\D_X B)^2 U, 
$$
and the nonlinear term ${\bf Q}_1$ takes the form
$$
 {\bf Q}_1=\frac{2}{3H}\D_X \big(H^3(\D_X U)^2\big)+H(\D_X U)^2 \D_X B+\frac{1}{2H}\D_X(H^2U^2 \D_X^2 B)+U^2 \D_X^2 B\D_X B;
$$
of course, dropping $O(\mu)$ terms in \eqref{GN}, one recovers the nonlinear shallow water equations \eqref{SW}.

Under the assumption that the water-depth never vanishes, the local well-posedness of \eqref{GN} has been assessed in several references \cite{Li,AL2,Israwi201181,Iguchi_new}. However, for practical applications (for the numerical modeling of submersion issues for instance), the Green-Naghdi equations are used up to the shoreline, that is, in configurations where the water depth vanishes, see for instance \cite{Bonneton:2011jo,FILIPPINI2016381}.  Our goal here is to study mathematically such a configuration, i.e. to show that the Green-Naghdi equations \eqref{GN} are well-posed in the presence of a moving shoreline.

\begin{figure}
\begin{center}
\includegraphics[height=7.5cm]{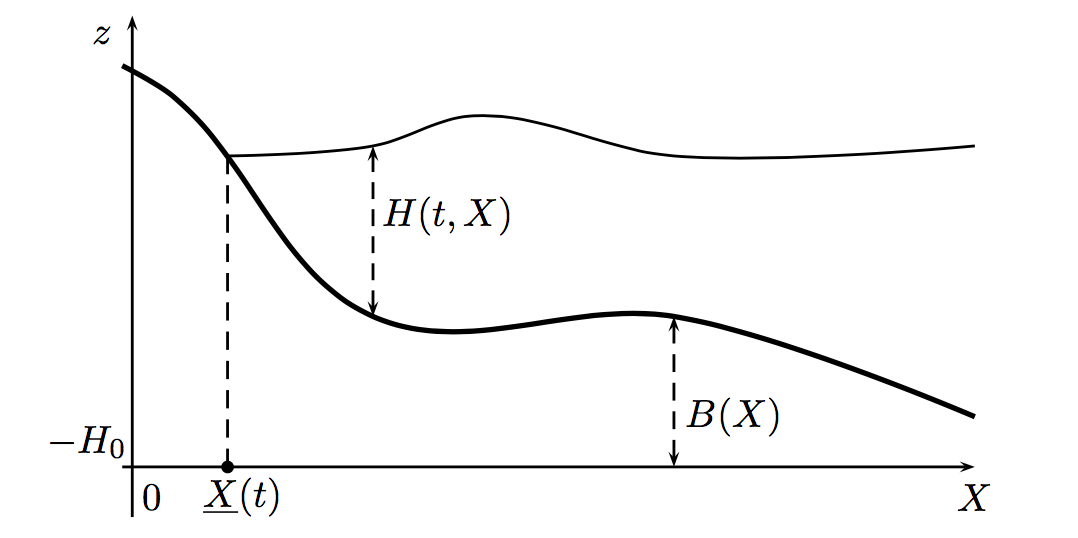} 
\caption{The shoreline}
\label{fig_shoreline}
\end{center}
\end{figure}

This problem is a free-boundary problem, in which one must find the horizontal coordinate $\underline{X}(t)$ of the shoreline (see Figure \ref{fig_shoreline}) and show that the Green-Naghdi equations \eqref{GN} are well-posed on the half-line $(\underline{X}(t),+\infty)$ with the boundary condition
\begin{equation}\label{Hnul}
 H\big(t,\underline{X}(t)\big) = 0.
\end{equation}
Time-differentiating this identity and using the first equation of \eqref{GN}, one obtains that $\underline{X}(\cdot)$ must solve the kinematic boundary condition 
\begin{equation}\label{kinematic}
\underline{X}'(t)=U\big(t,\underline{X}(t)\big).
\end{equation}

When  $\mu=0$, the Green-Naghdi equations reduce to the shallow water equations \eqref{SW} which, when the bottom is flat ($B=0$), coincide with the compressible isentropic Euler equations ($H$ representing in that case the density, and the pressure law being given by $P=g H^2$). The shoreline problem for the nonlinear shallow water equations with a flat bottom coincide therefore with the vacuum problem for a compressible gas with physical vacuum singularity in the sense of \cite{Liu1996}. This problem has been solved in \cite{JM1,CS1} ($d=1$) and \cite{CS2,JM2} ($d=2$). Mathematically speaking, this problem is a nonlinear hyperbolic system with a characteristic free-boundary condition. Less related from the mathematical viewpoint, but closely related with respect to the physical framework are \cite{Poyferre} and \cite{MingWang}, where a priori estimates are derived for the shoreline problem for the water waves equations (respectively without and with surface tension).

Though the problem under consideration here is related to the vacuum problem for a compressible gas, it is different in nature because the equations are no longer hyperbolic due the presence of the nonlinear dispersive terms ${\bf D}$ and ${\bf Q}_1$. Because of this, several important steps of the proof, such as the resolution of linear mixed initial boundary value problems, do not fall in existing theories and require the development of new tools.
The major new difficulty is that everything degenerates at the boundary $ H = 0$: strict hyperbolicity (when $\mu = 0$) is lost, the dispersion vanishes, the energy degenerates ;  the topography increases  the complexity  since it makes the problem fully nonlinear as we will explain later on.  An important feature of the problem is the structure of the degeneracy at the boundary. As in the vacuum problem for Euler, it allows to use Hardy's inequalities to ultimately get the  $L^\infty$ estimates which are necessary to deal with nonlinearities. The precise structure of the dispersion is crucial  and used at many places in the computations. Even  if they are not made explicit, except at time $t = 0$, the properties of ${\bf D}^{-1}$  are important.  The dispersion ${\bf D}$ appears as a degenerate elliptic operator
(see e.g. \cite{BolleyCamus} for a general theory). A similar problem was met in \cite{BreschMetivier} in the study of the lake equation with vanishing topography at the shore: the pressure was given by a degenerate elliptic equation. 
To sum up in one sentence, all this paper turns around the influence of the degeneracy at the boundary. 

\medbreak

Our main result is to prove the local in time (uniformly in $\mu$) well-posedness of the shoreline problem for the one-dimensional Green-Naghdi equations.  The precise statement if given in Theorem~\ref{theomain} below. Stability conditions are required. They are introduced in  \eqref{inita} and \eqref{initan} and discussed  there.   The spirit of the main theorem is given in the following qualitative statement.   Note that the case $\mu=0$ corresponds to the shoreline problem for the nonlinear shallow water equations \eqref{SW}. 

\begin{theo*}
For smooth enough initial conditions, and under certain conditions on the behavior of the initial data at the shoreline, there exists a non trivial time interval independent of $\mu \in [0,1]$ on which there exists a unique triplet $(\underline{X},H,U)$ such that $(H,U)$ solves \eqref{GN} with $H>0$ on $(\underline{X}(t),+\infty)$, and $H(\underline{X}(t))=0$.
 \end{theo*}

\subsection{Outline of the paper}

In Section \ref{sectreform0}, we transform the equations \eqref{GN} with free boundary condition \eqref{Hnul} into a formulation which is more appropriate for the mathematical analysis, and where the free-boundary has been fixed. This is done using a Lagrangian mapping, together with an additional change of variables.

The equations derived in Section \ref{sectreform0} turn out to be fully nonlinear because of the topography terms. Therefore, we propose in Section \ref{sectlineq} to quasilinearize them  by writing the extended system formed by the original equations and by the equations satisfied by the time and conormal derivatives of the solution. 
The linearized equations thus obtained are studied in \S \ref{sectlinest} where it is shown that the energy estimate involve degenerate weighted $L^2$ spaces. The extended quasilinear system formed by the solution and its derivatives is written in \S \ref{sectquasi0}; this is the system for which a solution will be constructed in the following sections. 

Section \ref{sectstrategy} is devoted to the statement (in \S \ref{sectstate})  and sketch of the proof of the main result. The strategy consists in constructing a solution to the quasilinear system derived in \S \ref{sectquasi0} using an iterative scheme. For this, we need a higher order version of the linear estimates of \S \ref{sectlinest}. These estimates, given in \S \ref{sectHO},  involve Sobolev spaces with degenerate weights for which standard Sobolev embeddings fail. To recover a control on non-weighted $L^2$ norms and $L^\infty$ norms, we therefore need to use the structure of the equations and various  Hardy-type inequalities (of independent interest and therefore derived in a specific section). Unfortunately, when applied to the iterative scheme, these energy estimates yield a loss of one derivative; to overcome this difficulty,  we introduce an additional elliptic equation (which of course disappears at the limit) regaining one time and one conormal derivative;  this is done in \S \ref{sectiter}. \\
The energy estimates for the full augmented system involve the initial value of high order time derivatives; for the nonlinear shallow water equations ($\mu=0$), the time derivatives can easily be expressed in terms of space derivatives but the presence of the dispersive terms make things much more complicated when $\mu>0$; the required results are stated in \S \ref{sectIVTD} but their proof is postponed to Section \ref{Sect9}.\\
We then explain in \S \ref{sectexIBVP} how to solve the mixed initial boundary value problems involved at each step of the iterative scheme. There are essentially two steps for which there is no existing theory: the analysis of elliptic (with respect to time and conormal derivative) equations on the half line, and the theory of mixed initial boundary value problem for dispersive perturbations of hyperbolic systems. These two problems  being of independent interest, their analysis is postponed to specific sections. We finally sketch (in \S \ref{sectbounds} and \S \ref{sectconv}) the proof that the iterative scheme provides a bounded sequence that converges to the solution of the equations.

Section \ref{Sec6} is devoted to the proof of the Hardy-type inequalities that have been used to derive the higher-order energy estimates of \S \ref{sectHO}. We actually prove more general results for a general family of operators that contain the two operators $h_0\D_x$ and $h_0\D_x +2h_0'$ that we shall need here. These estimates, of independent interest, provide Hardy-type inequalities for $L^p$-spaces with various degenerate weights.

In Section \ref{sectcommut}, several technical results used in the proof of Theorem \ref{theomain} are provided. More precisely, the higher order estimates of Proposition \ref{propHO} are proved with full details in \S \ref{sectproofHO} and the bounds on the sequence constructed through the iterative scheme of \S \ref{sectiter} are rigorously established in \S \ref{sectproofbounds}.

The elliptic equation that has been introduced in \S \ref{sectiter} to regain one time and one conormal derivative in the estimates for the iterative scheme is studied in Section \ref{Sell}. Since there is no general theory for such equations, the proof is provided with full details. We first study a general family of elliptic equations (with respect to time and standard space derivatives) on the full line, for which classical elliptic estimates are derived. In \S \ref{sectHline}, the equations and the estimates are then transported to the half-line using a diffeomorphism that transforms standard space derivatives on the full line into conormal derivatives on the half line. Note that the degenerate weighted estimates needed on the half line require elliptic estimates with exponential weight in the full line.

In Section \ref{Sex} we develop a theory to handle mixed initial boundary value problems for dispersive perturbations of characteristic linear hyperbolic systems. To our knowledge, no result of this kind can be found in the literature. The first step is to assess the lowest regularity at which the linear energy estimates of \S \ref{sectlinest} can be performed. This requires duality formulas in degenerate weighted spaces that are derived in \S \ref{sectduality}. As shown in \S \ref{sectestLR}, the energy space is not regular enough to derive the energy estimates; therefore, the weak solutions in the energy space constructed in \S \ref{sectweak} are not necessarily unique. We show however in \S \ref{sectstrong} that weak solutions are actually strong solutions, that is, limit in the energy space of solutions that have the required regularity for energy estimates. It follows that weak solutions satisfy the energy estimate and are therefore unique. This weak=strong result is obtained by a convolution in time of the equations. Provided that the coefficients of the linearized equations are regular, we then show in \S \ref{sectsmooth} that if these strong solutions are smooth if the source term is regular enough. The last step, performed in \S \ref{sectlast}, is to remove the smoothness assumption on the coefficients.

Finally, Section \ref{Sect9} is devoted to the invertibility of the the dispersive operator at $t=0$ in various weighted space. These considerations are crucial to control the norm of the time derivative of the solution at $t=0$ in terms of space derivative, as raised in \S \ref{sectIVTD}. We reduce the problem to the analysis of an ODE with regular singularity that is analyzed in full details.
\bigbreak
\noindent
{\bf N.B.} A glossary gathers the main notations at the end of this article.
\bigbreak
\noindent
{\bf Acknowledgement.} The authors want to express their warmest thanks to Didier Bresch (U. Savoie Mont Blanc and ASM Clermont Auvergne) for many discussions about this work.

\section{Reformulation of the problem} \label{sectreform0}

This section is devoted to a reformulation of the shoreline problem for the  Green-Naghdi equations \eqref{GN}. The first step is to fix the free-boundary. This is done in \S \ref{sectlagr} and \S \ref{GNlagr} using a Lagrangian mapping. We then propose in \S \ref{sectqu} a change of variables that transform the equations into a formulation where the coefficients of the space derivatives in the leading order terms are time independent.
\subsection{The Lagrangian mapping}\label{sectlagr}
As usual with free boundary problems, we first use a diffeomorphism mapping the moving domain $(\underline{X}(t),+\infty)$ into a fixed domain $(\underline{X}_0,+\infty)$ for some time independent $\underline{X}_0$. A convenient way to do so is to work in Lagrangian coordinates.
More precisely, and with $\underline{X}_0=\underline{X}(t=0)$, we define for all times a diffeomorphism $\varphi(t,\cdot): (\underline{X}_0,+\infty)\to (\underline{X}(t),+\infty)$ by the relations
\begin{equation}\label{Lag1}
\D_t \vp (t,x) = U(t,\varphi(t,x)), \qquad \vp(0,x)=x;
\end{equation}
the fact that $\vp(t,\underline{x}_0)=\underline{x}(t)$ for all times stems from \eqref{kinematic}. {\it Without loss of generality, we can assume that $\underline{X}_0=0$.}\\
We also introduce the notation
\begin{equation}\label{Lag1bis} \eta = \D_x \vp
\end{equation}
and shall use upper and lowercases letters for Eulerian and Lagrangian quantities respectively, namely,
$$
h(t,x)=H(t,\vp(t,x)),\qquad u(t,x)=U(t,\vp(t,x)), \quad \mbox{etc.}
$$

\subsection{The Green-Naghdi equations in Lagrangian coordinates}\label{GNlagr}
Composing the first equation of \eqref{GN} with the Lagrangian mapping \eqref{Lag1}, and with $\eta$ defined in \eqref{Lag1bis}, we obtain
$$
\D_t h+\frac{h}{\eta}\D_x u=0;
$$
when combined with the relation 
$$
\D_t \eta - \D_x u = 0
$$
that stems from \eqref{Lag1bis}, this easily yields
$$
\D_t (\eta h)=0.
$$
We thus recover the classical fact that in Lagrangian variables, the water depth is given in terms of $\eta$ and of the water depth $h_0$ at $t=0$,
\begin{equation}\label{Lag2}
 h = \frac{h_0}{\eta}. 
\end{equation} 
In Lagrangian variables, the Green-Naghdi equations therefore reduce to the above equation on $\eta$ complemented by the equation on $u$ obtained by composing the second equation of \eqref{GN} with $\vp$,
\begin{equation}\label{GNlag}
\left\{ \begin{aligned}
& \D_t \eta - \D_x u = 0 \\
 & {\bf d}    \D_t u  + \frac{1}{\eta}  \D_x  h +\mu {\bf q}_1  =  - B'( \vp),
 \end{aligned} \right. 
\end{equation}
with $h=h_0/\eta$ and $\bf d$ defined as
$$
 {\bf d} u    =    u   - \frac{\mu}{3h\eta}  \D_x  (\frac{h^3}{\eta} \D_x u ) 
 +\frac{\mu}{2h\eta}\big[\D_x(h^2 B'(\varphi) u )-h^2 B'(\varphi) \D_x u \big]+\mu B'(\varphi)^2 u
$$
while the nonlinear term ${\bf q}_1$ is given by
$$
 {\bf q}_1=\frac{2}{3h\eta}\D_x \big(\frac{h^3}{\eta^2}(\D_x u)^2\big)+\frac{h}{\eta^2}(\D_x u)^2 B'(\varphi)+\frac{1}{2h\eta}\D_x(h^2u^2 B''(\varphi))+u^2 B''(\varphi)B'(\varphi).
$$
 
 \subsection{The equations in $(q,u)$ variables}\label{sectqu}
 
 In the second equation of \eqref{GNlag}, the term $1/\eta \D_x h$ in the second equation is nonlinear in $\eta$; it is possible and quite convenient to replace it by a linear term by introducing 
\begin{equation}\label{eq25}
q = \mez  \eta^{-2}  \quad\mbox{ and therefore } \quad  \eta = \eta (q) := (2 q) ^{-\mez}.
\end{equation}  
The resulting model is  
\begin{equation}
\label{mod317}
\left\{ \begin{aligned}
& c\D_t q  + \D_x u = 0 \\
 & {\bf d} \D_t u  +   {\bf l}q  + \mu {\bf q}_1=  - B'( \vp)  
 \end{aligned} \right. 
\end{equation}
where 
\begin{equation}\label{eqc}
c=c(q)=(2q)^{-3/2}>0,\quad  \mbox{ as long as }q>0
\end{equation}
(recall that $\eta_{\vert_{t=0}}=1$ and therefore $q_{\vert_{t=0}}=1/2$).
The operators ${\bf l}$ and ${\bf d}={\bf d}[V]$ are given by
\begin{equation}\label{eql}
{\bf l}=\frac{1}{h_0}\D_x (h_0^2 \cdot )=h_0 \D_x     +   2 h'_0 
\end{equation}
and, denoting $\uV=(\uq,\uu)$ and $\underline{\varphi}=x+\int_0^t \uu$, 
 \begin{align}
 \nonumber
 {\bf d}[\uV] u   =&    u    -   \mu \frac{4}{3h_0} \D_x  ( h_0^3 \uq^2  \D_x u ) 
  +\frac{\mu}{h_0}\big[\D_x(h_0^2  \uq B'(\underline{\varphi}) u )-h_0^2 q B'(\underline{\varphi}) \D_x u \big]\\
 \nonumber &+\mu B'(\underline{\varphi})^2 u\\
  \label{defd}
 =& u + \mu {\bf l} \big[ - \frac{4}{3}\uq^2 h_0 \D_xu+\uq uB'(\underline{\varphi}) \big]-\mu \uq B'(\underline{\varphi})h_0\D_x u+\mu B'(\underline{\varphi})^2u, 
 \end{align} 
and the nonlinear term ${\bf q}_1={\bf q}_1(V)$, with $V=(q,u)$, is
\begin{equation}\label{defq1}
 {\bf q}_1(V)={\bf l}\big[\frac{4}{3}h_0 \frac{q}{c}(\D_x u)^2+q u^2 B''(\varphi)\big]+\frac{h_0}{c} (\D_x u)^2 B'(\varphi)+u^2 B''(\varphi)B'(\varphi).
\end{equation}

 \begin{rem}
 Since by \eqref{Lag1} we have $\varphi(t,x)=x+\int_0^t u(s,x)ds$, we treat the dependence on $\varphi$ in the topography term as a dependence on $u$, hence the notation ${\bf q}_1(V)$ and not ${\bf q}_1(V,\varphi)$ for instance.
 \end{rem}
 
 
 \section{Quasilinearization of the equations} \label{sectlineq}

When the water depth does not vanish, the problem \eqref{mod317} is quasilinear in nature \cite{Israwi201181,Iguchi_new}, but at the shoreline, the energy degenerates and as we shall see, some topography terms make \eqref{mod317} a fully nonlinear problem. In order to quasilinearize it, we want to consider the system of equations formed by \eqref{formFL} together with the evolution equations formally satisfied by  $V_1:=X_1 V$ and $V_2:=X_2 V$, where  $X_1=\D_t$ and $X_2=h_0\D_x$ are chosen because they are tangent to the boundary. 
After giving some notation  in \S \ref{sectreform}, we derive in \S \ref{sectlin} the linear system satisfied by $V_1$ and $V_2$ and provide in \S \ref{sectlinest} $L^2$-based energy estimates for this linear system. We then state in \S \ref{sectquasi0} the quasilinear system satisfied by $(V,V_1,V_2)$ (the fact that it is indeed of quasilinear nature will be proved in Section \ref{sectstrategy}).\\
Throughout this section and the rest of this article, we shall make the following assumption.
\begin{ass}\label{assBh}
{\bf i.} The functions $h_0$ and $B$ are smooth on $\RR^+$. Moreover, $h_0$ satisfies the following properties,
$$
h_0(0)=0,\quad h_0'(0)>0, \quad h_0(x)>0 \mbox{ for all } x>0 \mbox{ and }\liminf_{x\to \infty}h_0(x)>0.
$$
{\bf ii.} We are interested in the shallow water regime corresponding to small values of $\mu$ and therefore assume that $\mu$ does not take large values, say, $\mu\in [0,1]$.
\end{ass}
\begin{rem}
In the context of a compressible gas, this assumption corresponds to a physical vacuum singularity \cite{Liu1996}; the equivalent of flows that are smooth up to vacuum in the sense of \cite{Serre} is not relevant here.
\end{rem}

\bigbreak

\noindent
{\bf N.B. } For the sake of simplicity, the dependance on $h_0$ and $B$ shall always be omitted in all the estimates derived.
\subsection{A compact formulation}\label{sectreform}

For all $\uV=(\uq,\uu)^T$, let us introduce the linear operator $\cL[\uV,\D]$ defined by
\begin{equation}\label{defL}
\cL[\uV,\D]V=\begin{cases}
c(\uq) \D_t q +\D_x u\\
{\bf d}[\uV]\D_t u+{\bf l} q
\end{cases}
\quad\mbox{ for all }V=(q,u)^T,
\end{equation}
so that an equivalent formulation of the equation \eqref{mod317} is given by the following lemma.
\begin{lem}
If $V$ is a smooth solution to \eqref{mod317}, then it also solves
\begin{equation}\label{formFL}
\cL[V,\D]V=\quad \cS(V,X_1V,X_2V)
\end{equation}
with, writing $\varphi(t,x)=x+\int_0^t u(s,x)ds$,
\begin{equation}
\label{defS}
\cS(V,V_1,V_2)=\left(\begin{array}{l}0\\
-B'(\varphi)-\mu {\bf l}\big[-\frac{4}{3} q X_2 u q_1+q u^2 B''(\varphi)\big] \\
\phantom{-B'(\varphi)-}- \mu X_2 u q_1 B'(\varphi)+\mu u^2 B'(\varphi)B''(\varphi)
\end{array}\right).
\end{equation}
\end{lem}
\begin{proof}
One obtains directly that $V$ solves \eqref{formFL} with $\cS$ given by
$$
\cS(V,X_1V,X_2V)=\big(0,-B'(\varphi)-\mu {\bf q_1}(V)\big)^T
$$
and ${\bf q_1}(V)$ as defined in \eqref{defq1}.
In order to put it under the form given in the statement of the lemma, one just needs to use the first equation of \eqref{mod317} to rewrite ${\bf q}_1={\bf q}_1(q,u)$ under the form
\begin{equation}\label{equivq1}
{\bf q}_1={\bf l}\big[-\frac{4}{3} q X_2 u X_1 q+q u X_1 B'\big]- X_2 u X_1 q B'+u B'X_1 B'.
\end{equation}
\end{proof}


\subsection{Linearization}\label{sectlin}
As explained above, we want to quasilinearize \eqref{formFL}, by writing the evolution equations satisfied by $V$ and $X_m V$ ($m=1,2$). We therefore apply the vector fields  $X_1$ and $X_2$ to the two equations of \eqref{formFL}.
For the first equation, we have the following lemma, whose proof is straightforward and omitted.
\begin{lem}\label{lemmaLinear1}
If $V=(q,u)^T$ is a smooth enough solution to \eqref{mod317}, then one has, for $m=1,2$,
$$
c \D_t X_mq +\D_x X_mu=\cF^{(m)}(q,X_1 q,X_2q).
$$
with
$$
\cF^{(m)}(q,q_1,q_2)=-c'(q) q_1 q_m-c(q) \frac{X_m h_0}{h_0} q_1.
$$
\end{lem}
For the second equation, the following lemma holds. The important thing here is that the term $\mu {\bf l} (a(u)X_m q)$ cannot be absorbed in the right-hand-side. As explained in Remark \ref{remgnl} below, this terms makes the problem fully nonlinear.
\begin{lem}\label{lemmaLinear2}
If $V=(q,u)^T$ is a smooth enough solution to \eqref{mod317}, then one has, for $m=1,2$,
$$
{\bf d}[q]\D_t X_mu +{\bf l}\big[\big(1+\mu a(u)\big) X_mq\big]={\bf g}^{(m)} ,
$$
with ${\bf g}^{(m)}= g_0^{(m)}+\sqrt{\mu}{\bf l} g_1^{(m)}$ and
$$
a(u)=X_1(uB')\quad\mbox{ and }\quad
g_j^{(m)}=\cG^{(m)}_j (V,X_1V,X_2V)\qquad (j=1,2),
$$
and where, writing $\varphi_1=u$ and $\varphi_2(t,x):=h_0 +\int_0^t u_2$, one has
\begin{align*}
\frac{1}{\mu}\cG^{(m)}_0 (V,V_1,V_2)=&  (q_m B' +q B'' \varphi_m) X_2 u_1+q_1 B' X_2 u_m+X_2 u X_1 q_m B'\\
&- 2B'B'' \varphi_m u_1 - 2 u u_m B'B''-u^2 (B'B^{(3)}-(B'')^2)\varphi_m \\
&+X_2 u q_1 B''\varphi_m
-2 X_mh_0' \big(q u_1B' - \frac{4}{3}q^2 X_2 u_1\big)-\frac{1}{\mu}B'' \varphi_m,\\
\frac{1}{\sqrt{\mu}}\cG^{(m)}_1 (V,V_1,V_2)=& \frac{8}{3} q q_m X_2 u_1+ \frac{4}{3} q q_1 X_2 u_m+\frac{4}{3}q X_2 u X_1 q_m+\frac{4}{3}q_m X_2 u q_1\\
 &-u_1 q B''\varphi_m-q u_m u B''-qu^2 B^{(3)}\varphi_m.
\end{align*}
\end{lem}
\begin{proof}
From the definition \eqref{defd} of ${\bf d}$, we have
$$
{\bf d} \D_tu = \D_t u + \mu {\bf l} \big[ - \frac{4}{3}q^2 h_0 \D_x\D_tu+q \D_t uB' \big]-\mu q B'h_0\D_x \D_t u+\mu (B')^2 \D_t u, 
$$
so that, applying the vector field $X $ (throughout this proof, we omit the subscript $m=1,2$), we get
\begin{align*}
X {\bf d} \D_t u&= {\bf d}\D_t X  u + \mu {\bf l} \big[ - \frac{8}{3} q X  q X_2X_1 u+X_1 u B'X q +X_1 u q XB'\big]\\
&-\mu X  q B' X_2X_1 u-\mu q X B' X_2X_1 u+\mu X \big((B')^2\big) X_1 u\\
&+ 2\mu X h_0' \big( - \frac{4}{3}q^2 X_2 X_1u+q X_1uB' \big),
\end{align*}
where we used the fact that $[X,{\bf l}]=2Xh_0'$.\\
Before computing $X{\bf q}_1$, we first replace ${\bf q}_1$ by its equivalent expression \eqref{equivq1}. Applying $X $ we find therefore
\begin{align*}
X{\bf q}_1=&{\bf l}\big[-\frac{4}{3} X q X_2 u X_1 q-\frac{4}{3} q X_2 Xu X_1 q-\frac{4}{3} q X_2 u X_1 X q\\
&\phantom{{\bf l}\big[}+Xq u X_1 B'+q X(u X_1 B')\big]
- X_2 X u X_1 q B'- X_2 u X_1 Xq B'\\
&- X_2 u X_1 q XB'+X(u B' X_1B').
\end{align*}
Since moreover $X {\bf l}q={\bf l}Xq +2 Xh_0' q$, one gets
\begin{align*}
-\frac{1}{\mu}g_0^{(m)}=& -X_m(qB')X_2X_1 u-X_1q B' X_2X_m u-X_2 u X_1X_m q B'\\
&+ X_m((B')^2)X_1 u + X_m(uB'X_1 B')-X_2 u X_1 q X_m B'\\
&+2 X_mh_0' \big(- \frac{4}{3}q^2 X_2 X_1u+q X_1uB'\big)+\frac{1}{\mu}X_m B',\\
-\frac{1}{\sqrt{\mu}}g_1^{(m)}=&- \frac{8}{3} q X_m  q X_2X_1 u- \frac{4}{3} q X_1  q X_2X_m u-\frac{4}{3}q X_2 u X_1 X_m q\\
&-\frac{4}{3}X_mq X_2 u X_1q
 +X_1 u q X_mB'+qX_m(uX_1 B'),
\end{align*}
and the result follows easily.
\end{proof}
The previous two lemmas suggest the introduction of the linear operator $\cL_{a}[\uV,\D]$ defined as
\begin{equation}\label{defLa}
\cL_{a}[\uV,\D]V=\begin{cases}
c(\uq) \D_t q +\D_x u\\
{\bf d}[\uV]\D_t u+{\bf l} \big[\big(1+\mu a(\uu)\big)q\big]
\end{cases}
\quad\mbox{ for all }V=(q,u)^T.
\end{equation}
Denoting $V_m=X_m V$, we deduce from Lemmas \ref{lemmaLinear1} and \ref{lemmaLinear2} that
\begin{equation}\label{formQL}
\cL_{a}[V,\D]V_m=\cS_m(V,V_1,V_2)
\end{equation}
where, with the notations of Lemmas \ref{lemmaLinear1} and \ref{lemmaLinear2} for $\cF^{(m)}$ and $\cG^{(m)}_j$, one has
\begin{equation}\label{defSm}
\cS_m(V,V_1,V_2)=\left(\begin{array}{c}
{\mathcal F}^{(m)}(q,q_1,q_2)\\
\boldsymbol{\cG}^{(m)}(V,V_1,V_2)
\end{array}\right)
\quad\mbox{ with }\quad
\boldsymbol{\cG}^{(m)}:=\cG^{(m)}_0+\sqrt{\mu} {\bf l} \cG^{(m)}_1.
\end{equation}
The next section is devoted to the proof of $L^2$-based energy estimates for \eqref{defLa}.
\subsection{Linear estimates}\label{sectlinest}
As seen above, an essential step in our problem is to derive a priori estimates for the linear problem
\begin{equation}
\label{lineq}
\left\{ \begin{aligned}
& c(\uq)  \D_t   q   +  \D_x u  =    f   \\
 & {\bf d}[\uV] \D_t  u  +    {\bf l} \big((1+\mu a(\uu) ) q\big)   =  {\bf g}\quad\mbox{ with }\quad {\bf g}:=g_0 + \sqrt{\mu}  {\bf l}  g_1,
 \end{aligned} \right. 
\end{equation}
where we recall that
$$
c(q)=(2q)^{-3/2}\quad\mbox{ and }\quad a(u)=X_1(uB'(\varphi)).
$$
As we shall see, \eqref{lineq} is symmetrized by multiplying the first equation by $h_0^2$ and the second one by $h_0$; since $c(\uq)>0$ is bounded away from zero and since $({\bf d} u,u)$ controls $\Vert h_0 u\Vert_{L^2}^2+\mu\Vert h_0\D_x u\Vert_{L^2}^2$ (see   the proof of Proposition \ref{enlin} below),  it is natural  to introduce  the  weighted $L^2$ spaces 
\begin{equation}\label{wL2}
L^2_s=h_0^{-s/2} L^2(\RR_+) \mbox{ with the norm } \| u \|^2_{L^2_s} =  \int_{\RR_+} h_0^s | u(x)|^2 dx,
\end{equation}
where $h_0$ is the water height at the initial time. We shall also need to work with the following weighted versions of the $H^1(\RR_+)$ space
 \begin{equation}\label{defH1}
  \cH^1_s  = \{  u \in  L^2_s\ :    \ \sqrt{\mu}h_0 \D_x u \in L^2_s\} \subset  L^2_s
   \end{equation}
   endowed with the norm
\begin{equation}
\label{normH1}
\Vert u \Vert_{\cH^1_s}^2=\Vert u\Vert_{L^2_s}^2+\mu \Vert h_0\D_x u\Vert_{L^2_s}^2
\end{equation}
(the $\mu$ in the definition of the norm is important to get energy estimates uniform with respect to $\mu\in [0,1]$). \\
The dual space of $\cH^1_1$ is then given by
\begin{equation}\label{defHm1}
 \\
  \cH^{-1}_1  = \{ {\bf g}:= g _0 + \sqrt{\mu}{\bf l}  g_1   \ : \     (g_0, g_1)  \in L^2_1\times L^2_1  \} \subset H^{ -1}_{loc} (\RR_+)
\end{equation}
(this duality property is proved in Lemma \ref{lem82} below), with 
\begin{equation}
\label{normHm1} 
\Vert {\bf g}\Vert_{\cH^{-1}_1}^2=\Vert g_0\Vert_{L^2_1}^2+\Vert g_1\Vert_{L^2_1}^2.
\end{equation}
This leads us to define the natural energy space $\VV$ for $V=(q,u)$ and its dual space $\VV'$ by
 \begin{equation}\label{defVV}
 \VV = L^2_2 \times \cH^1_1 \quad \mbox{ and }\quad \VV' = L^2_2 \times \cH^{-1}_1.
 \end{equation}
We can now state the $L^2$ based energy estimates for \eqref{lineq}. Note that these estimates are uniform with respect to $\mu \in [0,1]$.
 \begin{prop} \label{enlin}
Under Assumption \ref{assBh}, let  $T>0$ and assume that
 \begin{equation}\label{assq}
 \uq,\D_t \uq,\frac{1}{\uq}, \uu,\D_t \uu,\D_{t}^2 \uu,\frac{1}{1+\mu a(\uu)}\in L^\infty([0,T]\times \RR_+).
 \end{equation}
 If $(f,{\bf g})\in L^1([0,T];\VV')$,  then if $V=(q,u)$ is a smooth enough solution of  \eqref{lineq}, one has
 $$
 \forall t\in [0,T],\qquad  \big\|  V(t)  \big\|_{\VV} 
 \leq  \underline{c}_1\times\Big[
  \big\| V(0)  \big\|_{\VV} 
+ \int_0^t   \big\| \big(f(t'),{\bf g}(t')\big)  \big\|_{\VV'}  dt' \Big],
 $$
 where $\underline{c}_1$ is a constant of the form 
 \begin{equation}\label{eqC1}
 \underline{c}_1=\underline{c}_1\big(T,\big\Vert (\uq,\D_t \uq,\frac{1}{\uq},\frac{1}{1+\mu a(\uu)},\uu,\D_t \uu)\big\Vert_{L^\infty([0,T]\times \RR_+)}\big).
 \end{equation}
 \end{prop} 
\begin{proof}

Remarking that
$$
\int h_0 ({\bf d}[\uV]u)u=\int h_0 u^2+\mu h_0 \big(\frac{2}{\sqrt{3}}h_0 \uq \D_x u -\frac{\sqrt{3}}{2}B'(\underline{\varphi}) u\big)^2+\mu h_0 \big(\frac{1}{2}B'(\underline{\varphi})u\big)^2
$$
where $\underline{\varphi}(t,x)=x+\int_0^t \uu(s,x)d s$,
the   density of energy is 
 $$
e = \mez \big[  h_0^2  \underline{c}(1+\mu \underline{a})      q^2 + h_0    u ^2  +\mu h_0 (\frac{2}{\sqrt{3}}h_0 \underline{q} \D_x    u -\frac{\sqrt{3}}{2}\underline{B}'   u)^2+\mu h_0 (\frac{1}{2}\underline{B}'   u)^2\big],
 $$
 with $\underline{c}=c(\uq)$, $\underline{a}=a(\uu)$ and $\underline{B}'=B'(\underline{\varphi})$.  We also set
 $$
 E(t):=\int e(t,x)dx. 
 $$
 We shall repeatedly use the following uniform (with respect to $\mu$) equivalence relations
 \begin{equation}\label{equiv1}
 E^\mez\leq C\big(\Vert (\frac{1}{\uq},\uq,\uu,\D_t\uu)\Vert_{L^\infty}\big) \Vert{V}\Vert_{\VV},\qquad
  \Vert{V}\Vert_{\VV}
 \leq C\big(\Vert( \frac{1}{\uq},\uq, \frac{1}{1+\mu \underline{a}})\Vert_{L^\infty}\big)E^{1/2}.
 \end{equation}
One  multiplies  the first equation of \eqref{lineq} by 
  $h_0^2   (1+\mu \underline{a})   q$ and the second by $h_0    u$. Usual integrations by parts show that  
 $$
\begin{aligned}
 \frac{d}{dt}E  =&  
  \mez  \int  \big[ \D_t (\underline{c}(1+\mu \underline{a})) \,   h_0^2    q^2  +h_0   u(\D_t {\bf d}[\uV])    u \big]
  \\
 &  +\int (  h_0^2   (1+\mu \underline{a})   f \,      q    +  h_0     g_0   \,      u   -  \sqrt{\mu} 
 h_0^2        g_1    \D_x    u    \big) . 
 \end{aligned}
 $$
 Remarking further that
 \begin{align*}
\int h_0   u(\D_t {\bf d}[\uV])    u&=\mu \int \big(\frac{8}{3} h_0^3 \uq \D_t \uq(\D_x    u)^2- 2h_0^2 \underline{B}' \D_t \uq    u \D_x    u\big)\\
&=\mu \int \big(\frac{8}{3} h_0^3 \uq \D_t \uq(\D_x    u)^2+ {\bf l}(\underline{B}' \D_t \uq ) h_0   u^2 \big),
 \end{align*}
 we easily deduce that
 \begin{align*}
 \nonumber
 \frac{d}{dt}E  \leq \underline{c}_1 \Vert{V}\Vert_{\VV}^2
 \nonumber
 + \Vert (f,{\bf g})\Vert_{\VV'}
 \Vert V\Vert_{\VV},
 \end{align*}
 with $\underline{c}_1$ as in the statement of the proposition. Integrating in time, using \eqref{equiv1}, and using a Gronwall type argument therefore gives the result. 
\end{proof}

\begin{rem}
\label{rempos}
\textup{The assumption \eqref{assq} contains two types of conditions:  $L^\infty$ bounds  and positivity conditions 
$\underline q >0$ and $1 + \mu a (\underline u) > 0$ which are essential to have a definite positive energy, thus for stability.  }

\end{rem}

 \subsection{The quasilinear system}\label{sectquasi0}
 
As  explained in Remark \ref{remgnl} below, the presence of the topography term $\mu {\bf l} (a(u)X_m q)$ in the equation for $X_m u$ derived in Lemma \ref{lemmaLinear2} makes the problem fully nonlinear. We therefore seek to quasilinearize it by writing an extended system for $V$ and $X_m V$.
 We deduce from the above that $V=(q,u)^T$ and $V_m=X_m V$ ($m=1,2$) solve the following system
 \begin{equation}\label{formQL2}
\begin{cases}
 \cL_{a}(V,\D)V_m&=\cS_m(V,V_1,V_2)\qquad (m=1,2),\\
 \cL (V,\D)V&=\cS(V,V_1,V_2),
 \end{cases}
  \end{equation}
  with $\cS$ and $\cS_m$ as defined in \eqref{defS} and \eqref{defSm} respectively.\\
As we shall show in the next sections, \eqref{formQL2} has a quasilinear structure in the weighted spaces associated to the energy estimates given in Proposition \ref{enlin}, or more precisely, to their higher order generalization.


\section{Main result} \label{sectstrategy}

In this section, we state and outline the proof of a local well posedness result for the shoreline problem for the Green-Naghdi equations \eqref{mod317}. Some necessary notations are introduced in \S \ref{sectnot} and the main result is then stated in \S \ref{sectstate}. An essential step in the proof is a higher order energy estimate stated in \S \ref{sectHO}; a sketch of the proof of this estimate is provided, but the details are postponed to Section \ref{sectcommut}.\\
In order to construct a solution, we want to construct an iterative scheme for the quasilinearized formulation \eqref{formQL2}. Unfortunately, with a classical scheme, the topography terms induce a loss of one derivative; in order to regain this derivative, we therefore introduce an additional variable and an additional elliptic equation (which of course become tautological at the limit). This elliptic equation is introduced in \S  \ref{sectiter} and its regularization properties (with respect to time and conormal derivatives) are stated; their proof, of specific interest, is postponed to Section \ref{Sell}. Solving each step of the iterative scheme also requires an existence theory for the linearized mixed initial value problem; the main result is given in \S \ref{sectexIBVP} (here again the detailed proof is of independent interest and is postponed, see Section \ref{Sex}).
The end of the proof of the main result consists in proving that the sequence constructed using the iterative scheme is uniformly bounded (see \S \ref{sectbounds}) and converges to a solution of the shoreline problem \eqref{mod317} (see \S \ref{sectconv}).



\subsection{Notations}\label{sectnot}

In view of the linear estimate of Proposition \ref{enlin},  it is quite natural to introduce for higher regularity 
based on the spaces $L^2_s$ introduced in \eqref{wL2} , using the derivatives 
$X^\alpha = X_1^{\alpha_1} X_2^{ \alpha_2} = \D_t^{\alpha_1} (h_0 \D_x)^{\alpha_2}$, $\alpha = (\alpha_1, \alpha_2)$. 
We use the following notations.

\begin{defi}
Given a Banach space  $B$  of functions on $\RR_+$,  $C^0_TB^n$ [resp; $L^\infty_TB^n$] [resp. $L^2_TB^n$] denotes the space of functions 
$u $ on $[0, T] \times \RR_+$ such that 
for all $|\alpha | \le n$, $X^\alpha u $ belongs to 
$C^0 ([0, T], B ) $ [resp; $L^\infty([0, T], B ) $] [resp. $L^2([0, T], B ) $], equipped with the obvious norm, 
which is the $L^\infty$ norm or $L^2$ norm of 
\begin{equation}\label{wL2n}
\nonumber
\| u (t)  \|_{B^n} =  \sum_{\vert \alpha\vert \leq n } \|  X^\alpha (t, \cdot ) \|_B . 
\end{equation}
\end{defi} 
We use this definition  for $B = L^2_s$, $\cH^{\pm 1}_s$ or $B = \VV, \VV'$, with  the associated notations 
$$
\| u(t) \|_{L^{2, n}_s} , \quad \| u(t) \|_{\cH^{\pm 1,  n}_s}, \quad  \| U(t) \|_{\VV^n} ,  \quad \| F(t) \|_{\VV'^n}. 
$$
When $B= L^\infty$, we simply write $ L_T^{\infty,p}$ for 
$L^\infty_T L^{\infty,p}$ 
which is  equipped with the norm
\begin{equation}\label{defLinfnT}
\Vert  f   \Vert_{L_T^{\infty,p}}=   \sup_{t \in [0, T] } \Vert  f (t)  \Vert_{L^{\infty,p}}
\end{equation}
where 
\begin{equation}\label{defLinfn}
\Vert  f (t)  \Vert_{L^{\infty,p}}=\sum_{\vert \alpha\vert\leq p}\Vert X^\alpha f (t, \cdot) \Vert_{ L^\infty(\RR_+)}.
\end{equation}


\subsection{Statement of the result}\label{sectstate}

Our main result states that the Cauchy problem for \eqref{mod317} can be solved  locally in time. 
More importantly, for initial data satisfying bounds independent of $\mu$, 
the solutions will exist on an interval of time $[0, T]$ independent of $\mu$. 
We look for solutions  in spaces 
$C^0_T \VV^n$, for $n$ large enough. However, functions in  this space  are not necessarily 
bounded, because of the weights. To deal with nonlinearities,  we have to add additional $L^\infty$ bounds 
on  low order derivatives.  
More precisely we look for solutions  in    
\begin{equation}\label{smoothness}
\begin{cases}
    q  \in C^0_T L^{2,n}_2 \cap C^0_T L^{2,n-1}_1    \cap C^0_T L^{2,n-2} \cap L^{\infty,p}_T  \\
      u   \in C^0_T L^{2,n}_1\cap  C^0_T L^{2, n-1}  \cap L^{\infty,p}_T \\
      \sqrt \mu h_0 \D_x u  \in C^0_T L^{2,n}, 
 \end{cases}
 \end{equation}
satisfying  uniform  bounds in these spaces, where $n$ and $p$ are integers such that  $n \ge 14$ and $2p\geq n\geq p+7$.

Next we describe admissible initial conditions. 
Following Remark~\ref{rempos}, the stability conditions $q > 0$ and $1 + \mu a (u)$ must be satisfied 
at $ t = 0$. The first one is satisfied since  \eqref{eq25} implies that the initial for $q$ is $q^0=1/2$. 
Next, recall that  $a(u)=\D_t (uB'(\varphi))$.  Because, by definition, $\D_t \varphi=u$ and $\varphi(t=0,x)=x$, one has
$$
a(u) (0, x)   =\D_t  u(0, x)  B'(x)+ u(0, x)  B''(x)u(0, x) . 
$$
Thus the condition $1 + \mu a(u) \ge \delta > $ at $t=0$ involves the time derivative $\D_t u$ at $t=0$.  Using 
that $\vp (0, x) = x$ and the equation \eqref{mod317}  under the form  
\begin{equation}
\D_t u=-{\bf d}^{-1}\big( {\bf l}q +\mu {\bf q}_1(V)+B'(\varphi)\big)_{\vert_{t=0}}, 
\end{equation}
(we refer to Section~\ref{Sect9} for the the invertibility of ${\bf d}$)  we see that the right-hand-side 
only involves $u^0$. Hence 
The condition $1 + \mu a(u) \ge \delta >  0$ at $t=0$ can therefore be expressed as a condition on the initial data $u^0$; using the convenient notation of Schochet \cite{schochet1986}, we shall write this condition
\begin{equation}
\label{inita}
\exists \delta>0, \qquad "1 + \mu a(u^0) \ge \delta   "  . 
\end{equation} 
Our result also requires a smallness condition on the contact angle at the origin, which can be formulated as follows
\begin{equation}\label{initan}
\sqrt{\mu} h_0'(0)<\eps,
\end{equation}
 for some $\eps>0$. We can now state our main result. 
\begin{theo}\label{theomain}
 Let $n\geq 14$ and assume that Assumption $\ref{assBh}$ holds. There exists $\eps>0$ such that for all $u^0\in  H^{n+2} (\RR^+)$ and $\mu \in [0,1]$ verifying \eqref{inita} and \eqref{initan} there exists $T=T(\Vert{u^0}\Vert_{H^{n+2}},\delta^{-1})>0$ and a unique classical solution $(q,u)$ to \eqref{mod317} with initial data $(1/2,u^0)$ and  satisfying \eqref{smoothness}.
 \end{theo}
\begin{rem}
With the dimensional variables used in the introduction, one observes that $\sqrt{\mu}h'(0)={\mathtt H}'(0)$. Since moreover, the angle $\alpha_0$ at the contact line is given by the formula
$$
\alpha_0=\arctan\Big( \frac{{\mathtt H}'(0)}{1-({\mathtt H}'(0)+{\mathtt B}'(0)){\mathtt B}'(0)}\Big),
$$
one has $\alpha_0\approx {\mathtt H}'(0)$ when ${\mathtt H'}(0)$ is small, and this is the reason why we say that the condition \eqref{initan} is a smallness condition on the angle at the contact line. Note that a smallness condition on the contact angle is also required to derive a priori estimates for the shoreline problem with the free surface Euler equations \cite{Poyferre,MingWang}.
\end{rem}

\begin{rem} As already said, the condition \eqref{inita} is necessary for the $L^2$ linear stability, since it is required for    the energy to be  positive.  The status of the other  condition \eqref{initan} is more subtle. It is a necessary condition for 
the inverse ${\bf d}^{-1}$ at time $t = 0$ to act in Sobolev spaces. So is has something to do with the consistency of the model for smooth solutions and, at least, seems necessary to construct smooth solutions from smooth initial data. 
\end{rem}


\subsection{Higher order linear estimates}\label{sectHO}

We derived in \S \ref{sectlinest} some $L^2$-based energy estimates for the linear system
\begin{equation}
\label{lineq2}
\left\{ \begin{aligned}
& c(\uq)  \D_t   q   +  \D_x u  =    f   \\
 & {\bf d}[\uV] \D_t  u  +    {\bf l} \big((1+\mu a(\uu) ) q\big)   =  {\bf g}\quad\mbox{ with }\quad {\bf g}=g_0 + \sqrt{\mu}  {\bf l}  g_1.
 \end{aligned} \right. 
\end{equation}
This section is devoted to the proof of higher order estimates. Before stating the main result, let us introduce the following notations, with $\uV=(\uq,\uu)$  and $\cS=(f,{\bf g})$,
\begin{equation}\label{defconstantes}
\begin{cases}
&  {\mathfrak m}_1(V;T) \dsp:= \big\|  V   \big\|_{L_T^{\infty,p}} +\big\Vert (\frac{1}{q},\frac{1}{1+\mu a(u)})\big\Vert_{L^\infty_T} ,\\
&  {\mathfrak m}_2(V;T)\dsp :=   \big\Vert q \big\Vert_{L^\infty_T (L^{2,p+3}\cap L^{2,n-2}_1)}  +
  \big\Vert u \big\Vert_{L^\infty_TL^{2,n-2}} ,  \\
&  {\mathfrak m}(V;T)  \dsp := 
  \big\Vert u \big\Vert_{L_T^{\infty,p+1}} + \big\Vert V\big\Vert_{L^\infty_T \VV^{n-1}} 
   \\
& \tilde {\mathfrak m}( V;T)  \dsp :=  \big\Vert V \big\Vert_{L^2_T\VV^n} +
\big\Vert q \big\Vert_{L^2_TL^{2,n-1}_1}, 
 \\
&  \begin{aligned}  {\mathfrak s}(\cS;T)   \dsp   :=    
 \big\Vert f  &\big\Vert_{C^0_T(L^{2,n-1}_2\cap L^{2,p+3})} \\& 
 +  \big\Vert (g_0,g_1)\big\Vert_{C^0_T (L^{2,n-1}_1\cap L^{2,n-2}\cap L^{\infty,p})}  .
 \end{aligned}
 \end{cases}
 \end{equation}
 Roughly speaking, ${\mathfrak m}_1$ is used to control the constants that appear in the $L^2$ linear estimate of Proposition \ref{enlin}; ${\mathfrak m}_2$ controls quantities that do not have the correct weight to be controlled by the $n-1$-th order energy norm, but that do not have a maximal number of derivatives; ${\mathfrak m}$ is basically the $(n-1)$-th order energy norm; $\widetilde{\mathfrak m}$ is used to control the $n$-th order energy norm (the reason why it involves an $L^2$ rather than $L^\infty$ norm in time is that the control of the $n$-th order energy norm comes from the elliptic regularization of \S  \ref{sectiter}); finally ${\mathfrak s}$ is used to control the source terms.
 \begin{rem}
 \textup{ Note that the parameter $\mu$ enters in the definition of $ \tilde {\mathfrak m}$ since, by \eqref{normH1}, }  
 $$
 \big\Vert V \big\Vert^2_{L^2_T\VV^n}  =  \big\Vert q \big\Vert^2_{L^2_TL^{2,n}_2}
 + \big\Vert u\big\Vert^2_{L^2_TL^{2,n}_1} + \mu \big\Vert h_0\D_x u  \big\Vert^2_{L^2_TL^{2,n}_1}. 
 $$
 \end{rem}

 The higher order estimates can then be stated as follows (note that they are uniform with respect to $\mu \in [0,1]$).

\begin{prop}\label{propHO}
Let $n\geq 14$. Under Assumption \ref{assBh}, let $T>0$, $\uV=(\uq,\uu)$, $\cS=(f,{\bf g})$, and let also $M_1$, $M_2$, $M$, $\widetilde M$ and $S$ be five constants such that
\begin{equation}\label{constMetc}
\begin{array}{lll}
  {\mathfrak m}_1(\uV;T) \leq  M_1,&{\mathfrak m}_2(\uV;T) \leq  M_2,&\quad
  {\mathfrak s}(\cS;T) \leq S,\\
{\mathfrak m}(\uV;T)\leq M ,& \widetilde{\mathfrak m}(\uV;T) \leq \widetilde M .& 
\end{array}
\end{equation}
 There exists a smooth function $\cT(\cdot)$, with a nondecreasing dependence on its arguments, such that
 if $T>0$ satisfies $T\cT\big( M_1, M_2,M,\widetilde{M}\big)<1$, any smooth enough solution  $V=(q,u)$ of \eqref{lineq2} on $[0,T]$ satisfies the a priori estimate
 $$
\Vert V\Vert_{L^\infty_T\VV^{n-1}} \leq C(T, M_1) \big[C_0+\sqrt{T}C(T,M_1,M_2,M,\widetilde{M}) S  \big],
 $$
 where $C_0$ is a constant depending only on the initial data of the form
 $$
 C_0=C_0\big( \Vert V(0)\Vert_{\VV^{n-1}},\Vert u(0)\Vert_{L^{\infty,p+2}},\Vert q(0)\Vert_{L^{2,p-1}\cap L^{\infty,p-1}}\big).
 $$
\end{prop}

\begin{rem}
\label{remgnl}
We emphasize here that the estimate above induces a loss of one derivative in the sense that we need $n$ $X$-derivatives on 
$\uV$ to get estimates of the $(n-1)$-th $X$-derivatives of the solution. This is due to the topography term $a(\uu)$ in \eqref{lineq2}. It is therefore the topography that makes the problem fully nonlinear.
\end{rem}

 \begin{proof}
 We only provide here a sketch of the proof; the details are postponed to Section \ref{sectcommut}.
 Introduce the quantities 
\begin{equation}\label{notaQmj}
\begin{array}{ll}
\cQ_{m,j} (t) =\Vert q(t) \Vert_{L^{2,m}_{2-j}}, & \quad \ucQ_{m,j}(t) =\Vert \uq (t) \Vert_{L^{2,m}_{2-j}} \qquad (j=0,1,2),\\
\cU_{m,j}(t) =\Vert u(t) \Vert_{L^{2,m}_{1-j}}, & \quad \ucU_{m,j}(t) =\Vert \uu (t) \Vert_{L^{2,m}_{1-j}}  \qquad (j=0,1),
\end{array}
\end{equation}
and, for $j=0$, we simply write $\cQ_m=\cQ_{m,0}$, etc. When $j=0$ these quantities correspond to the components of the $m$-th order energy norm; when $j\neq 0$, the number of derivatives involved is the same, but the weight is not degenerate enough to allow a direct control by the energy norm.\\
Throughout this proof, we denote by $p$ an integer such that $p+7\leq n \leq 2p$ (such an integer exists since we assumed that $n\geq 14$).

\smallbreak
{\bf Step 1.} Applying $X^\alpha$,  to the first equation of \eqref{lineq2}, we obtain
\begin{equation}\label{eqstep1}
c(\uq)\D_t X^\alpha q +\D_x X^\alpha u=f^{(\alpha)},
\end{equation}
where the source term $f^{(\alpha)}$ satisfies on $[0,T]$ and for $\vert \alpha\vert \leq n-1\leq 2p$,
\begin{equation}
\label{eststep1}
\big\|  f ^{(\alpha)} (t)  \big\|_{L^2_2}    \le   S
+C( M_1)  \Big( M \Vert q(t) \Vert_{L^{\infty,p}}  + \cQ_{n-1} (t)   \Big). 
\end{equation}
This estimate is  proved in Proposition \ref{PropNL1} below.
\smallbreak

{\bf Step 2.} Applying $X^\alpha$, with $\vert \alpha\vert \leq n-1\leq 2p$, to the second equation of \eqref{lineq2}, we obtain
\begin{equation}\label{eqstep2}
 {\bf d}[\uV] \D_t  X^\alpha u  +    {\bf l} \big((1+\mu a(\uu) ) X^\alpha q\big)   =  {\bf g}^{(\alpha)} 
 \end{equation}
where $ {\bf g}^{(\alpha)}=g_0^{(\alpha)} + \sqrt{\mu}  {\bf l}  g_1^{(\alpha)}$
satisfies on $[0,T]$ and with $\cH^{-1}_1$ as in \eqref{defHm1},
\begin{equation}
 \label{eststep2}
 \begin{aligned}
 \big\|  {\bf g}^{(\alpha)}  (t) \big\|_{\cH_1^{-1}}    \le   
    C (T,M_1)  \Big(
 \| V(t) \|_{L^{\infty, p} }  \big( M + \| \uu (t) \|_{L^{2, n}_1}  + \| \uq(t) \|_{L^{2, n-1}_1}  \big) \\
+\cQ_{n-2,1 }(t) +\cU_{n-1}(t) \Big)  + S . 
\end{aligned}
\end{equation} 
This assertion will be proved in Proposition \ref{PropNL2}. 
We now need to control  the terms $\cQ_{n-2,1}$ and $\Vert V\Vert_{L^{\infty,p}}$ that appear in \eqref{eqstep1} and \eqref{eqstep2} in terms of the energy norm $\Vert V\Vert_{\VV^{n-1}}$; roughly speaking, we must trade some derivatives to gain a better weight. This is what we do in the following two steps.

\smallbreak
{\bf Step 3.} To control  $\cQ_{n-2,1}$, we use the Hardy  inequality 
$$
\Vert  f (t) \Vert_{L^2_1}\lesssim \Vert {\bf l} f(t) \Vert_{L^2_1}+\Vert f(t)  \Vert_{L^2_2}
$$
which is proved in Corollary \ref{coroH2}. 
Using the definition \eqref{defd} of ${\bf d}[\uV]$, the equation \eqref{eqstep2} implies that 
\begin{equation}\label{orag}
    {\bf l} q^{(\alpha)}  =  g_0^{(\alpha)} +\mu \uq \underline{B}'X_2X_1X^\alpha u-\mu (\underline{B}')^2X_1X^\alpha u.
\end{equation}
with $\underline{B}'=B'(\underline{\varphi})$ and
\begin{align*}
q^{(\alpha)}&=(1+\mu a(\uu) ) X^\alpha q-\sqrt{\mu} g_1^{(\alpha)}+\mu \frac{4}{3}\uq^2 X_2 X_1 X^\alpha u+\mu \uq \underline{B}' X_1 X^\alpha u.
\end{align*}
Using the above Hardy inequality on this equation satisfied by $q^{(\alpha)}$ for $\vert \alpha\vert \leq n-2$, we show in Proposition \ref{PropHq} that
 \begin{align*}
\cQ_{n-2, 1} (t)  \le C (T,M_1,M_2,M) \big[C_0+ \| V(t) \|_{L^{\infty, p}}  +\Vert V  \Vert_{L^\infty_T\VV^{n-1}} +S\big].
\end{align*}
\smallbreak

{\bf Step 4.} Control of $\Vp$. We need a control on $X^\alpha u$ and $X^\alpha q$ in $L^\infty$ for $\vert \alpha\vert \leq p$. For $X^\alpha u$, we use the Sobolev embedding
$$
\Vert X^\alpha u\Vert_{L^\infty}\lesssim \Vert X^\alpha u\Vert_{L^2}+ \Vert \D_x X^\alpha u\Vert_{L^2}
$$
and use the equation \eqref{eqstep1} to control the last term. 
For $X^\alpha q$, we need another Hardy inequality  proved in Corollary \ref{coroH3}
$$
\Vert f (t) \Vert_{L^\infty}\lesssim \Vert {\bf l}u (t)  \Vert_{L^\infty\cap L^2}+\Vert f(t) \Vert_{L^2},
$$
which we apply to \eqref{orag}. This is the strategy used in Proposition \ref{propcontinf} to prove that for $T$ small enough (how small depending only on $M$), one has
$$
 \Vert V(t) \Vert_{L^{\infty,p}}\leq C(T,M_1,M_2,M)\big[ C_0+\Vert V\Vert_{L^\infty_T\VV^{p+6}}+S\big].
 $$
 
 \smallbreak
{\bf Step 5.}  Since $p+6\leq n-1$, we deduce from Steps 1-4 that
\begin{align*}
\big\| f^{(\alpha)}(t) \big\|_{L^2_2}& \leq \CMTT \big[C_0+\Vert V\Vert_{L^\infty_T\VV^{n-1}}+S\big], \\
\big\|   {\bf g}^{(\alpha)}(t)\big\|_{\cH_1^{-1}}  &  \le  \CMTT\big[C_0+\Vert V\Vert_{L^\infty_T\VV^{n-1}}+S\big] (1+\ucU_n(t)+\ucQ_{n-1,1}(t)).
\end{align*}
Using the linear energy estimates of Proposition \ref{enlin} and summing over all $\vert \alpha\vert \leq n-1$, this implies that
\begin{align*}
\Vert V\Vert_{L^\infty_T\VV^{n-1}} \leq C(T, M_1) \Big[C_0+&\CMTT\big(C_0+\Vert V\Vert_{L^\infty_T\VV^{n-1}}+S\big)  \\
&\qquad\times\int_0^T(1+\ucU_n(t)+\ucQ_{n-1,1}(t))dt \Big].
\end{align*}
Remarking that 
\begin{align*}
\int_0^T(1+\ucU_n(t)+\ucQ_{n-1,1}(t))dt&\leq T+\sqrt{T}  \vert \ucU_n,\ucQ_{n-1,1} \vert_{L^2(0,T)} \\
&\leq T+\sqrt{T}\widetilde{M},
\end{align*}
we finally get
$$
\Vert V\Vert_{L^\infty_T\VV^{n-1}} \leq C(T, M_1)
  \big[C_0+\sqrt{T}C(T,M_1,M_2,M,\widetilde{M})\big(C_0+\Vert V\Vert_{L^\infty_T\VV^{n-1}}+S\big)  \big].
$$
For $T$ small enough (how small depending only on $M_1$, $M_2$, $M$ and $\widetilde{M}$, this implies that
\begin{align*}
\Vert V\Vert_{L^\infty_T\VV^{n-1}} \leq C(T,& M_1) \big[C_0+\sqrt{T}C(T,M_1,M_2,M,\widetilde{M})S  \big],
\end{align*}
which completes the proof of the proposition.
\end{proof}
 
 
 \subsection{The iterative scheme}\label{sectiter}
 
 We derived in the previous section higher order energy estimates for the linearized equations from which a priori estimates for the nonlinear problem can be deduced. The question is how to pass from these estimates to an existence theory. One possibility (in the spirit of \cite{CS1,CS2}), could be to use a parabolic regularization of the equations. To  avoid boundary layers, such a regularization should be degenerate at the boundary; in the presence of dispersive and topography terms no general existence theorem seems available and we therefore choose to work with an alternative approach based on elliptic regularization.

The goal of this section is to propose an iterative scheme to solve the quasilinearized equation \eqref{formQL2}. A naive tentative would be to consider
the following iterative scheme
$$
 \begin{cases}
 \cL_{a}(V^k,\D)V^{k+1}_m&=\cS_m(V^k,V^k_1,V^k_2)\qquad (m=1,2),\\
 \cL (V^k,\D)V^{k+1}&=\cS(V^k, V_1^k,V_2^k),
 \end{cases}
 $$
 with $\cL$ and $\cL_a$ as in \eqref{defL} and \eqref{defLa} respectively. With such an iterative scheme however, the nonlinear estimates of Proposition \ref{propHO} say that if $V^{k}$ and $V_m^{k}$ ($m=1,2$) have respectively $\VV^{n}$ and  $\VV^{n-1}$ regularity then  $V^{k+1}$ and $V_m^{k+1}$ have only $\VV^{n-1}$ regularity: there is a loss of one derivative, as already noticed in Remark~\ref{remgnl}. If we know that $V_m^{k+1}=X_m V^{k+1}$ then the $\VV^n$ regularity for $V^{k+1}$ is recovered, but this information  is not propagated by the iterative scheme (even though it is true at the limit).

A usual way to circumvent the loss of derivatives is to use a Nash-Moser scheme, but here the definition of the
the smoothing operators would be delicate because we need weighted and non weighted norms. So we proceed in a different way and  introduce an additional variable $V_0$ whose purpose is to make the regularity of the family $(V^{k})_k$ one order higher than the regularity of $(V_1^k,V_2^k)_k$. Instead of \eqref{formQL2}, we rather consider
 \begin{equation}\label{formQL3}
\begin{cases}
 \cL_{a}(V,\D)V_m&=\cS_m(V,V_1,V_2)\qquad (m=1,2),\\
 \cL (V,\D)V&=\cS(V,V_1,V_2),\\
 {\bf E}(\D)V&=F(V_0,V_1,V_2)
 \end{cases}
  \end{equation}
and the iterative scheme we shall consider should therefore be of the form
 $$
 \begin{cases}
 \cL_{a}(V^k,\D)V_m^{k+1}&=\cS_m(V^k,V_1^k,V_2^k)\qquad (m=1,2),\\
 \cL (V^k,\D)V_0^{k+1}&=\cS(V^k,V_1^k,V_2^k),\\
 {\bf E}(\D)V^{k+1}&=F(V_0^{k+1},V_1^{k+1},V_2^{k+1})
 \end{cases}
 $$
 with ${\bf E}(\D)$ an elliptic operator (in time and space) so that $V^{k+1}$ is more regular than $F(V_0^{k+1},V_1^{k+1},V_2^{k+1})$, and the sequence is defined for all $k$. 
 
 We choose the elliptic equation 
 \begin{equation}
 \label{ellipeq} 
 {\bf E} (\D) V =  V_1 + {\bf F}_2 V_2 + {\bf F}_0 V_0 
 \end{equation}
 in such a way that it is tautological when $ V_0 = V$,  $ V_1 = X_1 V$ and $ V_2 =  X_2 V$. 
 We choose 
  \begin{equation} \label{defE}
 {\bf E} = \D_t + {\bf P}  , \quad
 {\bf F}_2 =  -   X_2    P^{ - 1} , \quad {\bf F}_0 = \kappa^2 P^{-1} 
 \end{equation} 
 where ${\bf P} = (\kappa^2 - X_2^2)^\mez = ( \kappa^2 - (h  \D_x)^2 ) ^\mez $  and $\kappa > 2$.  
 Note that this corresponds to a scalar  equation for each component $q$ and $u$. We consider the Cauchy problem
 for \eqref{ellipeq}. The independent proof of the next proposition is given in Section~\ref{Sell}. Note that the gain of one derivative only occurs if we consider an $L^2$-norm in time (this is the reason why we had to introduce the constant $\widetilde{\mathfrak m}(V;t)$ in \eqref{defconstantes}.
 \begin{prop}\label{propell}
  Let $n-1\in \NN$. For $(V_0, V_1, V_2)$ in   $L^\infty([0,T];\VV^{n-1})$ and initial data in $\VV^n$, the Cauchy problem for 
  \eqref{ellipeq} has a unique solution in $L^2([0,T];\VV^{n})\cap C([0,T];\VV^{n-1})$ and 
 \begin{align*}
\Vert V \Vert_{L^2([0,T];\VV^n)}&\lesssim \Vert V(0)\Vert_{\VV^n} +  \sqrt{T}\big(\Vert V_0\Vert_{L^\infty_T\VV^{n-1}}+\Vert V_1\Vert_{L^\infty_T\VV^{n-1}}+ \Vert V_2\Vert_{L^\infty_T\VV^{n-1}}\big)\\
 \Vert V\Vert_{L^\infty_T\VV^{n-1}}  &\lesssim \Vert V(0)\Vert_{\VV^{n-1}} +  T\big(\Vert V_0\Vert_{L^\infty_T\VV^{n-1}}+\Vert V_1\Vert_{L^\infty_T\VV^{n-1}}+ \Vert V_2\Vert_{L^\infty_T\VV^{n-1}}\big).
 \end{align*}
Moreover, the following $L^\infty$ bounds also hold,
 $$
 \Vert V\Vert_{L^{\infty,p}_T} \lesssim \Vert V(0)\Vert_{L^{\infty,p}} + \sqrt{T} \big(  \Vert V_0\Vert_{L^{\infty,p}} + \Vert V_1\Vert_{L^{\infty,p}} + \Vert V_2\Vert_{L^{\infty,p}} \big).
 $$
 \end{prop}

 As a conclusion, the iterative scheme we shall consider is the following
 \begin{equation}\label{scheme}
 \begin{cases}
 \cL_{a}(V^k,\D)V_m^{k+1}&=\cS_m(V^k,V_1^k,V_2^k)\qquad (m=1,2),\\
 \cL (V^k,\D)V_0^{k+1}&=\cS(V^k,V_1^k,V_2^k),\\
 {\bf E}(\D)V^{k+1}&=V^k_1 + {\bf F}_2 V^k_2 + {\bf F}_0 V^k_0
 \end{cases}
 \end{equation}
 (the choice of the first iterate $k=1$ will be discussed in \S \ref{sectexIBVP} below).
 We take the natural initial data  :  first we choose
 \begin{equation}
 \label{inititer1}
 V^{k}_{|  t = 0} =  V^{k}_0{}_{| t = 0}  = (\mez, u^0), \quad V^{k}_2{}_{| t = 0}
 = (0,  X_2 u^0) . 
 \end{equation}  
For the initial value of $V^{k}_1$, we take the data given by the equation \eqref{mod317} evaluated at $t = 0$
 \begin{equation}
\label{inititer2}V^{k}_1{}_{| t = 0}
 =   \big( - \frac{1}{c (q^0)} \D_x u^0, {\bf d}_0^{-1} (B'(x) - {\bf l} q^0 - \mu q_1{}_{ |t = 0} ) \big) 
 \end{equation}  
where ${\bf d}_0$ is the operator ${\bf d}$ at time $0$, which is known since it involves only the initial values of $q$ and $\vp$, that is $q^0 = \mez$ and $\vp^0 = x$ (the invertibility properties of ${\bf d}_0$ are discussed in Section \ref{Sect9}).


\subsection{The initial values of the time derivatives} \label{sectIVTD}

Because the right-hand-side of the energy estimates involve norms of $\D_t^j V^k _{| t = 0} $ (through $C_0$ in Proposition \ref{propHO} for instance), we have to show that theses quantities remain bounded through the iterative scheme. 

The initial values of $\D_t^j V$ are computed by induction on $j$, writing the equation \eqref{formFL} under the form 
$$
\D_t V  = \cA  (V) 
$$
where $\cA$ is a non linear operator acting on $V$. However, $\cA$ involves ${\bf d}^{-1}$, and it is easier to commute first the equation \eqref{formFL} with $\D_t^j$, before applying ${\bf d}^{-1}$. 
This yields an induction formula
\begin{equation}
\label{initdtj}
\D_t^{j+1} V_{| t = 0} = \cA _j  ( V_{| t = 0}, \ldots, \D_t^{j} V_{| t = 0}) 
\end{equation}
where the $\cA_j$ are non linear operators which involve only $\D_x$ derivatives and 
${{\bf d}^0}^{-1}$, 
where ${\bf d}^0 = { \bf d}_{| t = 0}$ is independent of the initial value $u^0$.

In particular, starting from $V^0_{in}  = (\mez, u^0)$ with $u^0$ sufficiently smooth,  this formula defines by induction 
functions $V^j_{in} $, as long as ${\bf d}_0$ can be inverted. In particular, if this allows to define a smooth enough $V_{app}$  such that  
\begin{equation}
\label{Vapp}
\D_t^j V_{app}{}_{| t = 0}  = V^j_{in}, \qquad j \le n . 
\end{equation} 
then  $V_{app} $ is an approximate solution of \eqref{formFL} in the sense of Taylor expansions up to order $n-1$. This is made precise in Section~\ref{Sect9} where we prove the following proposition that will play a central role in the construction of the first iterate of the iterative scheme in the next section.
\begin{prop}
\label{propvapp}
There is $\epsilon_n > 0$, which depends   only on  $n$, such that, for $\sqrt{\mu}h'(0) \le \eps_n$ and  $u^0\in H^{n+2} (\RR_+)$, the 
$V^j_{in} $ are well defined in $H^{ n+1 -j}  (\RR_+)$ so that 
there is a $V_{app} \in H^{n+1} ([0, 1] \times \RR_+)$ satisfying \eqref{initdtj}.  
\end{prop}

 From now on, we assume that the condition $\sqrt{\mu}h'(0) \le \eps_n$ is satisfied.

\medbreak

An important  remark is  that the $V^j_{in}$ remain the initial data of the time derivatives of the solutions, all along the 
 iterative scheme \eqref{scheme}. 
\begin{prop}
\label{propinit}
Suppose that $(\uV, \uV_0, \uV_1, \uV_2) $ is smooth and satisfies for $j \le n-1$, 
\begin{equation}
\label{initk} 
\left\{\begin{aligned} 
&
\D_t^j \uV{}_{| t = 0}  = \D_t^j \uV_0{}_{| t = 0}= V^j_{in},  \quad  
\\
&\D_t^j \uV_1 {}_{| t = 0} = V^{j+1}_{in}, 
 \quad  
\D_t^j \uV_2 {}_{| t = 0} = X_2 V^{j}_{in}, 
\end{aligned}\right. 
\end{equation} 
then any smooth solution $(V, V_0, V_1, V_2) $ of 
\begin{equation}
\label{theeq} 
 \begin{cases}
 \cL_{a}(\uV ,\D)V_m &=\cS_m( \uV, \uV_1, \uV_2 )\qquad (m=1,2),\\
 \cL (\uV ,\D)V_0 &=\cS(\uV, \uV_1, \uV_2 )\\
 {\bf E}(\D)V &=F(V_0,V_1,V_2)
 \end{cases}
 \end{equation}
with initial conditions
\begin{equation}
 \label{theinit}
 \begin{cases}
 V_{|  t = 0} =  V_0{}_{| t = 0}  = (\mez, u^0), \quad V_2{}_{| t = 0}
 = (0,  X_2 u^0) . 
 \\
V_1{}_{| t = 0}
 =  V^1_{in} =  \big( - \frac{1}{c (q^0)} \D_x u^0, {\bf d}_0^{-1} (B'(x) - {\bf l} q^0 - \mu q_1{}_{ |t = 0} ) \big) 
\end{cases}
 \end{equation}  
also satisfies  the conditions \eqref{initk}. 
\end{prop} 

\begin{proof}
The proof is by induction on $j$. This is true for $j = 0$ by the choice of the initial conditions.   Expliciting the time derivative, it is clear that the 
$\D_t^j V_*{}_{| t = 0} $ can be computed by induction on $j$, in a unique way. 
Therefore it is sufficient to show that the solution to \eqref{theeq}, \eqref{theinit} satisfies the required condition \eqref{initk}
This is true because 
$ (\D_t V , X_2 V, V, V) $ is  an approximate solution of \eqref{formQL3}  in the sense of Taylor expansions up to order $n-2$. 
\end{proof}


 \subsection{Construction of solutions for the linearized mixed initial boundary value problem}
 \label{sectexIBVP}
 
 We have already proved in Proposition \ref{propell} that the initial value problem for the elliptic equation 
 is well posed. In order to construct a sequence of approximate solutions $(V^k)_k$ using the iterative scheme \eqref{scheme}, it remains to solve linear problems of the form
 \begin{equation}\label{eqlinex}
 \cL_a(\uV,\partial)V= F.
 \end{equation}

 They do not enter in a known framework, because of the dispersive term of the second equation and also because of the weights. However, one can solve such systems using a  scheme which we now sketch. 
 The key ingredient  are the high-order a priori estimates 
 proved in Proposition \ref{propHO}. 
 We proceed as follows. 
 
\quad 1.  Assume first that  the coefficients $a$ and $\underline V$ are very smooth. The linear system can be cast in a variational form, and the {\sl a priori} estimates for the backward problem imply the existence of weak solutions in 
weighted $L^2$ spaces. 

\quad 2. Using tangential  mollifications (convolutions  in time) and variations on Friedrichs' Lemma, one proves that 
the weak solutions are strong, that is limit of smooth solutions. Therefore they satisfy the {\sl a priori} estimates 
in  $L^2$ and in weighted Sobolev spaces. 

\quad 3. Approximating the coefficients, this implies the existence of solutions when $a$ and $\underline V$ have 
the limited smoothness.

\quad 4. The gain of weights and $L^\infty$ estimates are proved using Hardy type inequalities. 
 
 Details are given in Section~\ref{Sex} below.  
 The next proposition summarizes the useful conclusion for the Cauchy problem with vanishing initial condition.

 \medbreak

 Consider  $\underline V $  and $F $ such that the quantities 
 \begin{equation}
   {\mathfrak m}_1(\uV ;T) ,  \quad  \tilde {\mathfrak m}_2(\uV;T),\quad 
 {\mathfrak m} (\uV;T),  \quad 
\tilde {\mathfrak m}(\uV;T), \quad \tilde {\mathfrak s}(F;T) , 
 \end{equation}    
defined at \eqref{defconstantes} are finite. Suppose in addition that 
\begin{equation}
\label{initnul}
\D_t^jF_{| t = 0} = 0 , \quad j \le n-2. 
\end{equation}
\begin{prop}\label{propexist}
 Suppose that $F \in L^2_T \VV'{}^{n-1}$ satisfies \eqref{initnul}.  Then, the Cauchy problem for  \eqref{eqlinex}  with initial data  $V_{| t = 0} = 0$ 
   has a unique solution in $C^0_T\VV^{n-1}$ and 
 $   \D_t^j V_{| t = 0} = 0 $ for $ j \le n-2$. 
\end{prop}

Together with Propositions~\ref{propell} for the elliptic equation and \ref{propinit}  to treat the initial condition, 
one can now solve the linearized equations \eqref{theeq} with the initial conditions~\eqref{theinit} and define the iterates  $[V^k] := (V_1^k,V_2^k,V_0^k,V_2^k)$.  We proceed as follows.

Consider a smooth initial data $u^0 \in   H^{n+2} (\RR_+)$. Using  Proposition~\ref{propvapp}   introduce  
$V_{app} \in H^{n+1} ( [0, T] \times \RR_+) $ satisfying the conditions \eqref{Vapp}.
We start the iteration scheme with  
\begin{equation}
\label{itere1}
[V^1] = \big(V_1^1,V_2^1,V_0^1,V^1 \big) = \big(V_{app},V_{app},\D_t V_{app},  X_2 V_{app}\big),  
\end{equation} 
and next define the sequence $([V^k] )_k$ by induction by solving \eqref{scheme} \eqref{inititer1} \eqref{inititer2}. 
Indeed, assume that 
the quantities 
 \begin{equation}
 \label{inducmin}
 \begin{cases}
   \quad \Vert V_m^k \Vert_{L^{\infty,p}}, \quad 
   {\mathfrak m}(V^k;T),  \quad  \tilde {\mathfrak m}(V^k;T)\\
 {\mathfrak m}_1(V^k;T), \quad 
{\mathfrak m}_2(V^k;T), \quad  {\mathfrak m}_2(V_m^k;T)\\
  \end{cases}
 \end{equation}
 are finite and that 
 \begin{equation}
\label{inititerk} 
\left\{\begin{aligned} 
&
\D_t^j V^k{}_{| t = 0}  = \D_t^j V^k_0{}_{| t = 0}= V^j_{in},  \quad  
\\
&\D_t^j V^k_1 {}_{| t = 0} = V^{j+1}_{in}, 
 \quad  
\D_t^j V^k_2 {}_{| t = 0} = X_2 V^{j}_{in}  
\end{aligned}\right.  , \qquad j \le n-1. 
\end{equation} 
 The following lemma is proved in the next section.
 \begin{lem}
 \label{lemgb49}
  The quantities ${\mathfrak s}(\cS^{k};T)$ and ${\mathfrak s}(\cS_m^{k};T)$ associated to 
  the source terms 
  $ \cS^k=\cS(V^k,V_1^k,V_2^k)$ and  $\cS_m^k(V^k,V_1^k,V_2^k)$ with $\cS$ and $\cS_m$ as given in \eqref{defS} and \eqref{defSm}, are finite. 
 \end{lem} 
 
 We look for  $[V^{k+1} ]  $ as $[V^1] + [ \delta V^k]$, where 
 $[ \delta V^k] = (\delta V_1^k,\delta V_2^k,\delta V_0^k,\delta V^k)$ solves a system of the form
  \begin{equation}\label{scheme0}
 \begin{cases}
 \cL_{a}(V^k,\D) \delta V_m^{k}&= \delta \cS^k_m \qquad (m=1,2),\\
 \cL (V^k,\D) \delta V_0^{k}&= \delta \cS^k ,\\
 {\bf E}(\D) \delta V^{k}&= \delta F^k , 
 \end{cases}
 \end{equation}
with vanishing initial condition $ [\delta V^k]_{| t = 0} = 0$. By  Proposition~\ref{propinit}, $[V^{k}]$ is an approximate solution of \eqref{formQL3} in the sense of Taylor expansion at  $t = 0$, and thus the source term
$[\delta \cS^k ] = (\delta \cS^k_1, \delta \cS^k_2, \delta \cS^k , \delta F^k) $
 satisfies: 
\begin{equation}
\D_t^j  [\delta \cS^k] _{| t = 0} = 0, \quad j \le n-2. 
\end{equation} 
Hence,   Propositions~\ref{propexist} and \ref{propell} imply the following result: 

\begin{prop}
\label{propgb410}
Under the assumptions above, 
the equation \eqref{scheme0} \eqref{inititer1} \eqref{inititer2} has a solution 
$[V^{k+1} ] = (V_1^{k+1},V_2^{k+1},V_0^{k+1},V^{k+1}) $  with each term in 
$C^0_T\VV^{n-1}$. Moreover, it satisfies 
\eqref{inititerk}. 
\end{prop} 
The last step needed is the following proposition, proved in the next section together with precise bounds on the different quantities. 
\begin{prop}
\label{propgb411} 
The quantities 
$ \Vert V_m^{k+1} \Vert_{L^{\infty,p}_T}$, 
   ${\mathfrak m}(V^{k+1};T) $ ,  $\tilde {\mathfrak m}(V^{k+1};T) $, 
$ {\mathfrak m}_1(V^{k+1};T)$, 
$ {\mathfrak m}_2(V^{k+1};T)$, ${\mathfrak m}_2(V_m^{k+1};T)
$
 are finite. 
\end{prop}


 \subsection{Bounds on the sequence $(V_1^k,V_2^k,V_0^k,V^k)_k$}\label{sectbounds}

We have just constructed  
  a sequence $(V_1^k,V_2^k,V_0^k,V^k)_k$.
   We want to prove that it converges as $k\to \infty$ to a solution $(V_1,V_2,V_0,V)$ of the quasilinearized equations \eqref{formQL3}. The first step consists in establishing the following uniform bounds on this sequence, where we used the notations given in \eqref{defconstantes},  and for some constants $M_1,M_2,M,\widetilde{M},N_1,N_2,S$ to be chosen carefully,
   \begin{equation}\label{induc}
 \begin{cases}
   {\mathfrak m}(V^k;T) \leq M,&\\
  \tilde {\mathfrak m}(V^k;T) \leq \widetilde{M},&\\
 {\mathfrak m}_1(V^k;T)\leq M_1, &\quad\mbox{and}\quad  \Vert V_m^{k}\Vert_{L^{\infty,p}} \leq N_1\quad \,\, (m=0,1,2),\\
{\mathfrak m}_2(V^k;T)\leq M_2& \quad \mbox{and}\quad  {\mathfrak m}_2(V_m^k;T)\leq N_2\quad (m=0,1,2),
  \end{cases}
 \end{equation}
 and a constant $S$ such that
 \begin{equation}
 \label{inducter}
  {\mathfrak s}(\cS^{k};T)\leq S,\quad \mbox{and}\quad{\mathfrak s}(\cS_m^{k};T)\leq S,
  \end{equation}
for $m=0,1,2$ and $k\in \NN$.

 \begin{prop}\label{propbounds}
 There exists $T>0$ and some nonnegative constants $M_1$, $M_2$, $M$, $\widetilde{M}$, $N_1$, $N_2$ and $S$ such that the bounds \eqref{induc} hold for all $k\in {\mathbb N}$.
 \end{prop}
 \begin{proof}
 Here again, we only sketch the proof and postpone the details to \S \ref{sectproofbounds}.\\
The proof is  by induction on $k$. In order to show that \eqref{induc}$_{k+1}$ holds if  \eqref{induc}$_k$ is satisfied, we first derive the necessary bounds on
$V_m^{k+1}$ ($m=0,1,2$) which are a consequence of the higher order estimates of Proposition \ref{propHO} for the $\VV^{n-1}$ estimates, and of Proposition \ref{propcontinf} for the estimates based on $L^\infty$. The required estimates on $V^{k+1}$ are then deduced from the estimates on $V_m^{k+1}$ using the elliptic regularization properties stated in Proposition \ref{propell}. These results are rigorously stated and proved in Lemma \ref{lemMM}.\\
These upper bounds are then used to prove Lemma \ref{lemsource}, which provides the required estimates on $\cS^{k+1}$ and $\cS^{k+1}_m$.
\end{proof}

\begin{rem}
\textup{We note that the proof of the proposition which gives precise bounds,  includes a proof of the Lemma~\ref{lemgb49} 
and Proposition~\ref{propgb411} above. } 
\end{rem}

\subsection{Convergence and end of the proof of Theorem \ref{theomain}}\label{sectconv}

We show here that the sequence constructed in the previous sections converges to a solution of \eqref{formQL3}, and that the solution $(V_1,V_2,V_0,V)$ satisfies $V_0=V$, $V_1=V$ and $V_2=X_2 V$ if these identities are satisfied at $t=0$. It follows that $V$ is the solution claimed in the statement of Theorem \ref{theomain}.

Let us write $W^{k+1}:=V^{k+1}-V^k$, $W^{k+1}_m:=V_m^{k+1}-V_m^k$ ($m=0,1,2$). 
From \eqref{scheme0}, these quantities solve
\begin{equation}\label{eqdiff}
\begin{cases}
 \cL_{a}(V^k,\D)W_m^{k+1}=&\widetilde\cS_m^k,\qquad (m=1,2)\\
 \cL (V^k,\D)W_0^{k+1}=&\widetilde\cS^k,\\
 {\bf E}(\D)W^{k+1}=&W^k_1 + {\bf F}_2 W^k_2 + {\bf F}_0 W^k_0 ,
 \end{cases}
\end{equation}
with, using again the notation $\cS_m^k=\cS_m(V^k,V_1^k,V_2^k)$, etc., 
\begin{align*}
\widetilde\cS_m^k&:=\big(\cS_m^k-\cS_m^{k-1}\big) -\big(\cL_{a}(V^k,\D)-\cL_{a}(V^{k-1},\D)\big) V_m^k,\\
\widetilde\cS^k&:=\big(\cS^k-\cS^{k-1}\big) -\big(\cL(V^k,\D)-\cL(V^{k-1},\D)\big) V_0^k.
\end{align*}
Using the bounds proved on the sequences $(V^k)_k$ and $(V_m^k)_k$ in Proposition \ref{propbounds} one easily gets that the right-hand-side in \eqref{eqdiff} has a Lipschitz dependence on $(W^k,W_0^k,W_1^k,W_2^k)$. Taking a smaller $T$ if necessary, one can therefore classically show that 
the series $V^{k+1}-V^0=\sum_{j=1}^{k+1}W^j$ and $V_m^{k+1}=V_m^0+\sum_{j=1}^{k+1}W_m^j$ converge in $\VV$ to some functions $V$ and $V_m$ in $C([0,T];\VV)$. Using again the bounds provided by Proposition \ref{propbounds} and interpolation inequalities, one obtains that $(V,V_0,V_1,V_2)$ is a classical solution of 
$$
 \begin{cases}
 \cL_{a}(V,\D)V_m&=\cS_m(V,V_1,V_2)\qquad (m=1,2),\\
 \cL (V,\D)V_0&=\cS(V,V_1,V_2),\\
 {\bf E}(\D)V&=V_1 + {\bf F}_2 V_2 + {\bf F}_0 V_0
 \end{cases}
$$
and that $V\in L^2([0,T];\VV^n)$ and $V_0,V_m\in L^\infty([0,T];\VV^{n-1})$ for $m=1,2$.\\
We now need to prove that $V_m=X_m V$ and $V=V_0$ if these quantities coincide at $t=0$. Differentiating the equation on $V_0$ with respect to $X_m$, one gets
$$
\cL_{a}(V,\D)X_mV_0=\widetilde\cS(V_0,V,V_1,V_2),
$$
where the exact expression for $\widetilde\cS(V_0,V,V_1,V_2)$ can be obtained as for Lemma \ref{lemmaLinear2}.
Writing $Z_m:=V_m-X_m V_0$, one obtains therefore that
$$
 \cL_{a}(V,\D)Z_m=\cS_m(V,V_1,V_2)-\widetilde\cS(V_0,V,V_1,V_2)  \qquad (m=1,2),
$$
and (using the equation to substitute $X_1q_m$ as in \eqref{eqX1qm} below), one easily gets that the right-hand-side has a Lipschitz dependence on $V_m-X_m V=Z_m+X_m (V-V_0)$ and $V-V_0$ in $L^2([0,T];\VV)$. Remarking further that
$$
{\bf E}(\D)V_0=X_1 V_0 + {\bf F}_2 X_2 V_0 + {\bf F}_0 V_0.
$$
and taking the difference with the above equation on $V$, we get through Proposition \ref{propell} that $V-V_0$ is controlled in $L^2([0,T];\VV^1)$ by $Z_m$ in $L^2([0,T];\VV)$. We get therefore   from Gronwall's inequality that $V_0=V$, $V_1=V$ and $V_2=X_2 V$ if these identities are satisfied at $t=0$, which concludes the proof of Theorem \ref{theomain}.

 \section{Hardy type inequalities} 
\label{Sec6} 
As explained in the previous section, we shall need Hardy type inequalities to obtain non-weighted estimates on $X^\alpha u$ and $X^\alpha q$ using the equations. We prove here several general Hardy type inequalities of independent interest; the inequalities we shall actually use are the particular cases stated in Corollaries \ref{coroH1}, \ref{coroH2} and \ref{coroH3}. Throughout this section, we shall denote by $h$ any function $h\in C^1([0,\infty))\cap L^\infty(\RR^+)$ satisfying
\begin{equation}\label{assh}
h(0)=0, \quad h'(0)>0, \quad h(x)>0\mbox{ for all }x>0,\mbox{ and } \liminf_{ x\to \infty} h(x)>0.
\end{equation}
We also need to introduce the operator $D_\alpha$ defined as
\begin{equation}\label{eqDalpha}
D_\alpha u := h \D_x u+\alpha h'(x)u=h^{1-\alpha}\D_x (h^\alpha u)
\end{equation}
\begin{prop} 
\label{Hardy1} 
Let $p\in [1,\infty]$ and $\alpha,\sigma\in \RR$ be such that $\sigma>\alpha-1/p$.
If $h$ is as in \eqref{assh} and $h^\sigma D_\alpha u \in L^p (\RR^+)$ and $h^{\sigma+1} u \in L^p(\RR^+)$, then $h^\sigma u \in L^p(\RR^+)$ and
$$
\Vert h^\sigma  u \Vert_{L^p} \lesssim \Vert h ^\sigma D_\alpha u\Vert_{L^p} +\Vert h^{\sigma+1} u\Vert_{L^p}.
$$
\end{prop}
\begin{proof}
Let $\chi$ be a smooth positive function such that $\chi(0)=1$ and for some $X_2>0$, $\chi(x)=0$ for all $x\geq X_2$. We decompose $u$ into
$$
u=u_1+u_2,\qquad u_1:=\chi u,\qquad u_2:=(1-\chi)u.
$$
Let $f:=D_\alpha u_1$; one has 
$$
h(x)^\sigma u_1(x)=-\int_x^\infty \frac{h(y)^{\alpha-1}}{h(x)^{\alpha-\sigma}}f(y)dy\\
$$
Since $u_1$ and $f$ are supported in $[0,X_2]$, one has
\begin{align*}
h(x)^\sigma\vert u_1(x) \vert &\lesssim \int_{x}^\infty \frac{y^{\alpha-\sigma-1}}{x^{\alpha-\sigma}}y^\sigma\vert f(y)\vert dy\\
&=\int_1^\infty t^{\alpha-\sigma- 1}\vert (tx)^\sigma f(tx)\vert dt,
\end{align*}
and therefore
$$
\Vert h^\sigma u_1\Vert_{L^p} \leq \big( \int_1^\infty t^{\alpha-\sigma -1 -1/p} dt\big) \Vert x^\sigma f\Vert_{L^p}
$$
Remarking that $\Vert x^\sigma f\Vert_{L^p}\lesssim \Vert h^\sigma D_\alpha u\Vert_{L^p} +\Vert h^{\sigma+1} u\Vert_{L^p}$ (recall that $h(x)\sim x$ on $[0,X_2]$), we deduce that
$$
\Vert h^\sigma u_1\Vert_{L^p} \lesssim \Vert h^\sigma D_\alpha u\Vert_{L^p} +\Vert h^{\sigma+1} u\Vert_{L^p}
$$
provided that the integral in $t$ converges, which is the case if $\sigma   > \alpha-1/p$.\\
Since for $u_2$, one trivially has $\Vert h^\sigma u_1\Vert_{L^p} \lesssim \Vert h^{\sigma+1} u\Vert_{L^p}$, the result follows.
\end{proof}
We shall use in this paper the following direct corollary of Proposition \ref{Hardy1} (just take $p=2$, $\sigma=0$ and $\alpha=0$).
\begin{cor}\label{coroH1}
Assume that $h$ is as in \eqref{assh} and that $h \D_x u\in L^2(\RR^+)$ and $h^{\sigma_1}u\in L^2(\RR^+)$ for some $0\leq \sigma_1\leq 1$. Then $u\in L^2(\RR^+)$ and
$$
\Vert u \Vert_2\lesssim \Vert h\D_x u\Vert_2+\Vert h^{\sigma_1} u \Vert_2.
$$
\end{cor}
\begin{prop}
\label{Hardy2} 
Let $p\in [1,\infty]$ and $1/p+1/q=1$, and assume that $p\geq q$.
Let $u$ be compactly supported and assume that for $\sigma  < \alpha - 1/p$, 
one has $h^{\sigma}  D_\alpha  u \in  L^p (\RR^+)$ and $h^{\sigma_1} u \in L^q(\RR^+)$ where $\sigma \leq \sigma_1\leq \alpha -1/q$; then 
$ h^{\sigma} u \in L^p(\RR^+)$
and
$$
\Vert h^\sigma u \Vert_{L^p}\lesssim \Vert h^{\sigma}  D_\alpha  u \Vert_{L^p}.
$$
\end{prop} 
\begin{proof}
Let $X_2>0$ be such that $u$ is supported in $[0,X_2]$. With  $ f = D_\alpha u   $  introduce 
 $$
u_1 (x) =  \int_0^x \frac{h(y)^{\alpha-1}} {h (x)^\alpha} f(y) dy . 
$$
Since $\sigma<\alpha-1+1/q$, one has $h(y)^{\alpha-\sigma-1 }  \sim  y^{\alpha - \sigma -1} \in L^{q}  ([0, X_2])$; therefore,  the integral converges and 
$$
| h(x)^\sigma u_1 (x) | \lesssim  \int_0^x \frac{y^{\alpha-1- \sigma}} {x^{\alpha- \sigma}}    g(y)  dy = 
 \int_0^1 t^{\alpha - \sigma-1} g (t x) dt 
$$
with $g(y) = | y^\sigma f(y)| $. Moreover, since  $\alpha- \sigma - 1/p > 0$, 
$$
\| h^\sigma u_1 \|_{L^p} \le  
\frac{1}{\alpha- \sigma - 1/p}   \| g \|_{L^p}\lesssim \Vert h^\sigma D_\alpha u\Vert_{L^p}.
 $$
 Thus $h^\sigma u_1$ and hence  $h^{\sigma_1} u_1$  belong to $L^p$. Since $p\geq q$, their restriction to $[0,X_2]$ also belongs to $L^q([0,X_2])$. From the assumption made on $u$, we deduce  that $h^{\sigma_1}(u-u_1)\in L^q([0,X_2])$.
 Remarking further that $D_\alpha (u - u_1) = 0$, we have  $u - u_1 = c h^{- \alpha} $ and therefore  $h^{\sigma_1}(u-u_1)\sim c x^{\sigma_1-\alpha}$; this quantity  has to be in $L^q([0,X_2])$ and the condition on $\sigma_1$ 
 implies that this is possible only if $c = 0$. 
 Hence $u = u_1$ and the proposition is proved. 
\end{proof}
Taking $p=q=2$ in Proposition \eqref{Hardy2}, one gets the following corollary.
\begin{cor}\label{coroH2}
Assume that $\sigma<\alpha-1/2$ and that $h^\sigma D_\alpha u\in L^2(\RR^+)$ and $h^{\sigma_1}u\in L^2(\RR^+)$ with $\sigma +1\geq \sigma_1\geq \sigma$ and $\sigma_1\leq \alpha-1/2$. Then one has $h^\sigma u\in L^2(\RR^+)$ and
$$
\Vert h^\sigma u\Vert_{L^2}\lesssim \Vert h^\sigma D_\alpha u\Vert_{L^2}+\Vert h^{\sigma_1} u \Vert_{L^2}.
$$
\end{cor}
\begin{proof}
Let $\chi$ be a smooth positive function such that $\chi(0)=1$ and for some $X_2>0$, $\chi(x)=0$ for all $x\geq X_2$. We decompose $u$ into
$$
u=u_1+u_2,\qquad u_1:=\chi u,\qquad u_2:=(1-\chi)u.
$$
Remarking that
$$
D_\alpha u_1=\chi D_\alpha u +h\chi' u,
$$
one has $h^\sigma D_\alpha u_1\in L^2(\RR^+)$ since $\sigma+1\geq \sigma_1$, and one can apply Proposition \ref{Hardy2} to $u_1$ with $p=q=2$. This yields
$$
\Vert h^\sigma u_1\Vert_{L^2}\lesssim \Vert h^\sigma D_\alpha u\Vert_2+\Vert h^{\sigma_1} u\Vert_2.
$$
Since $u_2$ is supported away from the origin, we also get from \eqref{assh} that
$$
\Vert h^\sigma u_2\Vert_{L^2} \lesssim \Vert h^{\sigma_1} u\Vert_{L^2}
$$
and the result follows.
\end{proof}
We shall also need the following corollary corresponding to the case $p=\infty$, $q=1$.
\begin{cor}\label{coroH3}
If  $\sigma\leq  \alpha-1$, $h^\sigma D_\alpha u\in L^2\cap L^\infty(\RR^+)$ and $h^{\sigma+1} u\in L^2(\RR^+)$ then $h^\sigma u \in L^2(\RR^+)$ and
$$
\Vert h^\sigma u \Vert_{L^\infty}\lesssim \Vert h^\sigma D_\alpha u\Vert_{L^\infty}+\Vert h^\sigma D_\alpha u\Vert_{L^2}+\Vert h^{\sigma} u \Vert_{L^2}.
$$
\end{cor}
\begin{proof}
With the same decomposition as in the proof of the corollary above, we have by the proposition that
$$
\Vert h^\sigma u_1 \Vert_{L^\infty}\lesssim \Vert h^\sigma D_\alpha u_1\Vert_{L^\infty}\lesssim \Vert h^\sigma D_\alpha u\Vert_{L^\infty} +\Vert h^{\sigma+1} u \Vert_{L^\infty}.
$$
Since $u_2$ is supported away from zero, we also deduce from \eqref{assh} that
$$
\Vert h^\sigma u_1 \Vert_{L^\infty}\lesssim \Vert h^{\sigma+1} u \Vert_{L^\infty};
$$
the result follows therefore from the observation that
\begin{align*}
 \Vert h^{\sigma+1} u \Vert_{L^\infty}&\lesssim \Vert h^{\sigma+1} u \Vert_{L^2}+\Vert \D_x (h^{\sigma+1} u)\Vert_{L^2}\\
 &\lesssim \Vert h^{\sigma} u \Vert_{L^2}+\Vert h^\sigma D_\alpha u\Vert_{L^2}.
\end{align*}
\end{proof}


\section{Technical details for the proof of Theorem \ref{theomain}} \label{sectcommut}

We prove here some technical results stated in Section \ref{sectstrategy} and used to prove Theorem \ref{theomain}. The first one is the proof of the higher order estimates of Proposition \ref{propHO}, presented in \S \ref{sectproofHO} below. The second one is the proof of the bounds on the sequence of approximate solutions constructed using the iterative scheme \eqref{scheme}; this is done in \S \ref{sectproofbounds}. Other elements of the proof of Theorem \ref{theomain} are of independent interest and are therefore presented in specific sections: see Section \ref{Sec6} for the Hardy estimates, Section \ref{Sell} for the analysis of the elliptic equation, and Section \ref{Sex} for the resolution of the mixed initial boundary value problem for the linearized equations.

\medbreak

{\it We recall that, for the sake of clarity, we do not track the dependance on $h_0$ and $B$ in the various constants that appear in the proof}.\\
Before proceeding further, let us state the following product and commutator estimates, whose proof is straightforward and therefore omitted. Recalling that the spaces $L^{2,s}_j$ and $L^{\infty,p}$ have been  introduced in \eqref{wL2} and \eqref{defLinfn} respectively, we have
 for all $m\geq 1$ and $\beta \in \NN^2$ such that $\vert \beta\vert\leq m$, and for $j=0,1,2$,
\begin{align}
\label{comm1}
\Vert X^\beta (fg)\Vert_{L^\infty}&\lesssim \Vert f \Vert_{\infty,m} \Vert g \Vert_{L^{\infty,m}}\\
\label{comm2}
\Vert [X^\beta,f]g\Vert_{L^2}&\lesssim \Vert f\Vert_{L^{\infty,{[\frac{m}{2}]}}} \Vert   g\Vert_{L^{2,m-1}_j}+\Vert g\Vert_{L^{\infty,{[\frac{m}{2}]}}} \Vert   f\Vert_{L^{2,m}_j}.
\end{align}
We also use the simplified notations 
$$
\cQ_{m,j}(t)=\Vert q(t)\Vert_{L^{2,m}_{2-j}},\quad \cU_{m,j}(t)=\Vert u(t)\Vert_{L^{2,m}_{1-j}},\quad \mbox{etc}
$$
 as in \eqref{notaQmj}, as well as
 $$
 \cQ_{m,j}[t]=\sup_{t'\in [0,t]}\cQ_{m,j}(t'), \quad \cU_{m,j}[t]=\sup_{t'\in [0,t]}\cU_{m,j}(t'), \quad \mbox{etc.}
 $$
\subsection{Technical details for the proof of Proposition \ref{propHO}}\label{sectproofHO}

Proposition \ref{propHO} has been proved in \S \ref{sectHO} assuming several technical results that we establish here.

The following proposition gives the equation on $X^\alpha q$ obtained by applying $X^\alpha$ to the first equation of \eqref{lineq2} and a control of the residual that has been used in Step 1 of the proof of Proposition \ref{propHO}.
\begin{prop}  
\label{PropNL1}
 Let $T>0$ and $\uV,\cS$ satisfy the bounds \eqref{constMetc}. Any smooth solution $V=(q,u)$  satisfies on  $[0,T]$, and for  $|\alpha |  \le m \le 2p  $, 
  $$
c(\uq)\D_t X^\alpha q +\D_x X^\alpha u=f^{(\alpha)},
$$
where, for $j \in \{ 0, 1, 2\} $ and $t \in [0, T]$,
$$
\big\| h_0^{ 1 - j/2}  f ^{(\alpha)} (t)  \big\|_{L^2}    \le    \Vert h_0^{1-j/2} f  (t) \Vert_{L^{2,m}}+C(M_1)   \big(\Vert q\Vert_{L_T^{\infty,p}}\ucQ_{m, j} (t)   + \cQ_{m, j} (t)   \big).
$$
\end{prop} 
\begin{proof}
Consider the first equation written as 
$$
h_0 c(\uq)\D_t q  - X_2 u = 0. 
$$
In this form, the second term commutes with $X^\alpha$. Applying $ h_0^{-1} X^\alpha$ to this equation we obtain that $f^{(\alpha)}$ is a linear combination 
  \begin{equation}\label{eqfalpha}
  f^{(\alpha)}  =X^\alpha f+ \sum  c_* (\uq)  \frac{1}{h_0}X^{\alpha_0}h_0X^{\alpha_1} \uq \ldots X^{\alpha_{k-1}}\uq X^{\alpha_k} q ,    
\end{equation} 
where  the  $c_*$ are derivatives of the function $c(\cdot)$ and the indices satisfy $  \sum | \alpha_i | = | \alpha | + 1$ and $ 1 \le | \alpha_i| \le | \alpha |$ if $i\ge 1$.   Moreover, there is at most one $i\geq 1$ such that $| \alpha_{i} | >p$ (if none, we can choose  ${i}$ as we want). If $i=k$ (i.e. if the higher order derivative is on $q$),   and remarking that $X_2^{\alpha_2}h_0=O(h_0)$, we have
\begin{align*}
\big\| h_0^{1- j/2}  \frac{1}{h_0}X_2^{\alpha_0}h_0X^{\alpha_1} q \ldots X^{\alpha_k} q (t) \big\|_{L^2} 
&\lesssim M_1^{k-1}  \big\| h_0^{1- j/2}  X^{\alpha_k} q (t)\big\|_{L^2} \\
&\lesssim  M_1^{k-1} \cQ_{m, j}(t).  
\end{align*}
If $i<k$ (i.e. if the higher order derivative is on $\uq$),   the same quantity is bounded from above by 
$$
M_1^{k-2} \Vert q(t)\Vert_{L^{\infty,p}} \big\| h_0^{1- j/2}  X^{\alpha_i} \uq (t)\big\|_{L^2} 
\lesssim  M_1^{k-2}  \Vert q(t)\Vert_{L^{\infty,p}}  \ucQ_{m, j}(t),
$$
and the proposition follows easily.
\end{proof}

Similarly, an equation on $X^\alpha u$ is obtained by applying $X^\alpha$ to  the second equation of \eqref{lineq2}; the control of the residual provided below has been used in Step 2 of the proof of Proposition \ref{propHO}. We recall that $a(\uu)=X_1(\uu B'(\underline{\varphi}))$ and that $\underline{U}_m[t]=\sup_{t'\in [0,t]}\underline{U}_m(t')$. 
\begin{prop}  
\label{PropNL2}
Let $T>0$ and $\uV,\cS$ satisfy the bounds \eqref{constMetc}. Any smooth solution $V=(q,u)$  satisfies on  $[0,T]$, and for   $|\alpha |  \le m $ and $m+1\le 2p  $, 
$$
 {\bf d}[\uV] \D_t  X^\alpha u  +    {\bf l} \big((1+\mu a(\uu) ) X^\alpha q\big)   =  g_0^{(\alpha)} + \sqrt{\mu}  {\bf l}  g_1^{(\alpha)}.
$$
where the source terms $g_0^{(\alpha)} $ and $   g_1^{(\alpha)}$ satisfy on $[0,T]$,
\begin{align*} 
\nonumber
\big\|  (g_0 ^{(\alpha)} ,g_1 ^{(\alpha)} )(t) \big\|_{L^2_1}   \le 
\uCp\big[&\Vpt\big(\ucQ_{m,1}(t)+\ucU_m[t]+\ucU_{m+1}(t)\big)\\
&+\cQ_{m-1,1}(t)+\cU_{m}(t)\big] +\big \Vert (g_0,g_1) (t)\big\Vert_{L^{2,m}_1},
\end{align*} 
with $\uCp=C(T,\ucM_p(T))$, as well as
\begin{align*} 
\nonumber
\big\|   (g_0 ^{(\alpha)} ,g_1 ^{(\alpha)} ) (t)\big\|_{L^2}   \le   \uCp\big[&\Vpt\big(\ucQ_{m,2}(t)+\ucU_{m,1}[t]+\ucU_{m+1,1}(t)\big)\\
&+\cQ_{m-1,2}(t)+\cU_{m,1}(t)\big] +\big\|   (g_0  ,g_1 ) (t)\big\|_{L^{2,m}}  
\end{align*} 
and
\begin{align*}
\big\| (g_{0}^{(\alpha)},g_{1}^{(\alpha)}) (t)\big\|_{L^\infty} \le \uCp\Big[& \Vpt \big(\Vert \uq(t)\Vert_{L^{\infty,m}}+\Vert \uu(t)\Vert_{L^{\infty,m+1}}\big)\\
&+\Vert q (t)\Vert_{L^{\infty,m-1}}+\Vert u (t)\Vert_{L^{\infty,m+1}}\Big]+\Vert (g_0,g_1)(t)\Vert_{L^{\infty,m}}.
\end{align*}
\end{prop}   
\begin{proof}
We only prove the fist and third estimates; the second one is established like the first one, with a straightforward adaptation. We first state the following lemma that we shall use to control the topography terms.
\begin{lem}
\label{lem51} 
Let $T>0$. For $j = 0, 1$,  $ |\alpha | \le m \le 2p$ and $t \in [0, T]$ ,
 $$
 \big\| h_0^{(1-j)/2}  X^{\alpha}  B' (\underline{\vp})  (t) \big\|_{L^2}  \lesssim  \uCp   \ucU_{m, j} [t] .  
 $$
 The same control holds for $X^\alpha X_1 B' (\underline{\vp}) $.
\end{lem} 
\begin{proof}[Proof of the lemma] By the chain rule, 
 \begin{equation}\label{XalphaB}
 X^{\alpha}  B' (\underline{\vp}) = \sum  B_*(\underline{\vp}) X^{\alpha_1} \underline{\vp} \ldots X^{\alpha_n} \underline{\vp}   ,\quad  \ \sum \alpha_j = \alpha, \   \alpha_j \ne 0,
\end{equation} 
where the $B_*$ are derivatives of the function $B'$.
Because  $\D_t  \underline{\vp}  = \uu $, we easily conclude that for $j= 0, 1$ 
 $$
 \big\|  X^{\alpha_k } \underline{\vp} (t)\big\|_{L^\infty} \lesssim 1 +tM_1 , \qquad
   \big\| B_*(\underline{\varphi})h_0^{(1-j)/2}X^{\alpha_l  } \underline{\vp}(t) \big\|_{L^2} \lesssim 1+t\ucU_{m, j}[t]
 $$
 if $| \alpha_k | \le p   $ and $|\alpha_l  | \le m $ respectively. Since at most one index $|\alpha_l |  > p$, 
 the first estimate of the lemma follows. The fact that the same estimate also holds for $X^\alpha X_1 B' (\underline{\vp})$ stems from the fact that this term is of the form \eqref{XalphaB} with $X^{\alpha_j}\underline{\vp}$ replaced by $X^{\alpha_j} \uu$ for some $j$.
  \end{proof} 
 Applying $X^\alpha$, with $\vert \alpha\vert \leq m$, to the second equation of \eqref{lineq2}, we obtain
 \begin{align*}
  {\bf d}[\uV] \D_t  X^\alpha u  &+    {\bf l} \big((1+\mu a(\uu) ) X^\alpha q\big)   =  -[X^\alpha, {\bf d}[\uV]] \D_t u-\mu {\bf l}\big([X^\alpha,a(\uu)]q\big)\\
&  -[X^\alpha,2h_0'] \big((1+\mu a(\uu) ) q\big)  
  +X^\alpha g_0 + \sqrt{\mu}  X^\alpha {\bf l}  g_1,
 \end{align*}
 where we used the fact that $[X^\alpha,{\bf l}]=[X^\alpha,2h_0']$.
 We now consider and give controls on the different components of the right-hand-side of the above equation.\\
 - {\it Control of $-[X^\alpha, {\bf d}[\uV]] \D_t u-\mu {\bf l}\big([X^\alpha,a(\uu)]q\big)$}. From the definition \eqref{defd} of ${\bf d}[\uV]$,  we can write
 $$
 -[X^\alpha, {\bf d}[\uV]] \D_t u-\mu {\bf l}\big([X^\alpha,a(\uu)]q\big)=g_{0}^i+\sqrt{\mu}{\bf l} g_{1}^i,
 $$
 with 
 \begin{align*}
 g_{0}^i=&   \mu [X^\alpha,\uq\underline{B}']X_2X_1 u-\mu [X^\alpha,(\underline{B}')^2]X_1 u \\
 &+\mu [X^\alpha,2h_0']\big(\frac{4}{3}\uq^2 X_2X_1 u+\uq X_1u \underline{B}'\big) \\
 g_{1}^i=&\sqrt{\mu}\frac{4}{3}[X^\alpha , \uq^2] X_2 X_1 u
-\sqrt{\mu} [X^\alpha,\underline{B}'\uq] X_1 u- \sqrt{\mu}[X^\alpha,a(\uu)]q,
 \end{align*}
where we recall that $a(\uu)=X_1(\uu B'(\underline{\varphi}))$.\\
 From the product and commutator estimates \eqref{comm1}-\eqref{comm2}, we deduce that for $j=0,1$, one has
 \begin{align*}
  \Vert  g_{j}^i (t)\Vert_{L^2_1}\leq \uCp \big[&\Vpt\big(\ucQ_{m,1}(t)+\ucU_{m}[t]+\ucU_{m+1}(t)\big)+\cQ_{m-1,1}(t)+\cU_{m}(t)\big].
 \end{align*}
 - {\it Control of $[X^\alpha,2h_0'] \big((1+\mu a(\uu) )  q\big)=:g_{0}^{ii}$.} Since $[X_1,h_0']=0$ and $[X_2,h_0']=h_0 h_0''=O(h_0)$, and recalling that $a(\uu)=X_1(\uu B'(\underline{\varphi}))$, one readily deduces from  the product and commutator estimates \eqref{comm1} and \eqref{comm2} that
 $$
  \Vert  g_{0}^{ii}(t) \Vert_{L^2_1}\leq \uCp\big[\Vpt\ucU_{m}[t]+\cQ_{m-1}(t)\big].
 $$
 - {\it Control of $X^\alpha g_0 + \sqrt{\mu}  X^\alpha {\bf l}  g_1$.} We can write
   $$
  X^\alpha g_0 +\sqrt{ \mu}  X^\alpha {\bf l}  g_1=g_0^{iii}+\sqrt{\mu} {\bf l}g_1^{iii}
  $$
  with
  $$
  g_0^{iii}=X^\alpha g_0+\mu [X^\alpha,2h_0'] g_1\quad \mbox{ and }\quad g_1^{iii}=X^\alpha g_1,
  $$
  so that we easily get
  $$
  \Vert g_0^{iii}(t)\Vert_{L^2_1} + \Vert g_1^{iii}(t)\Vert_{L^2_1}
  \lesssim \Vert (g_0,g_1)(t)\Vert_{L^{2,m}_1}.
  $$
  The second estimate of the proposition then follows directly by setting $g_j^{(\alpha)}=g_j^i+g_0^{ii}+g_j^{iii}$ for $j=0,1$.\\
 Finally, for the $L^\infty$ estimates we use \eqref{comm1} together with the following straightforward commutator estimate, 
$$
\forall \vert \beta \vert \leq  m,\qquad \Vert [X^\beta,f]g\Vert_{L^\infty}\lesssim \Vert f\Vert_{L^{\infty,{[\frac{m}{2}]}}} \Vert  g\Vert_{L^{\infty,m-1}}+\Vert g\Vert_{L^{\infty,{[\frac{m}{2}]}}} \Vert  f\Vert_{L^{\infty,m}},
$$
and  follow the same steps as above.
\end{proof}

  The controls provided by Propositions \ref{PropNL1} and \ref{PropNL2} involve the quantities $\cQ_{n-2,1}$ and $\Vp$ that need to be controlled in terms of the energy norm $\Vert V\Vert_{\VV^{n-1}}$. For the first quantity, this means that we need to gain a factor $h_0$ in the weighted $L^2$-norms on $q$, possibly loosing one derivative. This is done in the following proposition which is based on the Hardy type inequalities derived in Section \ref{Sec6}, and which has been used to prove Step 3 of the proof of Proposition \ref{propHO}.
 \begin{prop} 
 \label{PropHq}  Let $T>0$ and $\uV,\cS$ satisfy the bounds \eqref{constMetc}. There is a constant $\uCp=C(T,M_1)$ such that for all smooth solution $V=(q,u)$  of \eqref{lineq2} on  $[0,T]$, the following three properties hold, for all $t\in [0,T]$,\\
{\bf i.} For $m\leq 2p$, one has
 \begin{align*}
&\cU_{m,1}(t)\lesssim \Vert f(t)\Vert_{L^{2,m}_2}+\cU_m(t)+
\uCp\big(\Vpt\ucQ_{m}(t)+\cQ_{m+1}(t) \big),
\end{align*}
{\bf ii.} For $p<m+1 \le  2p$, one has
\begin{align*}
&\cQ_{m, 1} (t) \le \uCp\big[\Vert q(0)\Vert_{L^{2,p-1}} +\Vpt\big(1+\ucQ_{m,1}(t)+\ucU_{m+1}[t]\big)\\
&\phantom{\cQ_{m, 1}  \le \uCp\big[}+\cQ_m[t]+\cU_{m+1} (t)+\big \Vert  (g_0,g_1) (t)\big\Vert_{L^{2,m}_1}\big],
\end{align*}
{\bf iii.} For $p<m+1 \le  2p$, one has
\begin{align*}
\cQ_{m, 2} (t) \le \uCp\big[&\Vert q(0)\Vert_{L^{2,p-1}}+\Vpt\big(1+\ucQ_{m,2}(t)+\ucQ_{m+1}(t)+\ucU_{m+1,1}[t]\big)\\
&+\cQ_{m+2}[t]+\cU_{m+1}(t) +\big\Vert  f(t)\big\Vert_{L^{2,m+1}_2}  + \big \Vert  (g_0,g_1)(t) \big\Vert_{L^{2,m}}\big].
\end{align*}
 \end{prop} 
  \begin{rem}
 A quick look at the proof shows that when $m=p$ or $m=p+1$, the estimate given in the first point of the proposition can be simplified into
 \begin{equation}\label{Hqsimp}
\cU_{p,1}(t)\lesssim \Vert f(t)\Vert_{L^{2,p}_2}+\cU_p(t)+
\uCp \cQ_{p+1}(t)
\end{equation}
and
 \begin{equation}\label{Hqsimp2}
\cU_{p+1,1}(t)\lesssim \Vert f(t)\Vert_{L^{2,p+1}_2}+\cU_{p+1}(t)+
\uCp \big(\cQ_{p+1}(t)+\ucQ_{p+1}(t)\Vert X_1 q(t)\Vert_{L^\infty}\big).
\end{equation}
 \end{rem}
 \begin{proof}
 {\it Throughout this proof, the dependence on $t$ is omitted when no confusion is possible.}\\
{\bf i.} Control of $\cU_{m, 1} $.
We first rewrite the equation \eqref{eqstep1} under the form
$$
h_0 \D_x X^\alpha u=h_0 f^{(\alpha)}-c(\uq)h_0\D_t X^\alpha q,
$$
for $| \alpha | \le m$. Using the Hardy inequality provided by Corollary \ref{coroH1} (with $\sigma_1=1/2$), we deduce
 $$
 \big\| X^\alpha u \big\|_{L^2} \lesssim   \big\| h_0^\mez X^\alpha u  \big\|_{L^2} + 
  \big\| c(\uq)h_0 X^\alpha \D_t q   \big\|_{L^2}  +  \big\|h_0  f^{(\alpha)} \big\|_{L^2} ,
 $$
which, together with the control on $\big\|h_0  f^{(\alpha)}  \big\|_{L^2}$ provided by  Proposition \ref{PropNL1} yields the result.\\
{\bf ii.} Control of $\cQ_{m, 1} $. Using the definition \eqref{defd} of ${\bf d}[\uV]$, we can rewrite the equation \eqref{eqstep2} on $X^\alpha q$ under the form
\begin{equation}\label{Hardyq}
    {\bf l} q^{(\alpha)}  =  -X_1 X^\alpha u+g_0^{(\alpha)} +\mu \uq \underline{B}'X_2X_1X^\alpha u-\mu (\underline{B}')^2X_1X^\alpha u
\end{equation}
with
\begin{equation}\label{defqalpha}
q^{(\alpha)}=(1+\mu a(\uu) ) X^\alpha q-\sqrt{\mu} g_1^{(\alpha)}+\mu \frac{4}{3}\uq^2 X_2 X_1 X^\alpha u+\mu \uq \underline{B}' X_1 X^\alpha u,
\end{equation}
so that
$$
X^\alpha q=\frac{1}{1+\mu a(\uu) }\Big[q^{(\alpha)}+\sqrt{\mu} g_1^{(\alpha)}+\mu \frac{4}{3}\uq^2 X_2 X_1 X^\alpha u-\mu \uq \underline{B}' X_1 X^\alpha u\Big].
$$
In particular, for $| \alpha | \le m$,    we have
$$
\Vert  X^\alpha q\Vert_{L^2_1}\leq \uCp\times \big( \big\|  q^{(\alpha)} \big\|_{L^2_1}+\Vert \mu  g_1^{(\alpha)}\Vert_{L^2_1}+\cU_{m+1} \big).
$$
We now use on \eqref{Hardyq} the Hardy inequality provided by Corollary \ref{coroH2} (with $\alpha=2$, $\sigma=1/2$ and $\sigma_1=1$) to obtain
\begin{align}
\nonumber
\big\|  q^{(\alpha)} \big\|_{L^2_1} &\lesssim  \uCp \cU_{m+1 }   +    
\big\|  g_{0}^{(\alpha)}  \big\|_{L^2_1} + \big\|  q^{(\alpha)} \big\|_{L^2_2}\\
\label{esthardqa}
&\lesssim \uCp(\cQ_m+\cU_{m+1})  +    
\big\|  (g_{0}^{(\alpha)}  , g_{1}^{(\alpha)} ) \big\|_{L^2_1}.
\end{align}
It follows that 
$$
\Vert  X^\alpha q\Vert_{L^2_1}\lesssim \uCp(\cQ_m+\cU_{m+1}  ) +    
\big\| (g_{0}^{(\alpha)}  , g_{1}^{(\alpha)} ) \big\|_{L^2_1}.$$
Using Proposition \ref{PropNL2} and summing over all $\vert \alpha\vert \leq m$, this yields
\begin{align*}
\cQ_{m, 1}  \le \uCp\big[&\Vp\big(\ucQ_{m,1}+\ucU_m[t]+\ucU_{m+1}\big)\\
&+\cQ_{m-1,1}+\cQ_m+\cU_{m+1} +\big \Vert  (g_0,g_1) \big\Vert_{L^{2,m}_1}\big],
\end{align*}
and by induction, we deduce that
\begin{align}
\nonumber
\cQ_{m, 1}  \le \uCp\big[&\Vp\big(\ucQ_{m,1}+\ucU_m[t]+\ucU_{m+1}\big)\\
\label{intermestQ}
&+\cQ_{p-1,1}+\cQ_m+\cU_{m+1} +\big \Vert  (g_0,g_1) \big\Vert_{L^{2,m}_1}\big],
\end{align}
so that we are therefore left to control $\cQ_{p-1,1}$. We actually prove below a stronger result, namely, a control of $\cQ_{p-1,2}$. For $| \alpha | \le    p-1$,  integrating $\D_t X^\alpha q$ and using Assumption \ref{assBh} on $h_0$ for the second inequality, we see that on $[0,T]$,
\begin{align*}
\big\| X^\alpha q  (t) \big\|_{L^2} & \le
\big\| X^\alpha q  (0) \big\|_{L^2}+t\Big(\sup_{[0,t]}\Vert \D_t X^\alpha q\Vert_{L^2(0,1)}+\sup_{[0,t]}\Vert \D_t X^\alpha q\Vert_{L^2(1,\infty)}\Big)\\
& \le \big\| X^\alpha q  (0)  \big\|_{L^2} + t \Big(\sup_{[0,t]}\Vert \D_t X^\alpha q\Vert_{L^\infty(0,1)}+\sup_{[0,t]}\Vert h_0 \D_t X^\alpha q\Vert_{L^2(1,\infty)}\Big)
\end{align*}
Summing over all $| \alpha | \le    p-1$ and using Assumption \ref{assBh}, this yields
\begin{equation}\label{estinduc}
\cQ_{p-1,2}(t) \le  \underline C_1 \Big(  \Vert q(0)\Vert_{L^{2,p-1}} + t    \big(\Vert q\Vert_{L^\infty_tL^{\infty,p}}+Q_p[t]\big) \Big) .
\end{equation}
Plugging this estimate into \eqref{intermestQ}, we get the second estimate of the proposition.\\
{\bf iii.} Control of $\cQ_{m, 2} $. We proceed as for ${\bf ii.}$ but replace \eqref{esthardqa} by the inequality obtained by using Corollary \ref{coroH2} (with $\alpha=2$, $\sigma=0$ and $\sigma_1=1$), namely, for $\vert \alpha\vert\leq m$,
\begin{align}
\nonumber
\big\|  q^{(\alpha)} \big\|_{L^2} &\lesssim  \uCp\cU_{m+1,1 }   +    
\big\|  g_{0}^{(\alpha)}  \big\|_{L^2} + \big\|  q^{(\alpha)} \big\|_{L^2_2}\\
\label{estL2g0}
&\lesssim \uCp(\cQ_m+\cU_{m+1,1})  +    
\big\|  (g_{0}^{(\alpha)}  , g_{1}^{(\alpha)} ) \big\|_{L^2}.
\end{align}
Using Proposition \ref{PropNL2} and proceeding as for {\bf ii.}, this implies that 
\begin{align*}
\cQ_{m, 2}  \le \uCp\big[&\Vp\big(\ucQ_{m,2}+\ucU_{m+1,1}[t]\big)\\
&+\cQ_{m-1,2}+\cQ_m+\cU_{m+1,1} +\big \Vert  (g_0,g_1) \big\Vert_{L^{2,m}}\big].
\end{align*}
Using the result established in {\bf i.}, we deduce that
\begin{align*}
\cQ_{m, 2}  \le \uCp\big[&\Vp\big(\ucQ_{m,2}+\ucQ_{m+1}+\ucU_{m+1,1}[t]\big)\\
&+\cQ_{m-1,2}+\cQ_{m+2}+\cU_{m+1} +\big\Vert  f\big\Vert_{L^{2,m+1}_2}  + \big \Vert  (g_0,g_1) \big\Vert_{L^{2,m}}\big].
\end{align*}
By induction, and with \eqref{estinduc}, we finally obtain the result.
 \end{proof} 
 As said above, we need a control of $\Vp$ in terms of the energy norm $\Vert V\Vert_{\VV^{n-1}}$. The proposition below, also based on Hardy-type inequalities,  establishes the result used to prove Step 4 of the proof of Proposition \ref{propHO}.
 \begin{prop}\label{propcontinf}
 Let $T>0$ and $\uV, \cS$ satisfy the bounds \eqref{constMetc} and  let $p\geq 2$ and $p+5\leq n-1$.  There exists a nondecreasing function of its arguments $\cT(\cdot)$ such that if $T\cT(M_1,M_2,M)<1$, any smooth solution $V=(q,u)$  of \eqref{lineq2}   satisfies the following estimate, 
 $$
 \Vert V\Vert_{L^{\infty,p}_T}\leq C(T,M_1,M_2,M)\big[ C_0+\Vert V\Vert_{L^\infty_T\VV^{p+6}}+S\big],
 $$
 with $C_0=\Vert u(0)\Vert_{L^{\infty,p+2}}+\Vert q(0)\Vert_{L^{2,p-1}\cap L^{\infty,p-1}}$.
 \end{prop}
 \begin{proof}
The proof is decomposed into three lemmas. The first one gives a control of the sup norms of $u$ since they can be obtained directly; the second lemma provides some necessary controls on the source terms that allow the use of the Hardy type inequalities used in the third lemma to get upper bounds on the sup norms of $q$. {\it Throughout this proof, the time dependence in the norms is omitted when no confusion is possible.}
 \begin{lem}\label{propuinf}
 Let $T>0$ and $\uV,\cS$ satisfying the bounds \eqref{constMetc}. There is a constant $\uCp=C(T,M_1)$ such that for all smooth solution $V=(q,u)$  of \eqref{lineq2} on  $[0,T]$, 
  and for $m\geq p$ such that $m +2\le  2p$, one has for all $t\in [0,T]$,
\begin{align*}
\Vert u(t)\Vert_{L^{\infty,m}} & \le  \uCp \Big[\Vert q(0)\Vert_{L^{2,p-1}}+\Vpt\big(\ucQ_{m,2}(t)
+\ucQ_{m+2}(t)+\ucU_{m+2,1}[t] \big)\\
&+\Vert V(t)\Vert_{ \VV^{m+3}}+\Vert  f(t)\Vert_{L^{2,m+2}_2}+\big\| (g_0  , g_1 ) (t) \big\|_{L^{2,m+1}}\Big] +\Vert f(t)\Vert_{L^{2,m}}.
\end{align*}
 \end{lem}
 
\begin{proof}[Proof of the lemma]
From a classical Sobolev embedding and the equation \eqref{eqstep1} on $X^\alpha u$, we have
for all $\vert \alpha\vert \leq m$,
 \begin{align*}
 \big\| X^\alpha u(t) \big\|_{L^\infty}  &\lesssim  \big\| X^\alpha u(t) \big\|_{H^1} \\
& \lesssim   \big\| X^\alpha u(t)  \big\|_{L^2}  +  \big\|  c(\uq)X^\alpha \D_t q (t)  \big\|_{L^2} 
 +  \big\| f^{(\alpha)}   (t)  \big\|_{L^2} . 
 \end{align*}
The estimate of $f^{(\alpha)}$ given by Proposition~\ref{PropNL1}  implies that
\begin{equation}\label{eqst1}
\Vert u(t)\Vert_{L^{\infty,m}}  \le \cU_{m, 1}   + \uCp\big(\Vert V(t)\Vert_{L^{\infty,p}}\ucQ_{m,2}(t)+\cQ_{m+1, 2} (t)\big) +\Vert f(t)\Vert_{L^{2,m}};
\end{equation}
using the first and third point of Proposition \ref{PropHq} to  control  $ \cU_{m, 1}  $ and $\cQ_{m+1, 2} $  respectively, we get the result.
  \end{proof}
  The following lemma gives some bounds that will be used in the Hardy inequalities used to bound $q$ from above in $L^\infty$ norm. Note that the difference between the estimate on $\cQ_{p,2}$ in the second point of the lemma with respect with the third estimate of Proposition \ref{PropHq} is that the former only requires $\Vert u\Vert_{L^{\infty,p}}$ in the right-hand-side (instead of $\Vert V\Vert_{L^{\infty,p}}$).

  \begin{lem}\label{lem68}
 Let $T>0$ and $\uV,\cS$ satisfy the bounds \eqref{constMetc}. There is a constant $\uCp=C(T,M_1)$ such that for all smooth solution $V=(q,u)$  of \eqref{lineq2} on  $[0,T]$, the following properties hold for all $\vert \alpha\vert \leq p$. \\
  {\bf i.} The source terms are uniformly bounded on $ [0,T]$ in $L^2$ and $L^\infty$ norms,
\begin{align*}
 \begin{aligned} \Vert (g_0^{(\alpha)},g_1^{(\alpha)})(t)\Vert_{L^2} &\leq 
 \Vert (g_0,g_1)(t)\Vert_{L^{2,p}}
 \\
+  & \uCp \big( \cU_p(t)+\cQ_{p+1}(t)+\big(\Vert q(0)\Vert_{L^{2,p-1}}+t\Vert q\Vert_{L_t^{\infty,p}}\big)\Vert \uu \Vert_{L_t^{\infty,p+1}}\big), 
 \end{aligned}
 \\
 \begin{aligned} \Vert (g_0^{(\alpha)},g_1^{(\alpha)})\Vert_{L^\infty}&\leq \Vert (g_0,g_1)(t)\Vert_{L^{\infty,p}}
 \\
+ &  \uCp \big( \Vert u(t)\Vert_{L^{\infty,p+1}}+\big(\Vert q(0)\Vert_{L^{\infty,p-1}}+t\Vert q\Vert_{L_t^{\infty,p}}\big)\Vert \uu\Vert_{L_t^{\infty,p+1}}\big). 
\end{aligned}
\end{align*}
  {\bf ii.} If  $p\geq 2$, one has the following estimate
 $$ \begin{aligned}
\cQ_{p,2}(t)\leq 
& \uCp \Big[ \Vert (g_0,g_1)(t)\Vert_{L^{2,p}}+\Vert f(t)\Vert_{L^{2,p+1}_2}  +\Vert V(t)\Vert_{\VV^{p+1}}\\
 &  +\big(\Vert q(0)\Vert_{L^{2,p-1}}+\Vert \D_t q (0)\Vert_{L^\infty}+t\Vert q\Vert_{L_t^{\infty,p}}\big)\big(\Vert \uu(t)\Vert_{L^{\infty,p+1}}+\ucQ_{p+1}(t)\big)\Big].
\end{aligned}
$$
  \end{lem}
  
  \begin{proof}[Proof of the lemma] {\bf i.} For $\vert \alpha\vert \leq p$, we can use the expression for $g_0^{(\alpha)}$ and $g_1^{(\alpha)}$ given in the proof of Proposition \ref{PropNL2} to obtain using direct estimates in $L^2$,
  $$
  \Vert (g_0^{(\alpha)},g_0^{(\alpha)}) (t)\Vert_{L^2}\leq \uCp \big( \cU_{p,1}(t)+  \Vert \uu\Vert_{L_t^{\infty,p+1}}  \cQ_{p-1,2}(t)\big).
  $$
  Together with \eqref{Hqsimp} and \eqref{estinduc}, this implies the result.\\
For the $L^\infty$-estimate, we proceed similarly, but taking the sup norm  instead of the $L^2$-norm,
 $$
  \Vert (g_0^{(\alpha)},g_1^{(\alpha)})\Vert_{L^\infty}\leq \uCp \Big( \Vert u\Vert_{L^{\infty,p+1}}+\Vert \uu\Vert_{L_t^{\infty,p+1}}\Vert q\Vert_{L^{\infty,p-1}}\Big)+\Vert (g_0,g_1)\Vert_{L^{\infty,p}}.
  $$
  Since $\Vert q\Vert_{L^\infty_tL^{\infty,p-1}}\leq \Vert q(0)\Vert_{L^{\infty,p-1}}+t \Vert q\Vert_{L^\infty_tL^{\infty,p}}$, this implies the result.\\
{\bf ii.} For the second point, we get from \eqref{estL2g0} that
$$
\cQ_{p,2}\leq \uCp \big(\cQ_p+\cU_{p+1,1}\big)+\sum_{\vert \alpha\vert\leq p}\Vert (g_0^{(\alpha)},g_1^{(\alpha)})\Vert_{L^2}
$$
and we can use the first point of the lemma to get
$$
\cQ_{p,2}\leq \uCp \big(\cQ_{p+1}+\cU_{p+1,1}+(\Vert q(0)\Vert_{L^{2,p-1}}+t\Vert q\Vert_{L_t^{\infty,p}})\Vert \uu\Vert_{L^{\infty,p+1}}\big)+\Vert (g_0,g_1)\Vert_{L^{2,p}}.
$$
Together with \eqref{Hqsimp2}, and observing that $\Vert X_1 q\Vert_{L_t^\infty}\leq \Vert X_1 q(0)\Vert_{L^\infty}+t \Vert q\Vert_{L_t^{\infty,p}}$ if $p\geq 2$, this implies the result.
    \end{proof}

A control of the  $L^\infty$ norms of $q$ and of its derivatives is then provided by the following lemma.  

\begin{lem}\label{lem55}
Let $T>0$ and $\uV,\cS$ satisfy the bounds \eqref{constMetc}. Any smooth solution $V=(q,u)$  of \eqref{lineq2}  satisfies the following estimate provided that  $p\geq 2$ and $p+5\leq n-1$,
$$
\big\| q   \big\|_{L_t^{\infty,p}}   \leq C(T,M_1,M_2,M)\big( C_0+\Vert V\Vert_{L^\infty_t\VV^{p+6}}+S\big),
$$
with $C_0=\Vert u(0)\Vert_{L^{\infty,p+2}}+\Vert q(0)\Vert_{L^{2,p-1}\cap L^{\infty,p-1}}$.
 \end{lem} 

\begin{proof}
Again we use the equation \eqref{Hardyq} on $q^{(\alpha)}$, with $q^{(\alpha)}$ as defined in \eqref{defqalpha}, for $\vert \alpha\vert \leq p$. Using Corollary \ref{coroH3} (with $\alpha=2$ and $\sigma=0$) we see that 
$q^{(\alpha)}$ satisfies
$$
 \big\| q^{(\alpha)}   \big\|_{L^\infty} \lesssim   \big\| q^{(\alpha)}   \big\|_{L^2} 
 +  \big\|  g_{0}^{(\alpha)}- X_1 X^\alpha u+\mu  \uq \underline{B}' X_2 X^\alpha X_1 u-\mu (\underline{B}')^2 X^\alpha X_1 u   \big\|_{L^2\cap L^\infty} 
$$
from which one readily gets using the definition \eqref{defqalpha} of $q^{(\alpha)}$ that
$$
\Vert q\Vert_{L^{\infty,p}}\leq \uCp \big( \cQ_{p,2}+\cU_{p+1,1}+\Vert u\Vert_{L^{\infty,p+2}}+\sum_{\vert\alpha\vert \leq p}\Vert g_0^{(\alpha)}\Vert_{L^2\cap L^\infty}\big).
$$
We now use Lemma \ref{lem68}{\bf .ii} to control $\cQ_{p,2}$, \eqref{Hqsimp2} for $\cU_{p+1,1}$ and Lemma \ref{lem68}{\bf .i} for $\Vert g_0^{(\alpha)}\Vert_{L^2\cap L^\infty}$; this yields 
\begin{align*}
\big\| q   \big\|_{L^{\infty,p}} \leq \uCp\Big[ & \Vert f\Vert_{L^{2,p+1}_2} +\Vert (g_0,g_1)\Vert_{L^{2,p}\cap L^{\infty,p}} + \Vert V\Vert_{\VV^{p+1}}+\Vert u\Vert_{L^{\infty,p+2}}
\\ &  +\big(\Vert q(0)\Vert_{L^{2,p-1}\cap L^{\infty,p-1}} +t\Vert q\Vert_{L_t^{\infty,p}}\big)  \big(\Vert \uu\Vert_{L^{\infty,p+1}}+\ucQ_{p+1}\big) 
 \Big].
 \end{align*}
 Since $\Vert u\Vert_{L^{\infty,p+2}}\leq \Vert u(0)\Vert_{L^{\infty,p+2}}+t \Vert u\Vert_{L_t^{\infty,p+3}}$, we can use Lemma \ref{propuinf}, \eqref{constMetc} and the fact that $p+5\leq n-1$ to obtain the result.
\end{proof} 

By time integration, one directly gets with Lemma \ref{propuinf} and the bounds \eqref{constMetc}  that
   \begin{align*}
\Vert u \Vert_{L_t^{\infty,p}}&\leq  \Vert u (0)\Vert_{L_t^{\infty,p}} +t\Vert u \Vert_{L^\infty_tL^{\infty,p+1}}\\
&\leq \Vert u (0)\Vert_{L_t^{\infty,p}}+t\uCp \big[ \Vert q(0)\Vert_{L^{2,p-1}}+ \Vert V\Vert_{L^\infty_t\VV^{p+4}}+\Vert V\Vert_{L_t^{\infty,p}} M_2 +S\big].
\end{align*}
Adding up with the estimate coming from Lemma \ref{lem55}, this implies
$$
\Vert V\Vert_{L_t^{\infty,p}}\leq C(T,M_1,M_2,M) \big[ C_0+ \Vert V\Vert_{L^\infty_t\VV^{p+6}}+t\Vert V\Vert_{L_t^{\infty,p}}  +S\big];
$$
for $T$ small enough, the term involving $\Vert V\Vert_{L^{\infty,p}}$ in the right-hand-side can be absorbed in the left-hand-side and the result follows.
\end{proof}
The fifth and last step of the proof of Proposition \ref{propHO} does not require any additional technical result, so that the proof is complete.

\subsection{Proof of Proposition \ref{propbounds}}\label{sectproofbounds}

The goal of this section is to prove the uniform bounds \eqref{induc} on the sequence $(V^k,V_0^k,V_1^k,V_2^k)_k$ constructed through the iterative scheme \eqref{scheme}, namely,
\begin{equation}\label{schemebis}
 \begin{cases}
 \cL_{a}(V^k,\D)V_m^{k+1}&=\cS_m(V^k,V_1^k,V_2^k)\qquad (m=1,2),\\
 \cL (V^k,\D)V_0^{k+1}&=\cS(V^k,V_1^k,V_2^k),\\
 {\bf E}(\D)V^{k+1}&=V^k_1 + {\bf F}_2 V^k_2 + {\bf F}_0 V^k_0. 
 \end{cases}
\end{equation}
More precisely, we want to show that there exists constants $M_1$, $M_2$, $M$, $\widetilde{M}$, $N_1$, and $N_2$ such that
 \begin{equation}\label{inducbis}
 \begin{cases}
   {\mathfrak m}(V^k;T) \leq M,&\\
  \tilde {\mathfrak m}(V^k;T) \leq \widetilde{M},&\\
 {\mathfrak m}_1(V^k;T)\leq M_1, &\quad\mbox{and}\quad  \Vert V_m^{k}\Vert_{L^{\infty,p}} \leq N_1\quad \,\, (m=0,1,2),\\
{\mathfrak m}_2(V^k;T)\leq M_2& \quad \mbox{and}\quad  {\mathfrak m}_2(V_m^k;T)\leq N_2\quad (m=0,1,2),
  \end{cases}
 \end{equation}
 and a constant $S$ such that
 \begin{equation}
 \label{inducquar}
  {\mathfrak s}(\cS^{k};T)\leq S,\quad \mbox{and}\quad{\mathfrak s}(\cS_m^{k};T)\leq S,
  \end{equation}
for $m=1,2$ and $k\geq 1$.

We shall prove  by induction that \eqref{inducbis} holds for all $k$. The case $k=1$ has been defined in \eqref{itere1} and Proposition \ref{propvapp}; we focus therefore our attention on the proof of \eqref{inducbis}$_{k+1}$ assuming that \eqref{inducbis}$_k$ is known. The desired bounds on $V^{k+1}$ and $V_m^{k+1}$ ($m=0,1,2$) are established in the following lemma using the higher order estimates of Proposition \ref{propHO} with source terms $\cS_j^k$ and $\cS^k$ (we recall that $\delta>0$ that appears in the statement comes from the positivity condition \eqref{inita}).
 \begin{lem}\label{lemMM}
Assume that $[V^k]=(V_1^k,V_2^k,V_0^k, V^k)$ and $\cS^k$, $\cS_1^k$ and $\cS_2^k$ satisfy the induction assumptions \eqref{inducbis} 	and \eqref{inducquar} and that $V^{k+1}$ and $V_m^{k+1}$ solve \eqref{schemebis}. There exists a constant $C_0$ that depends only on the initial data and of the form
$$
C_0=C_0\big(\Vert V(0)\Vert_{\VV^n}, \Vert q(0)\Vert_{L^{2,n-1}_1\cap L^{2,p+4}\cap L^{\infty,p}},  \Vert u(0)\Vert_{L^{2,n-1}\cap L^{\infty,p+3}}\big)
$$
and some nondecreasing functions of their arguments $\cC_1$, $\cC_2$, $\cC$, $\widetilde{\cC}$ $\cC_1'$, $\cC'_2$ and $\cT(\cdot)$ such that if 
\begin{align*}
M_1\geq \cC_1(C_0,\frac{1}{\delta}),\qquad M_2\geq \cC_2(C_0),\qquad
M\geq\cC\big(T,C_0,M_1\big),\qquad \widetilde{M}\geq \widetilde{\cC}(C_0),\\
 N_1\geq \cC'_1\big(T,C_0,M_1,M_2,M, S\big) , \qquad
  N_2\geq \cC'_2\big(T,C_0,M_1,M_2,M, N_1, S\big)
\end{align*}
and $T>0$ is small enough to have
$$
T\cT\big(M_1,M_2,M, \widetilde{M},N_1,N_2,S\big)<1
$$
then $V^{k+1}$ and $V_m^{k+1}$ also satisfy \eqref{inducbis}.
\end{lem}
\begin{proof}[Proof of the lemma]
From the first equation of \eqref{schemebis} and Proposition \ref{propHO} we get that for $m=1,2$,
$$
\Vert V_m^{k+1}\Vert_{L^\infty_T\VV^{n-1}} \leq C(T,M_1) \big[C_{0}+\sqrt{T}C(T,M_1,M_2,M,\widetilde{M}) S  \big],
$$
where $C(\cdot)$ always denotes a smooth, non decreasing function of its arguments and $C_0$ is a in the statement of the lemma.
 If $M>C(T,M_1) C_{0}$, there exists $T>0$ small enough such that 
\begin{equation}\label{estinterme}
\Vert V_m^{k+1}\Vert_{L^\infty_T\VV^{n-1}} \leq M.
\end{equation}
 We can therefore deduce from Proposition \ref{propcontinf} that
$$
\Vert V_m^{k+1} \Vert_{L^{\infty,p}_T } \leq C\big(T,M_1,M_2,M\big)\big[1+C_{0}+S\big];
$$
choosing $N_1$ larger than the right-hand-side yields the needed upper bound on $\Vert V_m^{k+1} \Vert_{L^{\infty,p}_T }$ for $m=1,2$. The case $m=0$ is treated similarly\footnote{The linear operator involved in the equation for $V_0^{k+1}$ is $\cL(V^k,\D)$ instead of $\cL_a(V^k,\D)$. The results proved on the latter obviously hold for the former by substituting $a\equiv 0$.}.\\
We now turn to give an upper bound on ${\mathfrak m}_1(V^{k+1};T)$. For this, we need:
\begin{itemize}
\item An upper bound on  $\Vert V^{k+1} \Vert_{L^\infty_T \VV^{n-1}} $. Using  Proposition \ref{propell} we get
$$
\Vert V^{k+1} \Vert_{L^\infty_T \VV^{n-1}}\leq \Vert V^{k+1} (0)\Vert_{\VV^{n-1}} +3\sqrt{T}N_1;
$$
if $M_1> \Vert V^{k+1} (0)\Vert_{\VV^{n-1}}$ and $T$ is chosen small enough, then $\Vert V^{k+1} \Vert_{L^\infty_T \VV^{n-1}}\leq M_1$.
\item An upper bound on $\Vert \frac{1}{1+\mu a(u^{k+1})}\Vert_{L^\infty_T}$. By assumption, $1+\mu a(u^{k+1})\leq \delta$ at $t=0$ so that
\begin{align*}
1+\mu a(u^{k+1})(t)\geq \delta -\mu T \Vert \D_t (a(u^{k+1})) \Vert_{L^{\infty}_T}.
\end{align*}
Using the bound on $\Vert V^{k+1} \Vert_{L^\infty_T \VV^{n-1}}$ derived in the previous point, we obtain that $\Vert \D_t (a(u^{k+1})) \Vert_{L^{\infty}_T}\leq C(M_1)$. Therefore, if $M_1>1/\delta$ and $T$ chosen small enough, we obtain that $\Vert \frac{1}{1+\mu a(u^{k+1})}\Vert_{L^\infty_T}\leq M_1$.
\item An upper bound on $\Vert \frac{1}{q^{k+1}}\Vert_{L^\infty_T}$. Proceeding as for the previous point, and taking into account that $q(0)=1/2$ we obtain that
$\Vert \frac{1}{q^{k+1}}\Vert_{L^\infty_T}\leq M_1$ is $M_1>2$ and $T$ small enough.
\end{itemize}
From this three points, we deduce that ${\mathfrak m}_1(V^{k+1};T)\leq M_1$.\\
Let us now turn to control the two components of ${\mathfrak m}(V^{k+1})$:
\begin{itemize}
\item  {\it Control of $\Vert V^{k+1}\Vert_{L^\infty_T\VV^{n-1}}$.} By \eqref{estinterme} and Proposition \ref{propell} one directly gets that $\Vert V^{k+1}\Vert_{L^\infty_T\VV^{n-1}}\leq M$ for $T$ small enough and $M>\Vert V^{k+1}(0)\Vert_{\VV^n}$.
\item  {\it Control of $\Vert u^{k+1} \Vert_{L^{\infty,p+1}}$}. Using the induction relation and the bounds on $\Vert V_m^{k+1} \Vert_{L^\infty_T}$ and $\Vert V^{k+1}\Vert_{L^\infty_T\VV^{n-1}}$ proved above, we get through Lemma \ref{propuinf} that for $m=0,1,2$,
$$
\Vert u_m^{k+1} (t)\Vert_{L^{\infty,p+1}}\leq C\big(T,M_1,M_2,M\big)\big[ 1+ C_{0}+S\big].
$$
With Proposition \ref{propell}, this implies that
$$
\Vert u^{k+1} (t)\Vert_{L^{\infty,p+1}}\leq \Vert u^{k+1} (0)\Vert_{L^{\infty,p+1}}+\sqrt{T}C\big(T,M,M_1,M_2,N\big)
$$
and therefore $\Vert u^{k+1} \Vert_{L_T^{\infty,p+1}}\leq M$ if $M>\Vert u^{k+1} (0)\Vert_{L^{\infty,p+1}}$ and $T$ small enough.
\end{itemize}
Gathering these two points we get that ${\mathfrak m}(V^{k+1};T)\leq M$.\\
The next step is therefore to derive the two estimates on ${\mathfrak m}_2(V_m^{k+1};T)$  and ${\mathfrak m}_2(V^{k+1};T)$. For the first one, we observe, owing to the induction assumption \eqref{induc}, Proposition \ref{PropHq} and the control on $\Vert V^{k+1}_m \Vert_{L^{\infty,p}_T}$ already proved, and \eqref{estinterme}  imply that
$$
{\mathfrak m}_2(V_m^{k+1};T)\leq C(M_1)\big[C_{0}+N_1(1+M_2+M) +M+S\big],
$$
so that we just have to choose $N_2$ larger than the right-hand-side. The estimate on ${\mathfrak m}_2(V^{k+1};T)$ then follows following the now usual procedure based on Proposition \ref{propell}.\\
Finally, the upper bound on $\tilde {\mathfrak m}(V^{k+1};T)$ stems from the upper bounds proved above on $ {\mathfrak m}(V^{k+1};T)$, ${\mathfrak m}_2(V^{k+1};T)$ and \eqref{estinterme} and to the elliptic regularization property given by Proposition \ref{propell}.
\end{proof}
The only thing left to prove is that the last inequality of \eqref{inducbis} holds at step $k+1$. This is done in the following lemma.
\begin{lem}\label{lemsource}
Assume that $[V^k]=(V_1^k,V_2^k,V_0^k, V^k)$ and $\cS^k$, $\cS_1^k$ and $\cS_2^k$ satisfy the induction assumptions \eqref{inducbis} 	and \eqref{inducquar} and that $V^{k+1}$ and $V_m^{k+1}$ solve \eqref{schemebis}.  Let also $C_0$ be as in Lemma \ref{lemMM}. There exist two smooth functions $\cC(\cdot)$  and $\cT(\cdot)$, with a nondecreasing dependence on their arguments, such that if 
$$
S\geq \cC_2\big(T,C_0,M_1,M_2,N_1,N_2\big) 
$$
and $T>0$ is small enough to have
$$
T\cT\big(M_1,M_2,N_1,N_2,S\big)<1,
$$
then $\cS$, $\cS_1$ and $\cS_2$ also satisfy \eqref{inducquar}.
\end{lem}
\begin{proof}[Proof of the lemma]
We prove here the required upper bounds on the different components of ${\mathfrak s}(\cS_m^{k+1})$ for $m=1,2$. The bounds for ${\mathfrak s}(\cS^{k+1})$ can be obtained similarly so that we omit the proof. \\
To alleviate the notations, we do not write the superscript $k+1$ throughout this proof, so that we write $\cF_m(q,q_1,q_2)$ instead of $\cF_m(q^{k+1},q^{k+1}_1,q^{k+1}_2)$, etc.\\
We recall that ${\mathfrak s}(\cdot)$ is defined through \eqref{defconstantes} and that $\cS_m$ is given by \eqref{defSm}; we therefore have to derive upper bounds for
$$
\Vert \cF^{(m)}(q,q_1)\Vert_{L^\infty_TL^{2,n-1}_2}, \qquad \Vert  \cF_m(q,q_1)\Vert_{L^\infty_TL^{2,p+3}},
$$
and
$$
\Vert (g_0,g_1)\Vert_{L^\infty_TL^{2,n-1}_1},\qquad\Vert g\Vert_{L^\infty_TL^{2,n-2}\cap L^{\infty,p}},
$$
with $g_0= \cG^{(m)}_0(V,V_1,V_2)$ and $g_1= \cG^{(m)}_1(V,V_1,V_2)$, and $\cG^{(m)}_0$ and $\cG^{(m)}_1$ as in Lemma \ref{lemmaLinear2}.\\
 Using the bounds proved in Lemma \ref{lemMM} and the control on  $\Vert q_m\Vert_{L^{2,n-1}_2} $ provided by \eqref{estinterme}, one readily gets 
$$
\Vert \cF^{(m)}(q,q_1)\Vert_{L^\infty_TL^{2,n-1}_2}+\Vert  \cF^{(m)}(q,q_1)\Vert_{L^\infty_TL^{2,p+3}}\leq C\big(T,M_1,M_2,N_1,N_2\big).
$$
For the bounds on $g_0$ and $g_1$, we first treat the topography term. Recalling that $\varphi_m=X_m x+\int_0^t u_m$ and using the control on $\Vert u_m\Vert_{L^{2,n-1}_1}$ provided by \eqref{estinterme}, we easily deduce from Lemma \ref{lemMM} that 
$$
\Vert B''(\varphi)\varphi_m\Vert_{L^\infty_T(L^{2,n-1}_1\cap L^{2,n-2}\cap L^{\infty,p})}
\leq  C\big(T,M_1,M_2,N_1,N_2\big).
$$
We are therefore left to give an upper bound on 
$$
\Vert  \cG^{(m)}_j(V,V_1,V_2)\Vert_{L^\infty_TL^{2,n-1}_1}, \quad \mbox{ and }\quad\Vert \cG^{(m)}_j(V,V_1,V_2)\Vert_{L^\infty_T(L^{2,n-2}\cap L^{\infty,p})},
$$
for $j=0,1$. We just treat the case $j=1$ here, the situation being similar for $j=0$. We recall that
\begin{align*}
\frac{1}{\sqrt{\mu}}\cG^{(m)}_1 (V,V_1,V_2)=&- \frac{8}{3} q q_m X_2 u_1- \frac{4}{3} q q_1 X_2 u_m-\frac{4}{3}q X_2 u X_1 q_m-\frac{4}{3}q_m X_2 u q_1\\
 &+u_1 q B''\varphi_m+q u_m u B''+qu^2 B^{(3)}\varphi_m.
 \end{align*}
 - Control of $\Vert \cG^{(m)}_j(V,V_1,V_2)\Vert_{L^\infty_TL^{2,n-1}_1}$. This stems from an upper bound on $\Vert  X^\alpha \cG_j^{(m)}\Vert_{L^{2}_1}$ for all $\vert \alpha\vert \leq n-1$. We just show how to handle two components of $\cG^{(m)}_1 $, namely, $q_m q X_2 u_1$ and   $X_2 u X_1 q_m q$ (the adaptation to the other components is straightforward).
 \begin{itemize}
\item[*] Control of $X^\alpha \big(q_m qX_2 u_1\big)$ with $\vert \alpha\vert \leq n-1$. We can develop this term into
$$
(h_0X^\alpha q_m) q \D_x u_1+\sum_{\alpha_1+\alpha_2+\alpha_3=\alpha,\vert \alpha_1 \vert <n-1}X^{\alpha_1}q_m X^{\alpha_2}q X_2 X^{\alpha_3} u_1
$$
from which we get, with $C_p(T)=C\big(T,\Vert V\Vert_{L^{\infty,p}_T},\Vert V_1\Vert_{L^{\infty,p}_T},\Vert V_2\Vert_{L^{\infty,p}_T}\big)$,
\begin{align*}
\Vert  q_m q X_2 u_1\Vert_{L^{2,n-1}_1}\lesssim& \Vert \D_x u_1\Vert_{L^\infty}\Vert q_m\Vert_{L^{2,n-1}_2}\\
&+C_p(T)\big(\Vert q_m\Vert_{L^{2,n-2}_1}+\cQ_{n-1,1}+\Vert u_1\Vert_{\cH^{1,n-1}_1}\big).
\end{align*}
Now, using the equation on $q_j$ ($j=1,2$), we have
\begin{equation}\label{dxuj}
\D_x u_j=-c(q)X_1 q_j+\cF_j(q,q_1,q_2)
\end{equation}
with $\cF_j(q,q_1,q_2)$  given by Lemma \ref{lemmaLinear1}; with Lemma \ref{lemMM}, this implies that $\Vert \D_x u_j\Vert_{L^\infty}\leq C(M_1)$ and therefore, using  \eqref{estinterme},
$$
\Vert  q_m q X_2 u_1\Vert_{L^\infty_TL^{2,n-1}_1}\lesssim C\big(T,M_1,M_2,N_1,N_2\big).
$$
\item[*] Control of $X^\alpha \big(X_2 u X_1 q_m q\big)$ with $\vert \alpha\vert \leq n-1$. We first develop this expression into
\begin{align*}
& (h_0 X^\alpha X_1 q_m) D_x u q+ \sum_{\alpha_1+\alpha_2+\alpha_3=\alpha, \vert \alpha_1\vert=n-2} (h_0X^{\alpha_1}X_1 q_m) \D_x X^{\alpha_2}u X^{\alpha_3}q\\
&\qquad+
\sum_{\alpha_1+\alpha_2+\alpha_3=\alpha, \vert \alpha_1\vert\leq n-3}  X^{\alpha_1}X_1 q_m X_2X^{\alpha_2}u X^{\alpha_3}q.
\end{align*}
 To handle the first term of this expression, we use  the equation on $q_m$ to replace $X_1 q_m$ by
\begin{equation}\label{eqX1qm}
X_1 q_m=\frac{1}{c(q)}\big(-\D_x u_m+\cF_m(q,q_1,q_2)\big),
\end{equation}
and then proceed as above. For the first summation, a control is easily obtained in terms of $\Vert \D_x X^{\alpha_2}u \Vert_{L^\infty}$. This quantity can in turn be controlled in terms of $\Vert V\Vert_{L^{\infty,p}}$ and $\Vert V_j\Vert_{L^{\infty,p}}$ ($j=1,2$) through \eqref{dxuj}. The second summation does not raise any particular difficulty.
Finally, one gets
$$
\Vert  X_2 u X_1 q_m q\Vert_{L^\infty_TL^{2,n-1}_1}\leq C\big(T,M_1,M_2,N_1,N_2\big).$$
\end{itemize}
Gathering all these elements, we deduce that
$$
\Vert  \cG^{(m)}_j(V,V_1,V_2)\Vert_{L^\infty_TL^{2,n-1}_1}\leq C\big(T,M_1,M_2,N_1,N_2\big).
$$
- Control of  $\Vert \cG^{(m)}_j(V,V_1,V_2)\Vert_{L^\infty_TL^{2,n-2}}$. Most of the terms can be treated with a slight adaptation of the above, 
The only significant change \footnote{We previously used the estimate
$$
\Vert  X_2X^\alpha u_j\Vert_{L^2_1}\leq \Vert u_j\Vert_{\cH^{1,n-1}_1} \quad \mbox{ for }\quad \vert \alpha\vert \leq n-1.
$$
 One should be careful that without the $h_0$ weight, one cannot use the extra control given by the $\cH^1_1$ based estimate, and one only has
$$
\Vert  X_2X^\alpha u_j \Vert_{L^2} \leq \Vert u_j\Vert_{L^{2,n-1}}\quad \mbox{ for }\quad \vert \alpha\vert \leq n-2
$$
(and not  $\Vert u_j\Vert_{L^{2,n-2}}$), so that there is an extra term to control.
} needed is for the control of the $L^2$-norm of terms of the form $X^\alpha X_2 u_j$ (with $\vert \alpha\vert\leq n-2$). One needs to use the equation on $q_j$  to write
$$
X^\alpha X_2 u_j=X^\alpha (-c h_0 X_1 q_j-c' h_0 q_1 q_j -cX_j h_0 q_1)
$$
so that
$$
\forall \vert \alpha\vert \leq n-2,\qquad \Vert X^\alpha X_2 u_j\Vert_{L^2}\leq C_{p}(T)\big(\cQ_{n-2}+ \Vert (q_1,q_2)\Vert_{L^{2,n-1}_2}\big)
$$
and one easily deduces that
$$
\Vert  \cG^{(m)}_j(V,V_1,V_2)\Vert_{L^\infty_TL^{2,n-2}}\leq C\big(T,M_1,M_2,N_1,N_2\big).
$$
- Control of $\Vert  \cG^{(m)}_j(V,V_1,V_2)\Vert_{L^{\infty,p}_T}$.  One readily checks that all the components of $X^\alpha \cG_j^{(m)}$ can be directly bounded from above in $L^\infty$ by 
$$
C_{p}(T)(1+\Vert (u,u_1,u_2)\Vert_{L^\infty_TL^{\infty,p+1}})
$$
 if $\vert \alpha\vert\leq p$, except $X_2 u X^\alpha X_1 q_m q$ for which one needs to use the substitution \eqref{eqX1qm} to get the same upper bound. We have therefore
 $$
\Vert  \cG^{(m)}_j(V,V_1,V_2)\Vert_{L^\infty_TL^{\infty,p}}\leq C_{p}(T)(1+\Vert (u,u_1,u_2)\Vert_{L^\infty_TL^{\infty,p+1}}),
$$
and we therefore need un upper bound on $\Vert (u,u_1,u_2)\Vert_{L^\infty_TL^{\infty,p+1}}$. Remarking that $u=u^0+\int_0^t \D_t u$ (and similar expressions for $u_1$ and $u_2$), we get that
$$
\Vert (u,u_1,u_2)\Vert_{L^\infty_TL^{\infty,p+1}}\leq C_0+T \Vert (u,u_1,u_2)\Vert_{L^\infty_TL^{\infty,p+2}}
$$
so that using Lemma \ref{propuinf} and Lemma \ref{lemMM}, we obtain that
$$
\Vert (u,u_1,u_2)\Vert_{L^\infty_TL^{\infty,p+1}}\leq C_0+ T C\big(M_1,M_2,N_1,N_2,S\big).
$$
Plugging this into the above estimate for $\Vert  \cG^{(m)}_j(V,V_1,V_2)\Vert_{L^\infty_TL^{\infty,p}}$, and using Lemma \ref{lemMM} once again, we get
 \begin{equation}\label{estLinf}
\Vert  \cG^{(m)}_j(V,V_1,V_2)\Vert_{L^\infty_TL^{\infty,p}}\leq C(T,M_1,N_1)\big[1+C_0+ T C\big(M_1,M_2,N_1,N_2,,S\big)\big].
\end{equation}
- Conclusion. Taking $S$ large enough (in terms of $M_1$, $M_2$, $N_1$, $N_2$ and $S$), all the components of ${\mathfrak s}(\sS_m)$ are bounded from above by $S$ except $\Vert  \cG_m^j(V,V_1,V_2)\Vert_{L^\infty_TL^{\infty,p}}$. For this term, we need further to choose $S$ large enough to have $S> C_2(M_1,N_1)(1+C_0)$; we can then use \eqref{estLinf} to get the needed control, provided that  $T>0$ is taken small enough. This completes the proof.
\end{proof}
Using Lemmas \ref{lemMM} and \ref{lemsource}, the result stated in Proposition \ref{propbounds} follows by a simple induction.


  \section{The elliptic equation}
 \label{Sell} 

As seen in \S \ref{sectiter}, it is necessary to introduce an additional elliptic equation in order to regain space and time regularity with respect to the regularity of the quasilinearized variables. To our knowledge, there is no existing theory of elliptic equations on the half line in degenerate weighted spaces; since this theory can be of independent interest, we present it here in a specific section.\\
We first consider the analysis on the whole line in \S \ref{sectline} and then transport them on the half-line in \S \ref{sectHline} using a change of variables that transform $\D_y$ into the conormal derivative $h_0\D_x$. The proof of Proposition \ref{propell} is a direct application of these results and is provided in \S \ref{sectppp}.

  \subsection{The equation on the full line} \label{sectline}

  We consider, for $t \ge 0$,  a general elliptic problem (in space and time) of the form
  \begin{equation}
  \label{mellipeq}
  \D_t u + P (D_y)  u =  Q(D_y) f, \qquad u_{| t = 0} = u^0 ,
  \end{equation}
  where the  operators  $P= p(D_y)$ and $Q(D_y) $ are Fourier multipliers of   symbol $p(\eta)$ and $q(\eta)$ satisfying 
  \begin{equation}
\label{symbel}
 \frac{1}{C} \langle \eta\rangle \le   \re p(\eta) \le | p (\eta) | \le C  \langle \eta\rangle
  \end{equation}
 and for  all $k$     
    \begin{equation}
    \label{symb}
   | \D_\eta^k p (\eta) | \lesssim   \langle \eta\rangle^{ 1-k} , \qquad  | \D_\eta^k q (\eta) | \lesssim   \langle \eta\rangle^{  -k} ,
  \end{equation}
where  $ \langle \eta\rangle = (1 + | \eta|^2)^\mez$. 
  
  We consider \eqref{mellipeq} as an elliptic boundary value problem on $[0, T] \times \RR$, with one boundary condition on $\{  t = 0\}$ and no boundary condition on $\{ t = T \}$. 
  The solution is given by 
  \begin{equation}
  \label{solmellipeq}
  \hat u (t, \eta ) = e^{ -t p(\eta) } \hat u^0 (\eta) + \int_0^t  e^{ (t'-t)p (\eta) } q(\eta)  \hat f(t', \eta) dt' 
  \end{equation}
  where the symbol $\, \widehat{ \, } \,  $  denotes the Fourier transform in $y$. In particular, 
  \begin{equation*}
 |  \hat u (t, \eta ) | \le  e^{ -t \langle \eta\rangle /C  } | \hat u^0 (\eta) | +  C 
 \int_0^t  e^{ - (t-t' )  \langle \eta\rangle /C  } | \hat f(t', \eta) | dt' .  
  \end{equation*}
 This implies the estimates
 \begin{equation*}
\|   u(t) \|_{L^2}   \lesssim  \|   u(0) \|_{L^2}  + \int_0^t \|   f(t') \|_{L^2} dt'  ,  
\end{equation*} 
and also 
 \begin{equation*}
\|   u \|_{H^1([0, T] \times \RR) }  \lesssim  \|   u(0) \|_{H^\mez(\RR) }  +  \|   f \|_{L^2([0, T] \times \RR)} .  
\end{equation*} 
 Commuting with derivatives, one obtains the following elliptic estimates.
  \begin{lem}
  Let $T>0$ and $k\in \NN$. If  $u^0 \in H^{k+\mez}(\RR)$
  and $f \in H^k ([0, T] \times \RR)$, then \eqref{mellipeq} has a unique solution 
  $u \in H^{k+1} ([0, T] \times \RR)$ and 
  \begin{equation}
\|  u \|_{H^{k+1}([0, T] \times \RR) }   \lesssim \| u^0 \|_{H^{k+1}} +  \|  f\|_{H^k([0, T] \times \RR) } .
\end{equation}
Moreover, for $| \alpha |  \le k$: 
 \begin{equation}
\|   \D_{t,y}^\alpha u(t) \|_{L^2}   \lesssim  \|  \D_{t,y}^\alpha  u(0) \|_{L^2}  + \int_0^t \|  \D_{t,y}^\alpha  f(t') \|_{L^2} dt'  .   
\end{equation} 

\end{lem}

There is no elliptic regularization in $L^\infty$-based spaces. Instead, we will use that the contribution of the source term is small for small times\footnote{The estimate of the lemma is not optimal, since one could replace the exponent $1/2$ of $T$ by any  $\theta < 1$, but it is sufficient for our purpose. }.

\begin{lem}
 Let $T>0$ and $k\in \NN$. For $f \in L^{ \infty} ([0, T] \times \RR) $ and $u^0 \in L^{\infty}(\RR)$,  the solution $u$  of \eqref{mellipeq} 
 belongs to $ L^{\infty} ([0, T] \times \RR) $ and 
 \begin{equation*}
 \| u  \|_{L^{ \infty}([0,T] \times \RR)}
  \lesssim  \| u^0 \|_{L^{\infty}(\RR)}   + \sqrt T\,   \| f \|_{L^{ \infty}([0,T] \times \RR)} . 
 \end{equation*}
\end{lem}

\begin{proof}
 The formula \eqref{solmellipeq} can be written using the 
 semi group $e^{ - t P }$: 
  \begin{equation}
  \label{solmellipeq2} 
  u (t) = e^{ -t P} u^0 + \int_0^t  e^{ (t'-t)P} Q(D_y)  f(t') dt' . 
  \end{equation}
     $e^{ - t P}$ is 
a convolution operator  on the real line with kernel 
 \begin{equation*}
 \Phi (t, y) = \frac{1}{2 \pi} \int_\RR e^{ - t p (\eta) } e^{ i y \eta} d \eta. 
 \end{equation*}
By Plancherel's theorem and \eqref{symb}, 
one has 
\begin{equation*}
\| \Phi (t, \, \cdot\,) \|_{L^2} \lesssim t^{- \mez} , \quad \| y  \Phi (t, \, \cdot\,) \|_{L^2} \lesssim t^{ \mez} ,
\end{equation*}
implying that 
\begin{equation}
\label{L1est}
\| \Phi (t, \cdot) \|_{L^1} \lesssim 1,
\end{equation}
and thus 
$$
\| e^{ - t P} u^0 \|_{L^\infty} \lesssim \| u^0 \|_{L^\infty}. 
$$

 The analysis of the second term  is more delicate when $Q $ is not a constant, since $Q(D_y)$ does not necessarily act in $L^\infty$. However, 
 the convolution kernel of $e^{ - t P } Q$ is 
  \begin{equation}
 \Psi (t, y) = \frac{1}{2 \pi} \int_\RR e^{ - t p (\eta) } e^{ i y \eta} q(\eta)  d \eta. 
 \end{equation}
 By Plancherel's theorem and \eqref{symb}, 
one has 
\begin{equation*}
\| \Psi (t, \, \cdot\,) \|_{L^2} \lesssim t^{- \mez} , \quad \| y  \Psi (t, \, \cdot\,) \|_{L^2} \lesssim (\sqrt{t} + 1)   
\end{equation*}
implying that for $t \le 1$, one has 
$$
\| \Psi (t, \cdot) \|_{L^1} \lesssim t^{ - \mez} . 
$$
Therefore, the  $L^\infty$ norm of the second term in the right hand side of \eqref{solmellipeq2} is dominated by
$$
\int_0^t (t - t')^\mez \| f(t') \|_{L^\infty} dt' 
$$
and the lemma follows. 
\end{proof}

\begin{rem}
\textup{ Since the equation commutes with $\D_t$ and $\D_y$, there are similar estimates for  derivatives}. 
\end{rem}

We apply the results above to 
 the operator 
  \begin{equation}
\label{ex1}
  P  = (\kappa^2 - \D_y^2 ) ^\mez, \quad Q =  \alpha + (\beta_1 \D_y + \beta_0) P^{ - 1} 
  \end{equation}
with  $\kappa  > 0$, and also to the conjugated operators 
$$
  P_{ \delta}  = e^{  y \delta} P  e^{ -y \delta} , \quad   Q_{ \delta}  = e^{  y \delta}  Q e^{ -y \delta} 
 $$
  for $| \delta | < \kappa $. 
  Their symbols are $p_{\delta} (\eta) = \big(( \eta + i \delta)^2 + \kappa^2 \big)^\mez$ and 
  $q_\delta (\eta) =  \alpha +  (i \beta_1 \eta + \beta_0 + i \delta \beta_1)  \big(( \eta + i \delta)^2 + \kappa^2 \big) ^{ - \mez} $ and satisfy the conditions \eqref{symbel} and \eqref{symb}. 
 
Introduce the spaces ${\bf H}^k_s([0, T] \times \RR)$ of functions  $u$ such that $e^{ \frac{s}{2} y} u \in H^k([0, T] \times \RR)$   and 
${\bf H}^k_{s, s'} ([0, T] \times \RR) = {\bf H}^k_s([0, T] \times \RR) + {\bf H}^k_{s'}([0, T] \times \RR)$. 
There are similar definitions of spaces 
${\bf H}^k_s(\RR)$ and ${\bf H}^k_{s, s'} ( \RR)$ for the initial data. 
 \begin{prop} 
  \label{propell75} 
  Assume that  $| s | < \kappa $ and $| s' | < \kappa $. 
  Let $T>0$ and ${ k -1}\in \NN$.  Then, for all   $f \in {\bf H}^k_{s, s'} ([0, T] \times \RR)$ and 
  $u^0 \in {\bf H}^{k+1}_{s, s'} ( \RR)$,  the problem \eqref{mellipeq} has a unique solution 
  in ${\bf H}^{k+1}_{s, s'} ([0, T] \times \RR)$ which satisfies
    \begin{equation*}
\|  u \|_{{\bf H}^{k+1}_{s, s'} ([0, T] \times \RR) }   \lesssim \|  u^0 \|_{{\bf H}^{k+1}_{s, s'} ( \RR) }   +  \|  f \|_{{\bf H}^k_{s,s'} ([0, T] \times \RR) } . 
\end{equation*}
Moreover,  
\begin{equation*}
\| u(t) \|_{{\bf H}^k_{s,s'}}   \lesssim  \|   u(0) \|_{{\bf H}^k_{s,s'} }  + 
\int_0^t \|  f(t') \|_{{\bf H}^k_{s,s'} } dt'  . 
\end{equation*} 
If $f \in W^{p, \infty} ([0, T] \times \RR)$, then $u \in W^{p, \infty} ([0, T] \times \RR)$ and 
 \begin{equation*} 
 \| u  \|_{W^{p, \infty}([0,T] \times \RR)}
  \lesssim  \| u^0 \|_{W^{p, \infty}(\RR)}   + \sqrt T\,   \| f \|_{W^{p, \infty}([0,T] \times \RR)} . 
 \end{equation*}
  \end{prop}
  \begin{proof}
  The first  lemma above applied to $P_{\frac{s}{2}}$ implies the existence together with estimates in spaces
  ${\bf H}^k_s$ for $| s | < \kappa$, and therefore in spaces ${\bf H}^k_{s, s'}$. 
  The uniqueness in the space of temperate distributions in $y$ 
 is clear  by Fourier transform.   The third estimate is a direct application of   the second  lemma above and the remark 
 which follows it. 
\end{proof}

 
 \subsection{The equation on the half line} \label{sectHline}
 
 On $\RR_+ = ( 0, + \infty[$, we consider the operator
 $ X =  h_0  \D_x$ with $h_0 > 0$  as smooth as needed and  such that $h_0\approx x$ near the origin and $h_0 \approx 1$ at infinity.
 We transport the results of the previous section to the half line using the change 
 of variables $y = \chi(x)  $  \begin{equation}
x \ \mapsto  \ y =  \chi (x) = \int_1^x  d x' / h_0  (x')   
 \end{equation} 
 which transforms
 $ X = h  \D_x $ into $\D_y$.   
 We note that 
 \begin{equation}
 \label{poids}
 y \sim  \ln x \quad \mathrm{for} \ x \le 1, \qquad 
  y \sim    x \quad \mathrm{for} \ x \ge 1. 
 \end{equation}
 
As in \eqref{wL2}, we consider the spaces $L^2_s(\RR_+)$ of functions $u$ such that 
 $h_0^{s/2} u \in L^2 (\RR_+)$.  The mapping 
 \begin{equation}
 u \ \mapsto \ v = u \circ \chi^{-1} 
 \end{equation} 
 is an isometry from $L^2_{-1} (\RR_+) $ onto $L^2(\RR)$. Moreover, for $s \ge 0$, it is 
   an isomorphism from $L^2_s(\RR_+)$ onto  $L^2_{ s+1, 0} (\RR) $, 
 since  $u \in L^2_s(\RR_+)$ if and only if  $e^{y (s + 1)/2} v \in L^2$ on $\{ y \le 0 \} $ and 
 $v \in L^2$ on $\{ y \ge 0\}$. \\
 Similarly, it is an isomorphism from 
 $\cH^k_s(\RR_+)$ onto  ${\bf H}^k_{ s+1, 0} (\RR) $ where we recall that 
 $\cH^k_s(\RR_+)$ is the space of functions $u$ such that 
 $h_0^{s/2} X^j u \in L^2$ for $j \le k$.

 The operators corresponding to  $(\kappa^2 - \D_y^2 )^{1/2}$  and 
$ \alpha + ( \beta_1 \D_y + \beta_0)  (\kappa^2 - \D_y^2) ^{ - \mez} $  are 
 \begin{equation} 
 {\bf P } = ( \kappa^2 - (h_0 \D_x)^2 ) ^\mez , \quad
 {\bf Q}  =  \alpha + ( \beta_1 X  + \beta_0 )    {\bf P} ^{ - 1} . 
 \end{equation} 
Therefore, the results of Proposition~\ref{propell75} are immediately transported to  the equation 
   \begin{equation}
  \label{vellipeq}
  \D_t u + {\bf P }  u = {\bf Q }  f, \qquad u_{| t = 0} = u^0   . 
  \end{equation}
The next proposition summarizes the results using the notations introduced in Section~\ref{sectnot}. 
 
   \begin{prop}
   \label{propell76}
Suppose that $\kappa \ge 3 $  and $0 \le s \le 2 $.  

i)   If
  $f \in L^2_T  \cH^{k}_s $ and $u^0 \in \cH^k_s$, then  \eqref{vellipeq}  has a unique solution 
 $u \in C^{0}\cH^k_s$ 
 and 
$$
\|  u(t) \|_{\cH^k_s}   \lesssim  \|     u(0) \|_{\cH^{k}_{s }}  + 
\int_0^t \|   f(t') \|_{\cH^k_s } dt'  . 
$$

ii)   If  $u^0 \in \cH^{k+1}_s$, then 
  $u \in L^2_T \cH^{k+1}_s $ and 
  $$
  \| u \|_{L^2_T \cH^{k+1}_s} \lesssim   \| u (0)  \|_{ \cH^{k+1}_s}   +    \| f  \|_{L^2_T \cH^{k}_s} 
  $$
  
 iii.)  If   $f \in L^{\infty, p} _T $ and $u^0 \in L^{\infty, p}$ the solution $u$  of \eqref{vellipeq} 
 belongs to $ L^{\infty, p}_T$ and 
$$ 
\| u   \|_{L^{\infty, p} _T} \lesssim  \| u(0)  \|_{L^{\infty, p} }   + \sqrt T \| f \|_{L^{\infty, p} _T} . 
$$
\end{prop}


\subsection{Proof of Proposition \ref{propell} }\label{sectppp}

It is a direct  application of the Proposition above. In \eqref{ellipeq} we have three terms in the right-hand-side : 
$V_1$,  $ - X_2 P^{-1} V_2$ and $ \kappa^2 P^{-1} V_0$. All of them are of the form ${\bf Q} V_j$ above. 
The unknown $V$  has two components, $(q, u)$, as well as the $V_j$ appearing in the right-hand-side, and the equations for the $q$'s  and the $u$'s decouple. We apply the Proposition above with  $s =1$ for the $u$'s and 
$s= 2$ for the $q$'s, to get the first two estimates of Proposition~\ref{propell}. The $L^\infty$ estimates follow from 
the last part of Proposition~\ref{propell75}. 





\section{Existence for the linearized equations}
\label{Sex} 
This section is devoted to the proof of Proposition~\ref{propexist}.  It turns out that the linearized equations do not enter any known framework, since there is no existing theory for initial boundary value problems for  degenerate dispersive systems,  with the  complication that   the weight $h$ has  different powers for the components $u$ and $q$. Thus, the results gathered here are of independent interest.

The structure of the linearized equations is given at \eqref{lineq}. The goal of this section is to construct solutions to these linear equations. 
\medbreak

\noindent
{\bf NB.} The discussion on the dependence on $\mu$ is irrelevant for the construction of a solution to \eqref{lineq}. {\it For the sake of clarity, we therefore set $\mu =1$ throughout this section.}

\medbreak

It is convenient to work with time independent differential operators in space; to this end, we introduce $p = (1+  a ) q$ as a new unknown so that the equations \eqref{lineq} read
 \begin{equation}
  \label{Lmodel} 
 \begin{cases}
 c_1  \D_t (c_2 p)   + \D_x u = f , \\
  { \bf d}  \D_t u +  {\bf l }\,  p = {\bf g} \quad \mbox{ with } \quad {\bf g}= g_0 +   {\bf l} g_1,
  \end{cases} 
  \end{equation}
  where
 \begin{equation}
 \label{dd} 
 {\bf d} \, u =   1  +   {\bf l} \big[ -  b_0  h_0 \D_x u+   b_1  u \big]-  b_1 h_0 \D_x u+  b_2 u, 
 \end{equation} 
with $c_1$, $c_2$ and $b_0 , b_1, b_2$ given. More precisely, we have
\begin{equation}\label{defb}
b_0 = \frac{4}{3} \underline q^2, \quad b_1 = \underline q B'(\underline \vp), \quad b_2 = B'(\underline \vp),^2, \quad c_1=c, \quad c_2=\frac{1}{1+ a},
\end{equation}
and we shall also make the following assumption.
\begin{ass}
\label{ass81}
The functions $\tilde c:=c_1 c_2$ and $\underline q$ are positive and bounded from $0$ by a positive constant on 
$[0, T] \times \RR_+$. 
\end{ass} 
  
We shall also denote by $\cL$ the linear operator associated to \eqref{Lmodel},
\begin{equation}
\label{defLbis}
\cL  \begin{pmatrix} p \\ u \end{pmatrix} =  \begin{pmatrix}  c_1  \D_t (c_2 p)  +\D_x u \\ 
{\bf d}  \D_t u  + {\bf l } p   \end{pmatrix}. \quad 
\end{equation}
 
 \medbreak
 
We  derive an energy estimate for this system in \S \ref{sectEEsmooth}, under the assumption that all the functions involved are smooth. If we want to generalize this energy estimate at low regularity, it is necessary to give sense to the integration by parts, which requires several duality formulas in weighted spaces that are studied in \S \ref{sectduality}. In \S \ref{sectestLR}, we identify the space $\WW_T$ of minimal regularity to justify the derivation of  the energy estimate. We then use this result in \S \ref{sectweak} to construct weak solutions in the energy space $\VV_T$. The energy space $\VV_T$ however is strictly larger than $\WW_T$ and therefore uniqueness is not granted by the energy estimate. This is why we prove in \S \ref{sectstrong} that weak solutions are actually strong solutions, i.e. limits in $\VV_T$ of solutions in $\WW_T$. They satisfy therefore the energy estimate and hence, are unique. Still assuming that the coefficient are smooth, we then discuss in \S \ref{sectsmooth} the smoothness of these strong solutions. Finally, we relax the smoothness assumptions on the coefficients in \S \ref{sectlast}, which allows us to prove Proposition \ref{propexist} on the well posedness of the linear initial boundary value problem for \eqref{eqlinex}.


\subsection{Energy estimates for smooth functions} \label{sectEEsmooth}

Energy estimates for smooth functions are easily obtained by a slight adaptation of the proof of Proposition \ref{enlin}. Multiplying the first equation by  $ h_0^2 p $, the second by 
$ h_0 u$ and integration. The two basic identities are that 
\begin{equation}
\label{dual1}
  \int_{\RR_+}   h_0^2  p \,   \D_x u    \ dx = -   \int_{\RR_+}   u \, \D_x (  h_0^2 p )   \ dx  = - \int _{\RR_+}   h_0 u \, {\bf l}  p \ dx 
\end{equation}
since the boundary term $ ( h_0^2 g  u) _{ | x = 0} $ vanishes, and  
\begin{equation}
\label{dual3}
 \int_{\RR_+}  h_0 u  {\bf d } u   dx  =  \int_{\RR_+}  h_0 \Big( u^2 +  \big(  b_0 ( h_0 \D_x u)^2 -  2 b_1 u  h_0\D_x u + 
b_2 u^2 \big) \Big) dx . 
\end{equation} 
They imply the following identity
\begin{equation}
\label{idener}
\frac{d}{dt}  E (t) =   \int_{\RR_+}  c'   h_0^2 p^2  dx   + 2  \int_{\RR_+}(  h_0^2 f p +  h_0 ( g_0 u +   g_1  h_0 \D_x u ) dx
\end{equation}
where  $c' = c_2  \D_t c_1- c_1 \D_t c_2$ and 
\begin{equation}
\label{energ1}
E (t) = \int_{\RR_+}    h_0^2 c_1 c_2 p^2 +  h_0  \Big( u^2 +  \big(  b_0 ( h_0 \D_x u)^2 + 2 b_1 u  h_0\D_x u + 
b_1 u^2 \big) \Big)  . 
\end{equation}
Remarking now that
$$
b_0 X^2 -  2 b_1 X Y + b_2 Y^2 = \frac{1}{3}  \underline q^2 X^2 +\big( B'(\underline \vp) Y +  \underline q X\big)^2
\ge \frac{1}{3}  \underline q^2 X^2,
$$
one has 
\begin{equation}
\label{energ2}
E (t) \approx  \int_{\RR_+}    h_0^2  p^2  dx +   \int_{\RR_+}   h_0  \big( u^2 +   ( h_0 \D_x u)^2 \big) dx
\end{equation}
and a Gronwall argument implies that
\begin{equation}
\label{estEEintro} 
\| U(t) \|_{\VV} \lesssim \| U(t) \|_{\VV}  + \int _0^t \|\cL U(t' ) \|_{\VV' } \, dt',
\end{equation} 
where $\VV$ and $\VV'$ are as defined in \eqref{defVV}

\subsection{Duality formulas}\label{sectduality}

In order to perform energy estimates at low regularity, we need to extend the necessary integration by parts formulas by duality. The duality formula giving sense to the identities \eqref{dual1} and \eqref{dual2} are gathered in this section.

To simplify the exposition, it is convenient to recall and introduce some notations. 
We have already used the weighted spaces $L^2_s =   h_0^{- s/2} L^2(\RR_+)$  equipped with the norm
\begin{equation}
\nonumber
\| u \|^2_{L^2_s} =  \int_{\RR_+}  h_0^s | u(x)|^2 dx. 
\end{equation}
The identities \eqref{dual1} and \eqref{dual2} lead to introduce/recall the following spaces 
 $$
 \begin{aligned}
  \cH^1_1 &  = \{  u \in  L^2_1\ :    \  h_0 \D_x u \in L^2_1\} \subset  L^2_1, 
\\ 
 W & = \{  u \in  L^2_1    \ : \   h_0 \D_x u \in L^2 \}   \subset \cH^1_1, 
 \\
  \cH^{-1}_1  & = \{  g _0 + {\bf l}  g_1   \ : \     (g_0, g_1)  \in L^2_1\times L^2_1  \} \subset H^{ -1}_{loc} (\RR_+),  
\end{aligned}
$$
equipped with the obvious norms. \\
We first prove that, as the notations suggests, $\cH_1^{-1}$ is the dual space to $\cH_1^1$.
\begin{lem} \label{lem82}

 $C^\infty_0 (\RR_+ ) $  is dense in $\cH^1_1$ and in $W$.  If one identifies $L^2_1$ as its own dual, the 
 dual space of $\cH^1_1  \subset L^2_1$ can be identified to $\cH^{-1}_1$ through the pairing 
\begin{equation}
\label{84}
\langle u , g_0 + {\bf l } g_1 \rangle_{\cH^1_1  \times \cH^{-1}_1 } 
=   \ ( u, g_0)_{L^2_1} -  (  h_0 \D_x u, g_1)_{L^2_1}     . 
\end{equation}
\end{lem} 

\begin{proof}
{\bf a) } Introduce the  cut off $\chi_\eps = \chi (x/ \eps) $ where $\chi = 0$  for $x \le 1$ and 
$\chi  (x) =1$ for $x \ge 2$.  Because $h_0 \approx x$ near the origin,
  the  $ h_0 \D_x \chi_\eps  $  are uniformly bounded in $L^\infty$, and 
by Lebesgue's Theorem, $\chi_\eps u \to u \in \cH^1_1$ when $u \in \cH^1_1$. 

We  show that $\chi_\eps u \to u \in W$ when $u \in W$. 
For this, it is sufficient to show that the commutator 
$   h_0 (\D_x \chi_\eps) u \to 0 $ in $ L^2$.  
One has
 \begin{equation*}
\|  h_0  (\D_x \chi_\eps) u \|_{L^2}^2 
\lesssim \int_{ \eps \le x \le 2 \eps}  | u|^2   . 
 \end{equation*} 
To prove that this tends to $0$ is is sufficient to show that for
$u \in W$ one has:  
 \begin{equation}
 \label{minihardy}
  x | u(x)|^2 \in L^\infty(0 ,1)\quad \mathrm{and} \quad   \lim_{x \to 0} x | u(x)|^2  = 0. 
 \end{equation}
Note that $u$ is locally $H^1$ and therefore continuous on $( 0, 1]$.  Moreover, with $f = x \D_x u \in L^2 (0, 1)$, one has 
 $$
 | u(x) | \le  | u(x_0 )| +  \int_x^{x_0} \frac{1}{y } | f(y) | dy  \le     | u(x_0 )| + \frac{1}{\sqrt x}  \| f \|_{L^2 ([0, x_0])} . 
$$  
Taking $x_0 = 1$, this shows that $x^\mez u $ is bounded on $(0, 1]$. In addition,  
$$
 \limsup_{x \to 0}  \sqrt x | u(x)|2-  \le \| f \|_{L^2 ([0, x_0])}. 
$$
Because $x_0$ is arbitrary, this implies \eqref{minihardy} and thus $\chi_\eps u \to u \in W$ when $u \in W$. 

Similarly, one can truncate near $\infty$ and functions with compact support 
in $ ( 0, + \infty [$ are dense in $\cH^1_1$ and $W$  One can approximate them by smooth functions and thus $C^\infty_0 (\RR_+ ) $  is dense both in $\cH^1_1$ and $W$. 
\smallbreak
{\bf b) } 
The mapping $ u \mapsto (u, h_0 \D_x u)$ sends $\cH^1_1$ in $L^2_1\times L^2_1 $  and its range is closed. Therefore 
linear forms on $\cH^1_1$ are exactly functionals of the form
\begin{equation*}
u \ \mapsto \ \rho (u) =  ( u, g_0)_{L^2_1} + (  h_0 \D_x u, g_1)_{L^2_1}    , 
\end{equation*}
with $ (g_0, g_1)  \in  L^2_1\times L^2_1$. Interpreted in the sense of distributions, one has for 
$u \in C^\infty_0$: 
$$
\rho (u) = \langle u,  h_0 (g_0 - {\bf l}\, g_1) \rangle_{ C^\infty_0 \times  H^{-1} }  ,
$$
where $H^{-1}$ is the usual Sobolev space of order $-1$ and the duality is taken in the sense of distributions. 
By density of $C^\infty_0$ in $\cH^1_1$,  the linear form $\rho$ vanishes on   $\cH^1_1$ if and only if 
${\bf l}\, g_1 = g_0$ in the sense of distributions. This shows that, as a space of distributions,  the dual space of $\cH^1_1$ is $  h_0 \cH^{-1}_1 \subset H^{-1} (\RR_+)$ and the link with the pairing defined at \eqref{84}  is that 
for  $ u \in C^\infty_0$ and ${\bf g}=g_0+{\bf l}g_1 \in \cH^{-1}_1$
\begin{equation}
\label{dual1c}
\langle u , {\bf g}  \rangle_{\cH^1_1  \times \cH^{-1}_1 }  =  \langle u,  h_0 {\bf g}  \rangle_{ C^\infty_0 \times  H^{-1} } .
\end{equation}
\end{proof}

We can now use Lemma \ref{lem82} to extend \eqref{dual1} at low regularity. By density  to $u \in \cH^1_1$ and $p \in \cH^1_1$ and 
\begin{equation}
\label{dual1bb}
\big(  h_0 \D_x u, p\big)_{L^2_1} = \big( u,{\bf l}\,  p\big)_{L^2_1}. 
\end{equation} 
With the identification above, $ h_0 \D_x$ and ${\bf l}$ map $\cH^{1}_1$ to $L^2_1$
and  $L^2_1$  to $\cH^{-1}_1$, and \eqref{dual1bb} extends to  $u \in \cH^1_1$ and    $p \in L^2_1$ as  
\begin{equation}
\label{dual2}
\begin{aligned}
 & \langle  u ,    {\bf l}  p  \rangle_{\cH^1_1  \times \cH^{-1}_1 }=   - (  h_0 \D_x u,  p)_{L^2_1}  , \\
&  (  {\bf l}  u,  p)_{L^2_1}  =  -   \langle   u ,        h_0 \D_x p  \rangle_{\cH^1_1  \times \cH^{-1}_1 }.  
\end{aligned}
\end{equation}
For the second key identity \eqref{dual3} of the energy estimates, one has similarly that $ {\bf d} $  maps $ \cH^1_1$ to $\cH^{-1}_1$ 
and for $u$ and $v$ in ${ \cH^{1}_1}$ one has that
\begin{equation}
\label{815} 
\langle     {\bf d } u ,      v  \rangle_{\cH^{-1}_1  \times \cH^1_1 } = 
\langle  u ,     {\bf d} v  \rangle_{ \cH^1_1  \times \cH^{-1}_1}  
\end{equation}
and that this is equal to the right-hand-side of \eqref{dual3} when $v=u$.\\
We also need another extension of \eqref{dual1}. 
\begin{lem}
\label{lem83} 
For  $ u \in W $  and $p \in L^2_2$ with ${\bf l} \,p \in \cH^{-1}_1$, one has 
\begin{equation}
\label{dual56}
\int_{\RR_+}   h_0^2 p\,  \D_x u    =  - \langle u, {\bf l}\, p \rangle_{ \cH^1_1  \times \cH^{-1}_1}. 
\end{equation} 
\end{lem} 
\begin{proof} 
If  $p \in L^2_2$ with ${\bf l} \,p \in \cH^{-1}_1$, both terms are defined and continuous  for $u \in W$. Thus it is sufficient 
to prove the equality for $u \in C^\infty_0$ in which case, by \eqref{dual1c},
$$
\langle u, {\bf l}\, p \rangle_{ \cH^1_1  \times \cH^{-1}_1} = 
\langle u,   h_0 {\bf l}\, p \rangle_{ C^\infty_0  \times \cH^{-1}} =  \int p \,  {\bf l}^* ( h_0 u) dx = 
- \int p  h_0^2 \D_x u dx 
$$
because $ {\bf l}^* ( h_0 u) = - \D_x ( h_0^2 u) + 2  h_0 '  h_0 u = -  h_0^2 \D_x u$. Thus the equality \eqref{dual56} is true
when $u \in C^\infty_0$, and thus by density for all $u \in W$. 
\end{proof} 


\subsection{Energy estimates at low regularity}\label{sectestLR}

A key step in the construction of solutions to \eqref{Lmodel} is to use the energy estimate \eqref{estEEintro} when $U=(p,u)$ has a very limited regularity. We therefore discuss here   the question  to know for which $(p, u)$ the  computations of \S \ref{sectEEsmooth}  are valid, using the duality formulas established in the previous section.

To state the energy estimate in short notations, we recall that $\VV = L^2_2 \times \cH^1_1$; this  is the natural energy space for 
$(p, u)$.   We identify its dual with  
$\VV' = L^2_2 \times \cH^{-1}_1$ through the duality 
\begin{equation}
\langle U, \Phi \rangle_{\VV' \times \VV} = (p, \phi)_{L^2_2} + \langle    u ,      \psi ,\rangle_{\cH^{-1}_1  \times \cH^1_1 },
\end{equation}
for $U = (p, u) \in \VV'$ and $\Phi = (\phi, \psi) \in \VV$.  We also introduce the spaces
\begin{equation}
\VV_T = L^2 ([0, T] ; \VV), \qquad \VV'_T = L^2 ([0, T] ; \VV'), 
\end{equation}
with the obvious duality. To justify the integrations by parts, we use a smaller space,
\begin{equation}\label{defWT}
\WW_T = \left\{  \begin{aligned} U = (p, u) &   \in \VV_T, \   \D_t U \in \VV_T, 
\\
&   h_0 \D_x u \in L^2 ([0, T], L^2), \ {\bf l } \, p \in L^2 ([0, T],  \cH_1^{-1})
\end{aligned} \right\} . 
\end{equation} 
One reason to introduce  this space is that the operator $\cL$ defined in \eqref{defLbis}
maps $\WW_T$ to $\VV'_T$. The other reason is that for 
$U \in \WW_T$ the integrations by parts  used to derive the identity \eqref{idener} 
are justified, thanks to Lemma~\ref{lem83} and \eqref{815}. 
Indeed, the energy $E(t) $ defined in \eqref{energ1} is well defined, satisfies
$E(t) \approx \| U(t) \|_\VV^2$ and 
$$
\frac{d}{dt}  E (t) =   \int_{\RR_+}  c'   h_0^2 p^2  dx   + 2  
\big\langle   \cL U ,  U \big\rangle_{\VV' \times \VV} .
$$
This implies the following result. 
\begin{prop}
\label{propEE}
Suppose that the coefficients are Lipschitz and  Assumption~$\ref{ass81}$ is satisfied. Then, 
the  space $\WW_T$ is contained in $C^0 ([0, T]; \VV)$, the operator 
$\cL$ maps $\WW_T $ to $\VV'_T$ and for $U \in \WW_T$
\begin{equation}
\label{estEE} 
\| U(t) \|_{\VV} \lesssim \| U(t) \|_{\VV}  + \int _0^t \|\cL U(t' ) \|_{\VV' } \, dt'. 
\end{equation} 
\end{prop}

\subsection{The dual problem and  weak solutions }\label{sectweak}

Since ${\bf d} $ is symmetric (actually, self adjoint), the dual problem of \eqref{Lmodel}  is  
 \begin{equation}
 \label{dualpb}
 \begin{cases}
 c_2 \D_t (c_1 \phi )  +   \D_x \psi  =  \tilde f , \\
  \D_t {\bf d}  \psi  +  {\bf l }   \phi  =  \tilde g  .
  \end{cases} 
  \end{equation}
  This  system is similar to  \eqref{Lmodel}, but of course the dual problem of the forward Cauchy problem 
  for \eqref{Lmodel} is the backward Cauchy problem for \eqref{dualpb}. 
Parallel to \eqref{defL} introduce 
\begin{equation}
\label{defLp}
\cL'  \begin{pmatrix} \phi \\ \psi  \end{pmatrix} =  \begin{pmatrix}  c_2 \D_t  (c_1 \phi)  +   \D_x \psi  \\ 
\D_t {\bf d} \psi    + {\bf l }   \end{pmatrix}. 
\end{equation}
Then,  for smooth functions, and writing $\Phi=(\phi,\psi)^T$,
\begin{equation}
\label{IPP2} 
\big(  \cL U, \Phi \big)_{ \HH_T   } + \big(   U,  \cL'  \Phi \big) _{\HH_T}  =  \big( U (T), \Gamma \Phi(T) \big)_{\HH} 
- \big( U (0), \Gamma \Phi(0) \big)_{\HH} 
\end{equation}
where $\HH_T = L^2 ([0, T]; \HH) $ with $\HH = L^2_2 \times L^2_1$    and
\begin{equation}
 \Gamma  \begin{pmatrix} c_1c_2\phi \\ \psi  \end{pmatrix} =  \begin{pmatrix}    \phi  \\ 
 {\bf d} \psi        \end{pmatrix}. 
\end{equation} 
 With the notations introduced in the previous section, $\cL'$ also acts from $\WW_T$ to $\VV'_T$, and using again Lemma~\ref{lem83}, the identity \eqref{IPP2} extends to functions $U$ and $\Phi$ in $\WW_T$ as
\begin{equation}
\label{IPP3} 
\begin{aligned}
\big\langle  \cL U ,  \Phi \big\rangle_{ \VV'_T \times \VV_T   }&  + \big\langle   U,  \cL'  \Phi \big\rangle _{\VV_T\times \VV'_T}   
\\
& =  \big\langle  U(T) ,  \Gamma \Phi(T) \big\rangle_{\VV \times \VV'} 
-  \big\langle  U(0) ,  \Gamma \Phi(0) \big\rangle_{\VV \times \VV'} . 
\end{aligned}
\end{equation} 
This motivates the following definition. 
\begin{defi}
Given $F = (f,{\bf g}) \in \VV'_T$ and $U^0 = (p^0, u^0) \in \VV$,  $U = (p, u) \in \VV_T$ is a weak solution of \eqref{Lmodel}
if for all smooth $\Phi = (\phi, \psi)$  which vanishes at $t = T$, one has 
\begin{equation}
\label{weak} 
\big\langle  F,  \Phi \big\rangle_{ \VV'_T \times \VV_T   } + \big\langle   U,  \cL'  \Phi \big\rangle _{\VV_T\times \VV'_T}  +  
 \big\langle  U^0 ,  \Gamma \Phi(0) \big\rangle_{\VV \times \VV'} = 0 . 
\end{equation}
\end{defi}

The Proposition~\ref{propEE} can be applied to $\cL'$ as well, and to the backward Cauchy problem  
since the structure of the equation is preserved when one changes $t$ to $- t$. 
Assuming  that the coefficients are Lipschitz and  Assumption~$\ref{ass81}$ is satisfied, this implies that 
for  smooth $\Phi = (\phi, \psi)$  one has for $t \in [0, T]$: 
\begin{equation}
\| \Phi(t) \|_{ \VV} \lesssim  \| \Phi(T) \|_{ \VV}  \ + \ \int_t^T  \| \cL' \Phi(t') \|_{ \VV'} dt'. 
\end{equation}  
In particular,   for smooth test functions $\Phi$ such that $\Phi (T) = 0$, one has
\begin{equation}
\nonumber 
\| \Phi(0 ) \|_{ \VV} +   \| \Phi  \|_{ \VV_T }  \lesssim    \| \cL' \Phi \|_{ \VV'_T}  . 
\end{equation}
Moreover
$$
\| \Gamma  \Phi(0 ) \|_{ \VV'} \lesssim \| \Phi(0 ) \|_{ \VV} \lesssim \| \cL' \Phi \|_{ \VV'_T} . 
$$
Consider the map  $\Phi \mapsto  \cL' \Phi$ defined on the space of smooth functions such that $\Phi(0) = 0$. 
 The estimates above imply that  it is invertible on its range  $\cR \subset \VV'_T$. Denote by  $\cL'^{-1} $ its inverse defined on  $\cR$. Then,   for $F \in \VV'_T$ and $U_0 \in \VV$ the linear form 
\begin{equation*}
\Psi \mapsto  \big\langle  F,  \cL'^{-1} \Psi   \big\rangle_{ \VV'_T \times \VV_T   } +  
 \big\langle  U^0 ,  \Gamma \cL'^{-1} \Psi ( 0) \big\rangle_{\VV \times \VV'}   
\end{equation*}
is continuous for the norm $\| \Psi \|_{\VV'_T}$. Therefore it can be written 
$-   \big\langle   U, \Psi \big\rangle _{\VV_T\times \VV'_T} $, and  $U$ satisfies \eqref{weak}.  
Therefore, we have proved the following result.

\begin{prop}
If  the coefficients are Lipschitz and  Assumption~$\ref{ass81}$ is satisfied, then for all $F \in \VV'_T$ and  $U^0 \in \VV$, the Cauchy problem \eqref{Lmodel} has a weak solution in $\VV_T$. 
\end{prop} 


\subsection{Strong solutions} \label{sectstrong}

Consider  first the case where the  initial data $U^0$  identically vanishes.  Let  $F    \in \VV'_T$
and let $U$ be a weak solution.  Extend the coefficients for negative times and  extend $U$ by $0$
to obtain a weak solution, still denoted by $U$,  on 
$]- \infty,T ] \times \RR_+ $, which vanishes for $t <0$.    Of course, $U$ satisfies the equations in the sense of distributions. 
We show that $U$ is indeed a strong solution, that is, a limit of solutions in $\WW_T$, and thus satisfies the energy estimate
and hence is unique. 
 \begin{prop}
 \label{w=s}
For  $F  \in \VV'_T$, the Cauchy problem for \eqref{Lmodel} with 
initial data $U^0 = 0$, has a unique weak solution $U$. Moreover, 
$U\in C^0 ([0, T], \VV)$ and satisfies the energy estimates \eqref{estEE} and is a limit in 
$  C^0 ([0, T], \VV)$ of a sequence 
$U_\eps \in \WW_T$ such that  $U_\eps (0) = 0$ and 
$\cL U_\eps \to F$ in $\VV'_T$. 
 \end{prop}

To prove this result, we first introduce mollifiers  and commute the equations with them. 
To prepare for the next section, it is convenient  to use the following smoothing operators,
\begin{equation}\label{molli}
 J_\eps u (t) = \eps^{-1}  \int_{-\infty}^t e^{ (s-t) / \eps  } u(s) ds = 
  \eps^{-1}  \int_0^\infty e^{  - s  / \eps  }   u(t - s) ds
\end{equation}
 The following lemma is elementary
 and the proof is omitted. 
 \begin{lem}
{\bf i.}  The operators $J_\eps$ and $\eps \D_t J_\eps$  are uniformly bounded in $L^2 (]-\infty , T])$  and 
for all $u \in L^2 (]- \infty, T]) $, 
$\| J_\eps u  - u \|_{L^2 ([0, T])} \to 0 $.  \\
 Moreover for all $u \in L^2 (] - \infty, T]) $, $J_\eps u \in H^1(] - \infty,  T])$   and 
$\eps \D_t J_\eps u =  u - J_\eps u$. \\
{\bf ii.} One has the commutation property 
\begin{equation}
\label{comm1b}
 [ J_\eps , c ] = \eps J_\eps (\D_t c)  J_\eps. 
\end{equation} 
 {\bf iii.}  If $u \in H^1([0, T])$,  then 
 \begin{equation}
  \D_t J_\eps u = J_\eps \D_t u. 
 \end{equation}
  \end{lem} 
 We can now give the structure of the mollification of the first term of the linear problem \eqref{Lmodel}.
 \begin{lem}
 \label{lemFried1} 
 The following holds
 $$
 J_\eps\big(   c_1 \D_t ( c_2 p)\big)   =  c_1    \D_t ( c_2 J_\eps p)   +  R_\eps J_\eps p 
 $$
 where $R_\eps$ is  bounded as a mapping $L^2((-\infty,T])\to H^1((-\infty,T])$. 
 \end{lem} 
 
 \begin{proof}
 By direct computations
 $$
 R_\eps    =  
 c_1 \eps \D_t  J_\eps\big(  (\D_t c_2)  v\big) 
 +    \eps J_\eps   \D_t \big( c_1   \D_t ( c_2 v)\big)
+  \eps ^2 J_\eps   \D_t \Big( c _1  \D_t J_\eps    \big( ( \D_t c_2  ) v \big)\Big). 
 $$
Since the  $\eps \D_t J_\eps = \eps J_\eps \D_t $  are uniformly bounded in $L^2 ((-\infty , T])$, the lemma follows. 
\end{proof} 
Finally, the structure of the mollified equations is given in the following lemma.
 \begin{lem}
Let $F     \in \VV'_T$. If $U$ is a weak solution of \eqref{Lmodel} on $( - \infty, T]$, vanishing for $t \le 0$, then 
  $ U_\eps  = J_\eps U \in H^1 ([0, T]; \VV) $  satisfies
  \begin{equation}
  \label{mollieq}
     \cL U_ \eps  =  J_\eps F +   R_\eps U _\eps 
 \qquad  U_\eps (0) = 0
  \end{equation} 
  where $R_\eps$ is bounded from $L^2([0, T], \VV) $  to $L^2([0, T], \VV')$  and from
  $H^1([0, T], \VV) $  to $H^1([0, T], \VV')$. 
  Moreover, $R_\eps U_\eps $ tends to $0$ in $L^2([0, T], \VV')$ as $\eps$ tends to $0$. 
\end{lem}

\begin{proof}
We know that $J_\eps$ commutes with  $\D_x$ and ${\bf l} $. Thus it is sufficient to 
commute $J_\eps$  in the 
term    $\D_t p$ and ${\bf d} \D_t u $.  By
Lemma~\ref{lemFried1} and because $J_\eps$ commutes with the weight $h_0$, 
$ c_1 \D_t  (c_2 J_\eps p)  - J_\eps (h_0 c _1\D_t (c_2))  = F_\eps J_\eps p$ 
with $F_\eps$ uniformly bounded from $L^2([0, T], L^2_2) $  to $L^2([0, T], L^2_2)$
and from $H^1([0, T], L^2_2) $  to $H^1([0, T], L^2_2)$.
Because the convergence is obviously true for smooth functions, the uniform bound also implies that 
\begin{equation*}
\|  F_\eps  J_\eps p  \|_{L^2([0, T]; L^2_2)} \to 0. 
\end{equation*} 
It remains to commute 
$J_\eps $ to ${\bf d}\, \D_t$. According to \eqref{defb}, the  terms to look at are $ [J_ \eps, b h_0 \D_x \D_t] u$  and  $[J_ \eps, b \D_t] u$. 
By Lemma~\ref{lemFried1}  and because $h_0$ commutes to $J$, one has  
$$
[J_ \eps, b h_0 \D_x \D_t] u = G_\eps  h_0 \D_x J_\eps u, \qquad   [J_ \eps, b  \D_t] u = G_\eps   J_\eps u,
$$
where the $G_\eps $  are   uniformly bounded from $L^2([0, T], L^2_1) $  to $L^2([0, T], L^2_1)$
and from $H^1([0, T], L^2_1) $  to $H^1([0, T], L^2_1)$.
Hence 
$$
[ J_\eps, {\bf d} \D_t u ] =   {\bf l}\,  G_{1, \eps } J_\eps u +    G_{0, \eps}  J_\eps u  =:   {\bf G}_\eps J_\eps u ,
$$
where  the $G_{k, \eps}  $  are   uniformly bounded from $L^2([0, T], H^1_1) $  to $L^2([0, T], L^2_1)$
and from $H^1([0, T], \cH^1_1) $  to $H^1([0, T], L^2_1)$, meaning that 
${\bf G}_{\eps}  $  is    uniformly bounded from $L^2([0, T], \cH^1_1) $  to $L^2([0, T], \cH^{-1}_1)$
and from $H^1([0, T], \cH^1_1) $  to $H^1([0, T], \cH^{-1}_1)$. 

Again, by density of smooth functions, this implies that 
 ${\bf G}_\eps J_\eps u $ tends to $0$ in $L^2([0, T], \cH^{-1}_1)$.  This finishes the proof of the lemma. 
\end{proof}

 \begin{cor}
Let $F     \in \VV'_T$. If $U$ is a weak solution of \eqref{Lmodel} on $] - \infty, T]$ vanishing for $t \le 0$ then 
  $ U_\eps  = J_\eps U \in \WW_T $  satisfies
  \begin{equation}
  \left\{ \begin{aligned}
  & \|U_ \eps - U \|_{ \VV_T} \to 0, 
  \\ 
 & \| \cL U_ \eps -  F  \|_{ \VV'_T}  \to 0, 
 \\
 & U_\eps (0) = 0.  
  \end{aligned} \right. 
  \end{equation} 
\end{cor} 
\begin{proof}
It only remains to prove that  for all $\eps > 0$, $U_\eps \in \WW_T$. 
 Since $U_\eps \in H^1 ([0, T]; \VV)$, it is sufficient to prove that 
 \begin{equation}
  \D_x u_\eps \in L^2 ([0, T], L^2_2),  \quad {\bf l }\, p_\eps \in L^2 ([0, T], \VV'), 
 \end{equation}
 which follows directly from the equations since we know that $   p_\eps  \in H^1 ([0, T]; L^2_2)$ and 
 ${\bf d } u_\eps \in H^1 ([0, T]; \VV')$. 
\end{proof}
We have now all the elements to start the proof of Proposition \ref{w=s}.
 \begin{proof}[Proof of Proposition~$\ref{w=s}$]
  Because the $U_\eps$ belong to $\WW_T$, we can apply the energy estimates \eqref{estEE} to  
 $U_\eps - U_{\eps'}$ and conclude that $U_\eps$ is a Cauchy sequence in $C^0 ([0, T]; \VV)$. 
  Therefore its limit $U$  belongs to $C^0 ([0, T], \VV)$, vanish for 
 $t = 0$  and satisfies the energy estimates. 
 In particular, $U$ is unique. The proof of the proposition is complete.
 \end{proof}
  \bigbreak
  
 We now construct strong solutions when $U^0 \ne 0$. 
If  $U^0 \in C^\infty_0 (\RR_+)$. We find a solution 
$U (t, x) = U^0 (x) + \tilde U (t, x) $ of the Cauchy problem for \eqref{Lmodel} by solving 
\begin{equation}
\cL \tilde U = F - \cL U^0 , \qquad \tilde U (0 ) = 0. 
\end{equation} 
Moreover, by  Proposition~\ref{w=s},  
 $U $ is the limit in  $C^0 ([0, T], \VV)$ of a sequence 
$U_\eps \in \WW_T $, which satisfies $ U_{\eps}{}_{ | t = 0} = U^0$ and 
$ 
\cL U_\eps  \to  F $   in  $ \VV'_T    $. 
By Proposition~\ref{propEE} the $U_\eps$ satisfy the energy estimates, and thus 
the limit $U$ also satisfies these energy estimates. 

Thus we have solved the Cauchy problem for a dense set of initial data in $\VV$, with solutions which satisfy 
\eqref{estEE}. The next theorem follows,  by approximating $U^0$ by functions in $ C^\infty_0 (\RR_+)$, 
\begin{theo}
For all  $U^0 \in \VV $, $F   \in   \VV_T' $, the Cauchy problem 
for \eqref{Lmodel} with initial data $U^0$ has a unique solution in $C^0 ([0, T], \VV)$, 
which is the limit  a sequence $U_ \eps \in \WW_T$ such that 

i)  $  U_\eps \to U $ in $C^0 ([0, T], \VV)$,

ii)    $ \cL  U_\eps \to \cL U $ in $L^2([0, T], \VV')$,

iii) $ {U_\eps}_{\vert_{ t = 0}}  \to U^0 $  in  $ \VV$. 

\end{theo} 


\subsection{Smooth solutions}\label{sectsmooth}

Remind that we still assume that the coefficients of $\cL$ are smooth. We show here that this induces smoothness on the strong solution constructed in the previous sections.

\begin{prop}
\label{propsmsol}
Suppose that $F \in L^2 (]- \infty, T]; \VV') $  vanishes for $t \le 0$ and satisfies
$\D_t^k (h_0 \D_x)^j  F \in L^2 (]- \infty, T]; \VV') $ for all $k$ and $j$. Then the equation 
$\cL U = F $ has a unique strong solution    in $ C^0 (]- \infty, T]; \VV)$ which  vanishes for $t \le 0$, and
$\D_t^k (h_0 \D_x)^j  V  \in L^2 (]- \infty, T]; \VV) $ for all $k$ and $j$. 
\end{prop}

The proof of the proposition is decomposed into several lemmas.
\begin{lem}
\label{lem813} 
Suppose that $F \in L^2 (]- \infty, T]; \VV') $  vanishes for $t \le 0$ and satisfies
$\D_t F \in L^2 (]- \infty, T]; \VV') $. Then the equation 
$\cL V = F $ has a unique solution strong   in $ C^0 (]- \infty, T]; \VV)$ which  vanishes for $t \le 0$; moreover, 
$U \in C^1 ((]-\infty, T]; \VV)$, and $\D_t U$ satisfies
\begin{equation}
\label{dereq1}
\cL \D_t U = \D_t F    +   R \D_t U ,
\end{equation}
where $R $  is a bounded operator from
$\VV_T $ to $\VV'_T$. 
\end{lem} 

\begin{proof}
We use the mollifiers \eqref{molli}, set $U_\eps = J_\eps U \in H^1 (]- \infty, T]; \VV)$ and use the equation 
\eqref{mollieq}. In particular, the right-hand-side 
$F_ \eps = J_\eps F + R_\eps U_\eps $ belongs to $H^1(]- \infty, T]; \VV')$  and satisfies 
\begin{equation}
\| \D_t  F_\eps \|_{ \VV'_T} \lesssim  \| \D_t  F  \|_{ \VV'_T} + \| \D_t  U_\eps \|_{ \VV_T} . 
\end{equation}
Because we know that $U_\eps \in H^1 (]- \infty, T]; \VV)$, we can differentiate in time the equation \eqref{mollieq}
and see that 
\begin{equation}
\label{dereq2}
\cL \D_t U_\eps  = \D_t F_\eps     +   R \D_t U_\eps ,
\end{equation}
where $R = [ \D_t ,  \cL] $ is a bounded operator from
$\VV_T $ to $\VV'_T$.  
  Therefore, one can apply Proposition~\ref{w=s} to $\D_tU_\eps $, hence 
  $\D_t U_\eps $ is bounded, and indeed a Cauchy sequence, in $C^0 ([0, T]; \VV)$.
  This implies that $\D_t U \in C^0 ((]-\infty, T]; \VV)$ satisfies \eqref{dereq1}. 
  \end{proof}

\begin{lem}
\label{lemsm814}
If in addition to the assumptions of Lemma \ref{lem813} one has $\D_t^k  F \in L^2 (]- \infty, T]; \VV') $ for $k \le n$, then 
$ \D_t^k U  \in C^0 ((]-\infty, T]; \VV)$ for $k \le n$. 
 
\end{lem} 

\begin{proof}
By induction on $n$, using the equation \eqref{dereq1} and checking that for smooth coefficients  the operator 
$[ \D_t^k , R ] $ maps  $L^2 (]- \infty, T]; \VV) $ into 
$L^2 ((]-\infty, T]; \VV')$. 
\end{proof}

The next lemma finishes the proof of Proposition~\ref{propsmsol} 
\begin{lem}
\label{lemsm815}
If  in addition  to the assumptions of Lemma \ref{lemsm814} one has $\D_t^k (h_0 \D_x)^j  F \in L^2 (]- \infty, T]; \VV') $ for all $k$ and $j$ then 
$ \D_t^k (h_0 \D_x)^j U  \in C^0 ((]-\infty, T]; \VV)$ for all $k$ and $j$. 
 
\end{lem} 

In the proof, we use an estimate  which we now state. 

\begin{lem}
\label{Grosseficelle} 
Suppose that  $\alpha > 0$ and 
$$(1 + \alpha h_0    \D^2_t )  p = f,   \qquad p_{| t = 0} = \D_t p _{| t = 0}  = 0. 
$$ 
Then there is $C$  which depends only on the  $L^\infty$ norm of 
$\D_t \alpha / \alpha$ such that
$$
\| p (t) \|_{L^2(\RR_+) }  \le  C \Big(  \| f(t) \|_{L^2( \RR_+) }   
  + \|  \D_t f \|_{L^2([0, T] \times \RR_+) } \Big). 
$$
\end{lem}
\begin{proof} Let $e = p^2 + h_0 \alpha  (\D_t p)^2 $. Then
$$
\D_t  e = 2 f \D_t p +  h_0 \D_t \alpha (\D_t p)^2    ,
$$
and thus 
$$
e =  2 f p   +  \int_0^t  (- 2 p \D_t f  + h_0 \D_t\alpha (\D_t p)^2 )  dt' .
$$
By Gronwall  and Cauchy Schwarz inequalities we conclude that 
$$
 | p (t, x) |^2 \lesssim  | f  (t, x) |^2   + \int_0^t | \D_t f (t' , x) |^2  dt' +\max\{1,\frac{\D_t \alpha}{\alpha}\}\int_0^t e(t',x)dt',
$$
and the lemma follows. 
\end{proof} 

\begin{proof}[Proof of Lemma~$\ref{lemsm815}$]
Before differentiating the equations in $h_0 \D_x$ we prove the necessary smoothness of the solution.  
Using Lemma~\ref{lemsm814} and  the first equation of \eqref{Lmodel} we gain that 
\begin{equation}
\label{boulon1}
\begin{cases}
h_0 \D_x u  &= -\tilde c h_0 \D_t p -c_1 \D_t c_2 h_0 p+h_0 f
, \\
h_0 \D_x \D_t u &= -\tilde c h_0  \D_t^2 p - (\D_t \tilde c +c_1\D_t c_2) h_0  \D_t p - \D_t (c_1 \D_t c_2) h_0 p+h_0 \D_t f
\end{cases}
\end{equation}
are in $C^0 ([0, T] ; L^2)$, where $\tilde c =   c_1 c_2$. Moreover, we can take one more derivative in time and 
$h_0 \D_x \D_t^2 u $ also belongs to $C^0 ([0, T] ; L^2)$. 

Next, together with the second equation of \eqref{Lmodel}, we draw that 
$$
\vp := h_0 {\bf l}  ( p - b_0 h_0 \D_x \D_t  u + b_1 \D_t u ) 
$$ belongs to  $C^0 ([0, T] ; L^2)$ as well as $\D_t \vp$. 
 Using \eqref{boulon1}, and after several commutations, we end up with the following
property that   
\begin{equation*}
\psi := (1 +  b_0 \tilde c  h_0  \D_t^2 ) (h_0 {\bf l}  p)  
\end{equation*}
and $\D_t \psi$ belong to $ C^0 ([0, T] ; L^2)$. 
By Lemma~\ref{Grosseficelle} (with $\alpha=b_0 \tilde c$), we deduce that $h_0 {\bf l } p $ is in $C^0 ([0, T] ; L^2)$ and therefore 
that $h_0 \D_x p \in C^0 ([0, T] ; L^2_2)$. 

Now we have enough smoothness to differentiate the equation in $h_0\D_x$, and obtain that 
$h_0 \D_x U $ is a weak solution of a system
\begin{equation}
\label{eqpbdm} 
\cL (h_0 \D_x U)  =  h_0 \D_x F  + C_1 \D_t U + C_0 U  ,
\end{equation}
where the commutators are computed as in Section~\ref{sectcommut}.   In particular, the right-hand-side belongs to 
$ \in L^2([0, T]; \VV')$ and hence  $h_0 \D_x U \in  C^0([0, T]; \VV')$. 

The Lemma~\ref{lemsm815} now easily follows by induction on $j$ differentiating  the equation \eqref{eqpbdm}
in powers of $h_0 \D_x$, estimating the commutators as in Section~\ref{sectcommut}, indeed in a much easier way since 
we assume here that the coefficients are infinitely smooth. We do not repeat the details here. 
\end{proof} 


\subsection{ Proof of Proposition \ref{propexist}} \label{sectlast}

It remains to relax the condition on the smoothness of the coefficients. 
The equations \eqref{eqlinex}  to solve read
\begin{equation}\label{defL8}
\cL_{a}[\uV,\D]V= F \in L^2_T\VV'^{n-1} , \qquad V_{| t = 0} = 0, 
\end{equation}
with either $a = 0$ or $a = \D_t (\underline u B' (\underline \vp))$. 
The assumption \eqref{initnul} implies that $F$ can be extended by $0$ for negative time, so that 
$ \D_t^j X_2^k F \in L^2(] - \infty, T]; \VV')$ for $k+ j \le n-1$. 

Our assumption is that the quantities  
 \begin{equation}
   {\mathfrak m}_1(\uV ;T) ,  \quad  \tilde {\mathfrak m}_2(\uV;T),\quad 
 {\mathfrak m} (\uV;T),  \quad 
\tilde {\mathfrak m}(\uV;T), \quad \tilde {\mathfrak s}(\cS;T) , 
 \end{equation}  
are finite.  We can extend $\uV$ for negative time and approximate it by a sequence $\uV^l $ of smooth functions such that the same quantities 
evaluated  at $\uV^l $ are bounded. Similarly, we approximate $F$ by a sequence of smooth functions $F^k$  
and the Proposition~\ref{propsmsol} provides us with a sequence $V^l $ such that $V^l = 0$ for $t < 0$ and 
$\D_t^k (h_0 \D_x)^j  V^l  \in  C^0 (]- \infty, T]; \VV) $ for  $k+j\le n-1$. 

By  Proposition~\ref{propHO}, we see that the sequence $V^l$ is bounded in $C^0_T( \VV^{n-1})$. Hence, passing 
to weak limits, we conclude that there exists a solution 
$V$ of \eqref{defL8} such that $V= 0$ for $t < 0$ and $\D_t^k (h_0 \D_x)^j  V \in  L^\infty (]- \infty, T]; \VV)$. 

Moreover, the proof  of Proposition~\ref{propsmsol} shows that  for $j + k \le n-1$
\begin{equation*}
\cL_{a}[\uV^l,\D] \D_t^j X_2^k V^l =  F^l_{j, k} , 
\end{equation*}
is bounded in $L^2 (]- \infty, T]; \VV')$ and 
 $ \D_t^j X_2^k V^l{}_{| t = 0} = 0 $. Thus passing to weak limits, we see that 
 $\D_t^j X_2^k V$  is a weak solution of an equation of the form
 \begin{equation*}
\cL_{a}[\uV,\D] \D_t^j X_2^k V  =  F_{j, k} \in  L^2 (]- \infty, T]; \VV') , \quad \D_t^j X_2^k V^l{}_{| t = 0} = 0 . 
\end{equation*}
Hence it is a strong solution and  $  \D_t^k (h_0 \D_x)^j  V \in   C^0(]- \infty, T]; \VV)$. 
This finishes the proof of Proposition~\ref{propexist}. 

 \section{The initial conditions}
 \label{Sect9}

\subsection{Invertibility of ${\bf d} $}

In this section we assume that $\mu > 0$, otherwise everything is trivial. The invertibility of ${\bf d} (\uV) $ is implicit in the proof of the energy estimates: we have already noticed that 
\begin{equation}
\label{dual36}
 \int_{\RR_+}  h_0  u  {\bf d } u   dx    \ge      c \big(  \| u\|^2_{L^2_1} +\mu  \| h_0 \D_x u\|^2_{L^2_1} \big) 
\end{equation} 
provided that $\underline q$ and  $ B'(\underline \vp)$ are bounded.  

We now focus on the inverse of ${\bf d} [\uV] $ at time $t = 0$, where $\uq = 1/2$ and 
$\underline \vp = x$. We call it 
${\bf d}^0$ : 
 \begin{equation}
\label{dd0} 
 {\bf d}^0 \, u =   u  +  \mu  \Big( {\bf l} \big[ -  \frac{1}{3}  h_0 \D_x u+   b  u \big]-  b  h_0 \D_x u+  4 b^2  u\Big) ,  
 \end{equation} 
 where $ {\bf l} = h_0 \D_x + 2 h_0' $ and  $b = \mez B'(x)$.  
The equation 
\begin{equation}
\label{eq92} 
{\bf d}^0 u = f 
\end{equation}
is seen as an elliptic "boundary" value problem on $\RR_+$, associated to the variational form  
\begin{equation}
\label{dual37}
  \int_{\RR_+}   h_0   \Big((1 +  4 \mu b^2 )u v  +   \frac{\mu }{3}  ( h_0 \D_x u) (h_0 \D_x v)  -  \mu  b h_0  (   v  \D_x u  
  + u   \D_x v)  
  \Big) dx . 
\end{equation} 
which by \eqref{dual36} is coercive on $\cH^1_1$. Using the density Lemma~\ref{lem82}, this implies
that ${\bf d}^0  $ is an isomorphism from $\cH^1_1 (\RR_+)$ to $\cH^{-1}_1 (\RR+)$. If $ f \in L^2_1$ the equation also gives that 
$ \mu (h_0 \D_x)^2 u \in L^2_1$. 
Commuting with derivatives $(h_0 \D_x)^k$, this implies the following 
\begin{lem}
\label{lem91} 
For all $k \ge 0$, 
${\bf d}^0  $ is an isomorphism from $\cH^{k+2}_1 (\RR_+)$ to $\cH^{k}_1 (\RR+)$  and there is a constant $C$, independent of 
$\mu$ such that 
\begin{equation}
\label{est95}
\| u\|_{\cH^k_1}  +   \mu \| u\|_{\cH^{k+2}_1}  \le   C \| {\bf d}^0 u\|_{\cH^k_1}.  
\end{equation} 
\end{lem} 

Next we consider the action of ${\bf d}^0$ in other  weighted spaces $L^2_s$ and also  in the usual Sobolev spaces. 
This operator enters in the category of degenerate elliptic boundary value problems and  we refer to \cite{BolleyCamus} for a general analysis of such problems. However, for the convenience of the reader,  we include short proofs of the needed results. 
Our goal is to prove the following estimates. 
\begin{prop}
\label{prop92} 
Given an integer $k$, there is $\eps_k > 0$ such that if $\sqrt{\mu}h'(0) < \eps_k$, and $f \in \cHH^k (\RR_+) $, 
the solution $ u \in \cH^2_1(\RR_+)$ of \eqref{eq92}, belongs to $\cHH^k (\RR_+)$ as well as 
$ h_0 \D_x u $ and $(h_0 \D_x)^2 u$.  Moreover,  there is a constant $C$ that depends only on $\eps_k$ such that 
\begin{equation}
\| u\|_{\cHH^k} + \sqrt \mu  \| h_0 \D_x u\|_{\cHH^k} +  \mu \| (h_0\D_x)^2 u\|_{\cHH^k} \le C \| f \|_{\cHH^k}. 
\end{equation} 
\end{prop}

The difficulty is only near the origin, where the equation ${\bf d}^0$ has a regular singularity. 
More precisely, ${\bf d}^0$ is a perturbation of 
  \begin{equation}
\label{dd00} 
 {\bf d}^{00} \, u =   u  +  \mu  \alpha (x \D_x + 2) ( -  \frac{\alpha }{3}  x \D_x  +   \beta)u - \mu \alpha  \beta     x \D_x u+  4 \mu \beta^2 u. 
 \end{equation} 
 where $\alpha = h_0'(0)$ and $\beta = b(0)$, in the sense that, near $x = 0$, 
 \begin{equation}
 \label{eq98} 
  {\bf d}^{0} =   {\bf d}^{00} + \mu \big( c_1  x^3 \D_x^2  +  c_2 x^2 \D_x + c_3  x \big). 
 \end{equation} 
 for some coefficients $c_j$. 
 
 Associated to $ {\bf d}^{00}$ is the indicial equation $e_\mu(r) = 0$, the roots of which are the exponents $r$ such that 
 $x^r$ is a solution of the homogeneous equation  $ {\bf d}^{00}  x^r = 0 $. Here, 
  $$
  e_\mu (r) =  1   -   \mu   \frac{\alpha^2 }{3}  r^2   -   \mu  \frac{2 \alpha^2}{3}     r + \mu \big( 2 \alpha \beta + 4 \beta^2 \big) .
 $$
 Note that 
 $$
 e_\mu (-1) = 1 + \mu \big( \frac{\alpha^2}{3} + 2 \alpha \beta + 4 \beta^2 \big) \ge 1 , 
 $$
 and 
  $$
 e_\mu (- \mez )   =   1   -   \mu   \frac{\alpha^2 }{12}         + \mu \big( \frac{\alpha^2}{3} + 2 \alpha \beta + 4 \beta^2 \big) \ge 1
 $$
  so  the indicial equation has two real roots $r_1 < r_2$ which satisfy 
 \begin{equation}
 r_1 < -1 < - \mez < r_2. 
 \end{equation} 
 The indicial equation determines in which weighted spaces the operator ${\bf d}^{00}$ is invertible. 
 
 \begin{lem}
 \label{lem93} 
 Let $s \in \RR$ be such that $e_\mu (  s - \mez) > 0$.  Then  there is a constant $C$ such that 
 \begin{equation}
 \begin{aligned}
 \|  x^ {- s} u \|_{L^2} + \sqrt \mu \| x^{ 1- s}  \D_x u \|_{L^2} + \mu &  \| x^{- s}   (x \D_x )^2 u \|_{L^2}
 \\
 &  \le C e_\mu ( s - \mez)^{-1} \| x^{-s}   {\bf d}^{00} u \|_{L^2} ,  
 \end{aligned} 
 \end{equation} 
and ${\bf d}^{00}$ is an isomorphism between the spaces associated to these norms.

 \end{lem} 
 \begin{proof} 
 
 The equation
 \begin{equation}
 {\bf d}^{00} u = - \mu \frac{\alpha^2} {3} ( x \D_x - r_1) (x \D_x - r_2 ) u = f 
 \end{equation}
can be solved explicitly, and its inverse in $L^2_1$ is given by 
 \begin{equation}
 \label{eq912} 
 ({\bf d}^{00})^{-1}  f (x) = K f(x) :  =  \int_0^\infty  K(x /  y) f(y) \frac{dy}{y} 
 \end{equation}
 where
 \begin{equation}
 K(x ) = \frac{3}{\mu \alpha^2 (r_2 - r_1) }  \left\{ \begin{aligned}  x^{r_2} , \quad x \le 1, 
 \\  x^{r_1} \quad , x \ge  1 . \end{aligned}  \right.
 \end{equation}
On $\RR_+$ with the measure $dx / x$, on has the convolution estimates 
$$
  \| x^s K  f \|_{L^2(dx/x) }   \le \|  x^s K \|_{L^1 (dx/x) } \| x^s  f \|_{L^2 (dx/x) } .
$$
Applied to $- s + \mez$ this gives
\begin{equation*}
 \| x^{-s} K  f \|_{L^2 }  \le \|  x^{- s- \mez}  K \|_{L^1  } \| x^{- s}   f \|_{L^2  } .
\end{equation*} 
We note that $ x^{- s- \mez}  K \in L^1 (\RR_+) $ if and only if $ s - \mez \in ]r_1, r_2[$, that is 
if and only if 
$ e_\mu (   s -  \mez ) > 0 $ and then
$$
\|  x^{s- \mez}  K \|_{L^1  }  =    \frac{3}{\mu \alpha^2 (r_2 - r_1) } \Big( \frac{1}{ r_2 - s + \mez} -
\frac{1}{ r_1 -  s + \mez} \Big) = \frac{1}{e (  s - \mez)} . 
$$
This implies that 
\begin{equation}
 \| x^{-s} K  f \|_{L^2 }  \le e_\mu(s - \mez)^{-1}  \| x^{- s}   f \|_{L^2  } .
\end{equation}

There is an expression for $ x \D_x K f$  similar to \eqref{eq912}, with a  kernel 
$K'$, which is  $r_2 K $ for $x \le 1$ and $r_1 K$ for $x \ge 1$. 
Because the roots $r _j $ are $O (\mu ^{- 1/2}) $, one obtains the desired estimate for 
$\sqrt \mu x \D_x Kf$. 
The estimate for $  \mu (x \D_x)^2 Kf$ follows then from the equation. 
\end{proof}

These estimates are then transported to ${\bf d}^0$. 
 \begin{lem}
 \label{lem95} 
 For all $\delta > 0$, there is a constant C such that for  $s \ge 0 $ and $\mu \le 1$  satisfying  $e_\mu (  s - \mez) \ge \delta$
 and $f$ such that $ h_0^{-s} f \in L^2(\RR_+)$,  
 the solution $u \in \cH^2_1  $ of \eqref{eq92}  belongs to $h_0^s L^2$ as well as $h_0 \D_x u$ and $(h_0 \D_x)^2 h_0$, and satisfies  
 \begin{equation}
 \begin{aligned}
 \|  h_0^ {- s} u \|_{L^2} + \sqrt \mu \| h_0^{ 1- s}  \D_x u \|_{L^2} + \mu &  \| h_0^{- s}   (h_0 \D_x )^2 u \|_{L^2}
 \\
 &  \le C \| h_0^{-s}  f  \|_{L^2}.   
 \end{aligned} 
 \end{equation} 
\end{lem} 

\begin{proof}
The estimates for $x \ge 1$ are immediate since there  $h_0 \approx 1$. Therefore it is sufficient 
to prove the estimate when $u$ is supported in $[0, 2 ]$. There we use \eqref{eq98} and Lemma~\ref{lem93} 
to improve by induction the integrability property from $  u \in  \cH^2_1$ to 
$u \in \cH^2_{ 1 - 2j}$  for  $2j -1 \le s$ and finally to $\cH^2_{- s}$, noticing that the condition 
$e_\mu (  s - \mez) > 0$ means that $s - \mez < r_2$ (recall that $r_1$ is negative) and thus is satisfied in the intermediate steps 
$2j  - 3/2  <  r_2$.
\end{proof}

\begin{prop}
\label{prop95}
Given an integer $k$ and $\mu \le 1$  satisfying 
\begin{equation}
\label{condindic} 
e_\mu (k - \mez)   > 0  
\end{equation} 
 and $f \in \cHH^k (\RR_+) $, 
the solution $ u \in \cH^2_1(\RR_+)$ of \eqref{eq92}, belongs to $\cHH^k (\RR_+)$ as well as 
$ h_0 \D_x u $ and $(h_0 \D_x)^2 u$.  Moreover,    For all $\delta > 0$, there is a constant C such that if $e_\mu (  s - \mez) \ge \delta$, then 
\begin{equation}
\| u\|_{\cHH^k} + \sqrt \mu  \| h_0 \D_x u\|_{\cHH^k} +  \mu \| (h_0\D_x)^2 u\|_{\cHH^k} \le  C \| f \|_{\cHH^k}. 
\end{equation} 
\end{prop} 
  \begin{rem}
By a standard Sobolev embedding, the proposition provides a control of $({\bf d}^0)^{-1}u$ in $L^\infty$ provided that $e_\mu(\mez)>0$. Actually, it could be shown that $({\bf d}^0)^{-1}$ acts in $L^\infty$ under the weaker condition $e_\mu(0)>0$.
\end{rem}
\begin{proof}
Because $f \in \cH^k_1$, we already know that  $u \in \cH^{k+2}_1$, so that 
$h_0^{j+ 1/2} \D_x^j u \in L^2$ for $j\leq k+2$. 
To gain the weights, we  commute ${\bf d}^0$  to  $ h_0^k \D_x^k $. Note that 
$$
[h_0^k \D_x^k , h_0 \D_x ]  = h_0^k [\D_x^k , h_0 ]\D_x - h_0 [\D_x , h_0^k ] \D_x^k 
= 
 \sum_{l=2}^k  h_0^k \big(^l_k\big) (\D_x^l h_0)  \D_x^{k - l +1} 
$$
because the first term in the first commutator is $k h_0^k h_0' \D_x^{k}$ and cancels out with the second term. 
Similarly
$$
[h_0^k \D_x^k , h_0^2 \D_x^2 ]  = h_0^k [ \D_x^k , h_0^2 ] \D_x^2  - h_0^2 [\D_x^2, h_0^k ] \D_x^k . 
$$
Developing $[ \D_x^k , h_0^2 ]$ and  $[\D_x^2, h_0^k]$  the terms in $\D_x^{k+1}$ occurring in   in $  h_0^k [ \D_x^k , h_0^2 ] \D_x^2$ 
and  $h_0^2 [\D_x^2, h_0^k ] \D_x^k $  are both equal to $ 2 k h_0' h_0^{k+1} \D_x^{k+1}$  and therefore cancel out. 
The terms in $\D_x^{k}$ are 
$$
\mez k (k-1) h_0^k  \D_x(h_0^2)   =   k (k-1) h_0^k ( h_0h_0'' + h_0'^2) 
$$
and 
$$
h_0^2 \D_x^2 (h_0^k) =  k  h_0'' h_0^{k+1}  + k (k-1) h_0'^2 h_0^k 
$$
respectively.  Therefore it remains 
$$
[h_0^k \D_x^k , h_0^2 \D_x^2 ]   = \sum_{l=2}^k  h_0^k \big(^l_k\big) (\D_x^l h_0^2)  \D_x^{k - l +2} - 
 k (k-2)  h_0^{k+1} h_0''  \D_x^k .  
$$
Commuting next $h_0^k \D_x^k$ to the coefficients, we conclude that 
\begin{equation}
\label{supercom}
{ \bf d^0} (h_0^k \D_x^k u) =  h_0^k  \D_x^k { \bf d^0}  u + \mu  \sum_{j = 0}^{k-1} h_0^k c_{k, j} \D_x^j u 
+ \mu c_{k, k}  h_0^{k+1} \D_x ^k . 
\end{equation}
for some coefficients $c_{k, j}$. 
Then, one easily proves the estimates by induction on $k$, using Lemma~\ref{lem95} as long as $k < r_2$ : 
if one knows that $\D_x^j u \in L^2$ for $j < k$ and $ \mu h_0 \D_x \D_x^{k-1} u \in L^2$, then the right-hand-side of 
\eqref{supercom} belongs to $h_0^k L^2$, and by the cited lemma, 
$h_0^k \D_x^k u \in h_0^k L^2$, that is $\D_x^k u \in L^2$, with similar estimates for 
$\sqrt \mu  h_0 \D_x \D_x^k u$ and $\mu (h_0 \D_x)^2 \D_x^k u$. 
 \end{proof}

 Since we need to solve equations of the form ${\bf d}u= {\bf g}$ with ${\bf g}=g_0+\sqrt{\mu} {\bf l} g_1\in \cH^{-1}_1$, we also need similar estimates for the equation  
 \begin{equation}
\label{eq92d} 
{\bf d}^0 u =  \sqrt \mu \,  h_0 \D_x  f 
\end{equation}
which has a unique solution in $\cH^1_1$ when $f \in L^2_1$ since 
 ${\bf  d}^0$ is an isomorphism from $\cH^1_1$ to $\cH^{-1}_1$. 
  \begin{prop}
\label{prop96}
 Under the condition \eqref{condindic} 
 and with  $f \in \cHH^k (\RR_+) $, 
the solution $ u \in \cH^1_1(\RR_+)$ of \eqref{eq92d}, belongs to $\cHH^k (\RR_+)$ as well as 
$ h_0 \D_x u $.  Moreover,    For all $\delta > 0$, there is a constant C such that if $e_\mu (  s - \mez) \ge \delta$, then 
\begin{equation}
\| u\|_{\cHH^k} + \sqrt \mu  \| h_0 \D_x u\|_{\cHH^k}  \le  C \| f \|_{\cHH^k}. 
\end{equation} 
\end{prop}
 \begin{proof}
Let $ v=  \sqrt{\mu } \,  h_0 \D_x ( {\bf d}^0)^{-1}  f  $.  The solution of    \eqref{eq92d} is $u = v + w$   
where $w  $ solves  
$$
{\bf d}^0 w  =  \sqrt \mu \, [{\bf d}^0, h_0 \D_x ]  v  = \mu^{ 3/2} \big( c_* ( h_0\D_x)^2 v + c_*  h_0 \D_x v + c_* v  \big) 
$$
where the $c_*$ denote various smooth coefficients. Then the result follows easily from the previous proposition. 
\end{proof}

 Since  
 \begin{equation*}
 e_\mu (r) =  1 + \mu \Big(2\beta + \mez \alpha\big)^2  - \mu \alpha^2 \Big( \frac{r^2}{3} + \frac{2r}{3} + \frac{1}{4}\Big), 
 \end{equation*}
 the condition \eqref{condindic} is satisfied if $\mu \alpha^2 $ is small enough, and therefore  Proposition~\ref{prop92} follows.


  \subsection{Proof of Proposition~\ref{propvapp}} 
  To compute the initial values  $V^j = (q^j, u^j)$ of the time derivatives
  $\D_t ^j V$  it is convenient to commute first the equations with $\D_t^j$ and next evaluate at $t = 0$. The commutations have already been written in \eqref{eqstep1} and \eqref{eqstep2}  with detailed computations made in Section~\ref{sectproofHO}. The first equation in \eqref{mod317}  yields an  induction formula 
  \begin{equation}
  \label{inducID1} 
  c (q^0) q^{j+1}  = -  \D_x u^j  +  \sum  c_*  (q^0) q^{j_1} \ldots q^{j_l}  
  \end{equation}
  where the $c_* $ are some smooth functions and the indices   in the  sum  satisfy 
  $j_1 + \ldots + j_l = j +1$  and $j_k \le j $. 
  The second equation  yields to 
   \begin{equation}
  \label{inducID2} 
  {\bf d}^0  u^{j+1}  =   - { \bf l} q^j   +  \sum b_*  u^{j_1} \ldots u^{j_l}   + \mu \big(  h_0 \D_x  \cN_0 + \cN_1\big) 
  \end{equation}
where we have used that ${\bf l}$ commutes with $\D_t$, $\vp_{t = 0 }  = x$ is known and $\D_t \vp = u$; 
in the  sum, the  $b_* $ are some smooth derivatives of $B$  and the indices      satisfy 
  $j_1 + \ldots + j_l = j $; and   $\cN_0$ and $\cN_0$ are non linear terms which are sum of terms 
  of the form
  \begin{equation*}
  a_* (q^0, u^0, \D_x u^0)  q^{j_1} \ldots q^{j_l} (h_0 \D_x)^{k_1} u^{j'_1} \ldots  (h_0 \D_x)^{k_{l'}}u^{j'_{l'}} 
  \end{equation*} 
  with smooth coefficients     $j_1 + \ldots + j_l   +    j'_1 + \ldots + j'_{l'} = j $ and $k_1 + \ldots + k_{l'} \le 2$, 
  $\max \{k_1, k_2\}  \le 1$
  which means that there are at most two terms involving at most one $h_0 \D_x$ derivative.

Recall that the initial condition for $q $ is $q^0 = \mez$, so we only 
  have to consider the initial condition $u^0$. Using  Propositions~\ref{prop95} and \ref{prop96} and elementary multiplicative properties of Sobolev spaces, the induction formulas above imply the following lemma which concludes the proof of Proposition~\ref{propvapp}.
    \begin{lem}
  If  $ u^0 \in \cHH^{k+1}$ and $h_0 \D_x u^0 \in \cHH^{k+1}$ and if the indicial condition \eqref{condindic} is satisfied, then 
  for $j \le k +1$, the $V^j = (q^j, u^j)$ are well defined and satisfy 
  \begin{equation}
 \label{inducID3} V^j \in \cHH^{ k+1 - j},  \qquad  \sqrt \mu h_0 \D_x V^j  \in \cHH^{k+1 -j}, 
\end{equation}
  with uniform bounds if the condition \eqref{condindic} is uniformly satisfied. 
  \end{lem}

  \begin{proof}
  If the condition \eqref{inducID3}  is satisfied up to order $j$, it is clear the right-hand-side of \eqref{inducID1}, 
  and thus $q^{j+1}$, satisfies the same condition at order $j+1$. 
  Similarly, the first two terms in the right-hand-side of \eqref{inducID2} satisfy the condition at order $j+1$, and we can 
 apply $({\bf d}^0)^{-1}$ to them, using Proposition~\ref{prop95}. The last two terms are more delicate. Using that there is at most one  derivative $h_0 \D_x$ acting on the $u^{j'} $, one obtains that 
 $\sqrt \mu \cN_0 $ belongs to $\cHH^{k + 1 -j}$. Then we can apply Proposition~\ref{prop96} to conclude. The term 
 with $\cN_1$ is easier, an can be treated with Proposition~\ref{prop95}. 
  \end{proof}
 \bibliographystyle{alpha}
\bibliography{BIB_LM}

\newcommand{\etalchar}[1]{$^{#1}$}
\begin{thebibliography}{BCL{\etalchar{+}}11}

\bibitem[ASL08a]{Alvarez-Samaniego:2008ai}
Borys Alvarez-Samaniego and David Lannes.
\newblock Large time existence for 3d water-waves and asymptotics.
\newblock {\em Invent. math.}, 171:485--541, 2008.

\bibitem[ASL08b]{AL2}
Borys Alvarez-Samaniego and David Lannes.
\newblock A {N}ash-{M}oser theorem for singular evolution equations.
  {A}pplication to the {S}erre and {G}reen-{N}aghdi equations.
\newblock {\em Indiana Univ. Math. J.}, 57(1):97--131, 2008.

\bibitem[BC73]{BolleyCamus}
P.~Bolley and J.~Camus.
\newblock Sur une classe d'op\'erateurs elliptiques d\'eg\'en\'er\'es \`a
  plusieurs variables.
\newblock {\em M\'emoires Soc. Math. France}, 34:55--140, 1973.

\bibitem[BCL{\etalchar{+}}11]{Bonneton:2011jo}
P.~Bonneton, F.~Chazel, D.~Lannes, F.~Marche, and M.~Tissier.
\newblock A splitting approach for the fully nonlinear and weakly dispersive
  green--naghdi model.
\newblock {\em J. Comput. Phys}, 230:1479--1498, 2011.

\bibitem[BM06]{BreschMetivier}
Didier Bresch and Guy M\'etivier.
\newblock Global existence and uniqueness for the lake equations with vanishing
  topography : elliptic estimates for degenerate equations.
\newblock {\em Nonlinearity}, 19, 2006.

\bibitem[CS11]{CS1}
Daniel Coutand and Steve Shkoller.
\newblock Well-posedness in smooth function spaces for moving-boundary 1-{D}
  compressible {E}uler equations in physical vacuum.
\newblock {\em Comm. Pure Appl. Math.}, 64(3):328--366, 2011.

\bibitem[CS12]{CS2}
Daniel Coutand and Steve Shkoller.
\newblock Well-posedness in smooth function spaces for the moving-boundary
  three-dimensional compressible {E}uler equations in physical vacuum.
\newblock {\em Arch. Ration. Mech. Anal.}, 206(2):515--616, 2012.

\bibitem[dP16]{Poyferre}
Thibault de~Poyferr\'e.
\newblock A priori estimates for water waves with emerging bottom.
\newblock {\em arXiv:1612.04103}, 2016.

\bibitem[FI14]{Iguchi_new}
Hiroyasu Fujiwara and T.~Iguchi.
\newblock A shallow water approximation for water waves over a moving bottom.
\newblock {\em Advanced Studies in Pure Mathematics}, 64:77--88, 2014.

\bibitem[FKR16]{FILIPPINI2016381}
A.G. Filippini, M.~Kazolea, and M.~Ricchiuto.
\newblock A flexible genuinely nonlinear approach for nonlinear wave
  propagation, breaking and run-up.
\newblock {\em Journal of Computational Physics}, 310:381 -- 417, 2016.

\bibitem[Igu09]{iguchi2009}
Tatsuo Iguchi.
\newblock A shallow water approximation for water waves.
\newblock {\em J. Math. Kyoto Univ.}, 49(1):13--55, 2009.

\bibitem[Isr11]{Israwi201181}
Samer Israwi.
\newblock Large time existence for 1d green-naghdi equations.
\newblock {\em Nonlinear Analysis: Theory, Methods \& Applications}, 74(1):81
  -- 93, 2011.

\bibitem[JM09]{JM1}
Juhi Jang and Nader Masmoudi.
\newblock Well-posedness for compressible {E}uler equations with physical
  vacuum singularity.
\newblock {\em Comm. Pure Appl. Math.}, 62(10):1327--1385, 2009.

\bibitem[JM15]{JM2}
Juhi Jang and Nader Masmoudi.
\newblock Well-posedness of compressible {E}uler equations in a physical
  vacuum.
\newblock {\em Comm. Pure Appl. Math.}, 68(1):61--111, 2015.

\bibitem[Lan13]{Lannes_book}
D.~Lannes.
\newblock {\em The Water Waves Problem: Mathematical Analysis and Asymptotics},
  volume 188 of {\em Mathematical Surveys and Monographs}.
\newblock AMS, 2013.

\bibitem[Li06]{Li}
Yi~A. Li.
\newblock A shallow-water approximation to the full water wave problem.
\newblock {\em Comm. Pure Appl. Math.}, 59(9):1225--1285, 2006.

\bibitem[Liu96]{Liu1996}
Tai-Ping Liu.
\newblock Compressible flow with damping and vacuum.
\newblock {\em Japan Journal of Industrial and Applied Mathematics}, 13(1):25,
  Feb 1996.

\bibitem[MW17]{MingWang}
Mei Ming and Chao Wang.
\newblock Water waves problem with surface tension in a corner domain i: A
  priori estimates with constrained contact angle.
\newblock {\em arXiv:1709.00180}, 2017.

\bibitem[Sch86]{schochet1986}
Steve Schochet.
\newblock The compressible euler equations in a bounded domain: existence of
  solutions and the incompressible limit.
\newblock {\em Comm. Math. Phys.}, 104(1):49--75, 1986.

\bibitem[Ser15]{Serre}
Denis Serre.
\newblock Expansion of a compressible gas in vacuum.
\newblock {\em Bulletin of the Institute of Mathematics, Academia Sinica (New
  Series)}, 10:695--716, 2015.

\end{thebibliography}

 
\end{document}